\documentclass{article}
\usepackage[dvips]{graphicx}
\usepackage{a4wide}

\usepackage[dvips]{graphics,color}      
\usepackage{verbatim}
\usepackage{amsmath}
\usepackage{xspace}
\usepackage{pifont}
\usepackage{graphicx}
\usepackage{amssymb}
\usepackage{epic, eepic}
\usepackage{dsfont}
\usepackage{amssymb}
\usepackage{makeidx}

\usepackage{amsmath,afterpage}
\usepackage{epsf}

\def\0{\mathbf{0}}

\def \cardim {\textrm{cardim}}
\def \dim {\textrm{dim}}
\def \< {\langle}
\def \> {\rangle}

\def\eps{\varepsilon}

\def\lam{\lambda}
\def\rr{\rightarrow}

\def\al{\kappa_0}
\def \cardim {\textrm{cardim}}
\def \dim {\textrm{dim}}
\def \< {\langle}
\def \> {\rangle}
\newcommand{\supp}{\operatorname{supp}}

\def\beqa{\begin{eqnarray}}
\def\eeqa{\end{eqnarray}}
\def\beqas{\begin{eqnarray*}}
\def\eeqas{\end{eqnarray*}}

\newtheorem{theorem}{Theorem}[section]
\newtheorem{lemma}[theorem]{Lemma}
\newtheorem{proposition}[theorem]{Proposition}

\newtheorem{corollary}[theorem]{Corollary}

\newtheorem{remark}[theorem]{Remark}

\newtheorem{definition}[theorem]{Definition}

\numberwithin{equation}{section}
\newcommand{\old}[1]{{}}

\def\endpf{{\ \hfill\hbox{\vrule width1.0ex height1.0ex}\parfillskip 0pt}}
\newenvironment{proof}{\noindent{\bf Proof:}}{\endpf}

\addtocounter{footnote}{0}

\newcommand{\qed}{\hfill\rule{2mm}{2mm}}

\newcommand{\bd}{\begin{displaymath}}
\newcommand{\ed}{\end{displaymath}}
\newcommand{\be}{\begin{equation}}
\newcommand{\ee}{\end{equation}}
\newcommand{\bq}{\begin{eqnarray*}}
\newcommand{\eq}{\end{eqnarray*}}
\newcommand{\bn}{\begin{eqnarray}}
\newcommand{\en}{\end{eqnarray}}
\newcommand{\dl}{\delta}
\newcommand{\re}{\mathds{R}}

\title{Pathwise Uniqueness of the Stochastic Heat Equations with Spatially Inhomogeneous White Noise}
\author{Eyal Neuman   \\ \\ Faculty of Industrial Engineering \\ and Management  \\ Technion - Institute of Technology \\ Haifa 3200 \\ Israel }
\date{}

\begin{document}

\maketitle

\paragraph{Abstract.}
We study the solutions of the stochastic heat equation driven by spatially inhomogeneous multiplicative white noise based on a fractal measure.
We prove pathwise uniqueness for solutions of this equation when the noise coefficient is H\"{o}lder continuous of index $\gamma>1-\frac{\eta}{2(\eta+1)}$. Here $\eta\in(0,1)$ is a constant that defines the spatial regularity of the noise.

\section{Introduction and Main Results} \label{Sec-results}
We study the solutions of the stochastic heat equation with spatially inhomogeneous white noise. This equation has the form
\be \label{SHE}
\frac{\partial}{\partial t} u(t,x) = \frac{1}{2}\Delta u(t,x) + \sigma(t,x,u(t,x))\dot{W} , \ \ t\geq0, \ \ x\in \re.
\ee
Here $\Delta$ denotes the Laplacian and $\sigma(t,x,u):\re_{+}\times\re^2\rr \re$ is a continuous function with at most a linear growth in the $u$ variable.
We assume that the noise $\dot{W}$ is white noise on $\re_{+}\times\re$ based on some $\sigma$-finite measure $\mu(dx)dt$. Equations like (\ref{SHE}) may arise as scaling limits of critical interacting branching particle systems. For example, in the case where $\sigma(t,x,u)=\sqrt{u}$ and $\mu(dx):=dx$,
such equations describe the evolution in time and space of the density of the classical \emph{super-Brownian motion} (see e.g. \cite{Shiga88}). If $\mu$ is any finite measure and $\sigma(u)=\sqrt{u}$, then (\ref{SHE}) describes evolution of the density of  \emph{catalytic super-Brownian motion} with the catalyst $\mu(dx)$ (see e.g. \cite{Zahle2005}). \\\\
In this work we consider the problem of the pathwise uniqueness for the solution of (\ref{SHE}) where $\sigma(\cdot,\cdot,u)$ is H\"{o}lder continuous in $u$ and $\dot{W}$ is a spatially inhomogeneous Gaussian white noise based on a measure $\mu(dx)dt$. More precisely $W $ is a mean zero Gaussian process defined on a filtered probability space $(\Omega,\mathcal{F},\mathcal{F}_t,P)$, where $\mathcal{F}_t$ satisfies the usual hypothesis and we assume that $W$ has the following properties. We denote by
\bd
W_t(\phi)=\int_{0}^{t}\int_{\re}\phi(s,y)W(dyds), \ \  t\geq 0,
\ed
the stochastic integral of a function $\phi$ with respect to $W$.
We denote by $ \mathcal{C}_c^\infty(\re_{+}\times\re)$ the space of compactly supported infinitely differentiable functions on $\re_{+}\times\re$.
We assume that $W$ has the following covariance structure
\bd
E(W_t(\phi)W_t(\psi))=\int_{0}^{t}\int_{\re}\phi(s,y)\psi(s,y)\mu(dy) ds, \ \ t\geq 0,
\ed
for $\phi,\psi\in \mathcal{C}_c^\infty(\re_{+}\times\re)$.
Assume that the measure $\mu$ satisfies the following conditions
\begin{itemize} \label{assump1-2}
  \item [{\bf (i)}] There exists $\eta\in (0,1)$ such that
\bn \label{eta-poten}
\sup_{x\in \re} \int_{\re} |x-y|^{-\eta+\eps} \mu(dy) <\infty, \ \ \forall \eps>0,
\en
 \item [{\bf (ii)}]
 $$\cardim(\mu)=\eta.$$
\end{itemize}
Note that (ii) means that there exists a Borel set $A\subset \re$ of Hausdorff dimension $\eta$ such that
$\mu(A^c)=0$, and this fails for $\eta'<\eta$ (see Definition \ref{def-cardim}).  \medskip \\
In what follows if a white noise is based on the measure $dx \times dt$ (that is $\mu(dx)$ is the Lebesgue measure), we will call it a homogenous white noise.
The stochastic heat equations deriven by homogeneous white noises were studied among many others, by Caba{\~n}a \cite{Cabana70}, Dawson \cite{Dawson72}, \cite{Dawson75}, Krylov and Rozovskii \cite{Krylov77}, \cite{Krilov79}, \cite{Krilov79b}, Funaki \cite{Funaki83}, \cite{Funaki84} and Walsh \cite{Walsh}.
Pathwise uniqueness of the solutions for the stochastic heat equation, when the white noise coefficient $\sigma$ is Lipschitz continuous was derived in \cite{Walsh}. In \cite{MP09}, the pathwise uniqueness was established for the solutions of the stochastic heat equation, when the white noise coefficient $\sigma(\cdot,\cdot,u)$ is H\"{o}lder continuous in $u$ of index $\gamma >3/4$. The $d$-dimensional stochastic heat equation driven by \textit{colored noise} was also extensively studied. Pathwise uniqueness for the solutions of the stochastic heat equation of type (\ref{SHE}) driven by colored noise, with H\"{o}lder continuous coefficients was studied in \cite{MPS06}. The result in \cite{MPS06} was later improved by Rippl and Sturm in \cite{Rippl-Sturm2012}.
The method of proof in \cite{MP09}, \cite{MPS06} and \cite{Rippl-Sturm2012} is a version of the Yamada-Watanabe argument (see \cite{YW71}) for infinite dimensional stochastic differential equations. However, in the case where (\ref{SHE}) is driven by inhomogeneous white noise, this method does not go through. In this work we needed to construct a new covering argument (see Remark \ref{remark-covering}) and to change the regularity argument which was developed in \cite{MP09} (see Remark \ref{Remark-Molifier}).     \\\\
Before we describe in more detail the known uniqueness results for the case of equations driven by homogeneous and inhomogeneous white noises, we introduce an additional notation and definitions.
\paragraph{Notation.}
For every $E\subset \re$, we denote by $\mathcal{C}(E)$ the space of continuous functions on $E$. In addition, a superscript $k$, (respectively, $\infty$), indicates that functions are $k$ times (respectively, infinite times), continuously differentiable. A subscript $b$, (respectively, $c$), indicates that they are also bounded, (respectively, have compact support).
For $f\in \mathcal{C}(\re)$ set
\be \label{CtemNorm}
\|f\|_\lam=\sup_{x\in \re}|f(x)|e^{-\lam|x|},
\ee
and define $$\mathcal{C}_{tem}:=\{f\in \mathcal{C}(\re), \|f\|_\lam<\infty \ \ \textrm{for every} \ \  \lam>0\}.$$ The topology on this space is induced by the norms $\|\cdot \|_\lam$ for $\lam>0$. \\\\
For $I\subset \re_{+}$ let $\mathcal{C}(I,E)$ be the space of all continuous functions on $I$ taking values in topological space $E$ endowed with the topology of uniform convergence on compact subsets of $I$.
Hence the notation $u \in \mathcal{C}(\re_{+},\mathcal{C}_{tem})$ implies that $u$ is a continuous function on $\re_{+}\times \re$ and
\bd
\sup_{t\in[0,T]}\sup_{x\in\re}|u(t,x)|e^{-\lam|x|}< \infty , \ \ \forall \ \lam>0, \ T>0.
\ed
In many cases it is possible to show that solutions to (\ref{SHE}) are in $\mathcal{C}(\re_{+},\mathcal{C}_{tem})$. \medskip \\
We set
\bn
G_t(x)=\frac{1}{\sqrt{2\pi t}}e^{\frac{-x^2}{2t}}, \ x\in\re , \ t>0.
\en
Let us define a stochastically strong solution to (\ref{SHE}), which is also called a strong solution to (\ref{SHE}).
\begin{definition} (Definition next to (1.5) in \cite{MP09}) \label{Def-Strong-sol}
Let $(\Omega,\mathcal{F},\mathcal{F}_t,P)$ be a probability space and let $W$ be a white noise process defined on $(\Omega,\mathcal{F},\mathcal{F}_t,P)$. Let $\mathcal{F}_t^W\subset \mathcal{F}_t$ be the filtration generated by $W$. A stochastic process $u:\Omega \times \re_{+}\times \re \rr \re $ which is jointly measurable and $\mathcal{F}^W_t$-adapted, is said to be a stochastically strong solution to (\ref{SHE}) with initial condition $u_0$ on $(\Omega,\mathcal{F},\mathcal{F}_t,P)$, if for all $t \geq 0$ and $x\in \re$,
\bn \label{MSHE}
u(t,x)&=&G_{t}u_0(x) + \int_{0}^{t}\int_{\re}G_{t-s}(x-y)\sigma(s,y,u(s,y))W(ds,dy), \ \  P-\rm{a.s.}
\en
Here $G_{t}f(x)=\int_{\re}G_{t}(x-y)f(y)dy$, for all $f$ such that the integral exists.
\end{definition}
In this work we study uniqueness for (\ref{SHE}) in the sense of pathwise uniqueness. The definition of pathwise uniqueness is given below.
 \begin{definition} (Definition before Theorem 1.2 in \cite{MP09})
We say that pathwise uniqueness holds for solutions of (\ref{SHE}) in $\mathcal{C}(\re_{+},\mathcal{C}_{tem})$ if for every deterministic initial condition, $u_0\in \mathcal{C}_{tem}$, any two solutions to (\ref{SHE}) with sample paths a.s. in $\mathcal{C}(\re_{+},\mathcal{C}_{tem})$ are equal with probability $1$.
\end{definition}
\paragraph{Convention.} Constants whose values are unimportant and may change from line to line are denoted by $C_i,M_i, \ \ i=1,2,..$, while constants whose values will be referred to later and appear initially in say, Lemma $i.j$ (respectively, Equation $(i.j)$) are denoted by $C_{i.j}$ (respectively, $C_{(i.j)}$). \\\\
Next we present in more detail some results on pathwise uniqueness for the solutions of (\ref{SHE}) driven by homogeneous white noise which are relevant to our context.
If $\sigma$ is Lipschitz continuous, the existence and uniqueness of a strong solution to (\ref{SHE}) in $\mathcal{C}(\re_+,\mathcal{C}_{tem})$ was proved in \cite{Shi94}. The proof uses the standard tools that were developed in \cite{Walsh} for solutions to SPDEs.
In \cite{MP09}, Lipschitz assumptions on $\sigma$ were relaxed and the following conditions were introduced:
for every $T>0$, there exists a constant $C_{(\ref{grow})}(T)>0$ such that for all  $(t,x,u)\in [0,T]\times \re^2$,
\be \label{grow}
|\sigma(t,x,u)|\leq C_{(\ref{grow})}(T)(1+|u|),
\ee
for some $\gamma>3/4$ there are $\bar R_1,\bar R_2>0$ and for all $T>0$ there is an $\bar R_0(T)$ so that for all $t\in[0,T]$ and all $(x,u,u')\in \re^3$,
\be \label{HolSigmaCon}
|\sigma(t,x,u) - \sigma(t,x,u')| \leq \bar R_0(T)e^{\bar R_1|x|}(1+|u|+|u'|)^{\bar R_2}|u-u'|^\gamma.
\ee
Mytnik and Perkins in \cite{MP09} proved that if $u_0\in \mathcal{C}_{tem}$, $\mu(dx)=dx$, and $\sigma:\re_{+}\times\re^2\rr\re$ satisfy (\ref{grow}), (\ref{HolSigmaCon}) then there exists a unique strong solution of (\ref{SHE}) in $\mathcal{C}(\re_+,\mathcal{C}_{tem})$. It was also shown in \cite{MP09} that addition of a Lipschitz continuous drift term to the right hand side of (\ref{SHE}) does not affect the uniqueness result. \medskip \\
Before we introduce our results let us define some spaces of measures that will be used in the definition of spatially inhomogeneous white noise.
\paragraph{Notation:}
Let $\eta \in (0,1)$. For a measure $\mu$ on $(\re,\mathcal{B}(\re))$, let us define
\bn \label{phi-eta}
\phi_{\eta,\mu}(x) &:=& \int_{\re^d}|x-y|^{-\eta}\mu(dy),
\en
Denote by $M_f(\re)$ the space of finite measures on $(\re,\mathcal{B}(\re))$. Let $\dim_H(A)$ be the Hausdorff dimension of any set $A\subset\mathcal{B}(\re)$. \medskip \\
We need the definition of carrying dimension before we state the result.
\begin{definition} \label{def-cardim} [Definition in Section 9.3.1 of \cite{Dawson93}]
A measure $\mu\in M_f (\re)$ is  said to have carrying dimension $\cardim(\mu)=l$, if there exists a Borel set $A$
such that $\mu(A^c)=0$ and $\dim_H(A)=l$, and this fails for any $l'<l$.
\end{definition}
We introduce the following subset of $M_f (\re)$:
\be \label{Meta}
 M^{\eta}_f (\re) := \bigg\{ \mu \in M_f (\re) \bigg|\sup_{x\in \re} \phi_{\eta-\eps,\mu}(x)<\infty, \ \forall \eps>0, \textrm{ and }\cardim(\mu)=\eta \bigg\}.
\ee
Next we define the inhomogeneous white noise that we are going to work with.
\begin{definition} \label{frac-W-noise}
A white noise $W$ based on the measure $\mu(dx)\times dt$, where $\mu \in  M^\eta_{f}(\re)$, is called a spatially inhomogeneous white noise based on $\mu$.
The corresponding white noise process $W_t(A):=W((0,t]\times A)$, where $t\geq 0$ and  $A \in \mathcal{B}(\re)$, is called a spatially inhomogeneous white noise process based on $\mu$.
\end{definition}
Now we are ready to state the main result of the paper: the pathwise uniqueness to the stochastic heat equation (\ref{SHE}) with spatially inhomogeneous white noise holds for some class of H\"{o}lder continuous noise coefficients. The existence of a weak solution to this equation under similar assumptions on $\mu$ and less restrictive assumptions on $\sigma$ was proved by Zahle in \cite{Zahle2005}.
\begin{theorem}\label{holder-conj}
Let $\dot{W}$ be spatially inhomogeneous white noise based on a measure $\mu \in M^\eta_f(\re)$, for some \\ $\eta \in (0,1)$. Let $u(0,\cdot)\in \mathcal{C}_{tem}(\re)$. Assume that $\sigma: \re_{+}\times \re^2 \rr \re$ satisfies (\ref{grow}), (\ref{HolSigmaCon}), for some $\gamma$ satisfying
\bn \label{optsol}
\gamma > 1-\frac{\eta}{2(\eta+1)},
\en
then pathwise uniqueness holds for the solutions to (\ref{SHE}) with sample paths a.s. in $\mathcal{C}(\re_{+},\mathcal{C}_{tem}(\re))$.
\end{theorem}
\begin{remark}
Assume that the condition $\cardim(\mu)=\eta$ is omitted from the assumptions of Theorem \ref{holder-conj}. Then from the proof of Theorem \ref{holder-conj} one can obtain that, in this case, pathwise uniqueness holds for the solutions of (\ref{SHE}) if $\gamma> 1-\eta/4$.
\end{remark}
Note that assumptions (i) and (ii) in (\ref{assump1-2}) are related as follows. Let
\bn
I_{\eta}(\mu)&:=& \int_{\re^d}\phi_{\eta,\mu}(x)\mu(dx).  \nonumber
\en
$\phi_{\eta,\mu}(\cdot)$ (in (\ref{phi-eta})) and $I_{\eta,\mu}(\cdot)$ are often called the $\eta$-potential and $\eta$-energy
of the measure $\mu$ respectively. \\\\
In Section 4.3 of \cite{falconer}, the connection between the sets of measures above and the Hausdorff dimension of sets that contain their support is introduced.
Theorem 4.13 in \cite{falconer} states that if a mass distribution $\mu$ on a set $F\subset \re$ has finite $\eta$-energy, that is,
\bq
I_{\eta}(\mu)<\infty,
\eq
then the Hausdorff dimension of $F$ is at least $\eta$. Recall that a measure $\mu$ is called a mass distribution on a set $F\subset \re^d$, if the support of $\mu$ is contained in $F$ and $0<\mu(F)<\infty$ (see definition in Section 1.3 of \cite{falconer}). \medskip \\
Let us discuss the connection between Theorem \ref{holder-conj} and Theorem 1.2 in \cite{MP09}. The case of $\eta=1$ formally corresponds to the "homogeneous" white noise case that was studied in \cite{MP09}. We see that in Theorem \ref{holder-conj}, our lower bound on $\gamma$ coincides with the bound $3/4$ obtained in \cite{MP09}. Note that as it was shown in \cite{Myt-Mull-Perk2011}, the $3/4$ bound is optimal in the homogeneous case. A counter example for $\gamma < 3/4$ was constructed in \cite{Myt-Mull-Perk2011}.
The optimality of the bound $1-\frac{\eta}{2(\eta+1)}$ in Theorem \ref{holder-conj} is a very interesting open problem. \medskip \\
In our proof we  use the Yamada-Watanabe argument for the stochastic heat equation that was carried out in \cite{MP09} for equations driven by homogeneous white noise.
We describe very briefly the main idea of the argument. Let $\tilde{u}\equiv u^1-u^2$ be the difference between two solutions to (\ref{SHE}). The proof of uniqueness relies on the regularity of $\tilde{u}$ at the points $x_0$ where
$\tilde{u}(t,x_0)$ is "small". To be more precise, we need to show that there exists a certain $\xi$, such that for points $x_0$ where $\tilde{u}(t,x_0)\approx 0$ and for points $x$ near by, we have
\be \label{strr}
|\tilde{u}(t,x)-C_1(\omega)(x-x_0)|\leq C_2(\omega)|x-x_0|^\xi,
\ee
for some (random) constants $C_1, \ C_2$.
Moreover, we will show that in our case, for any $\xi$ such that
\bd
\xi<\frac{\eta}{2(1-\gamma)}\wedge (1+\eta),
\ed
(\ref{strr}) holds for $x_0$ such that $\tilde{u}(t,x_0)\approx 0$.
This, will allow us to derive the following condition for the pathwise uniqueness
\bd
\gamma> 1-\frac{\eta}{2(1+\eta)}.
\ed
One of the by-products of the proof of Theorem \ref{holder-conj} is the following theorem. We prove under milder assumptions that the difference of two solutions of (\ref{SHE}) is H\"{o}lder continuous in the spatial variable with any exponent $\xi<1$ at the points of the zero set.
\begin{theorem}\label{thm-reg}
Assume the hypotheses of Theorem \ref{holder-conj}, however instead of (\ref{optsol}) suppose
\bn \label{gamma-reg}
\gamma > 1-\frac{\eta}{2}.
\en
Let $u^1$ and $u^2$ be two solutions of (\ref{SHE}) with sample paths in $\mathcal{C}(\re_{+},\mathcal{C}_{tem})$ a.s. and with the same initial condition $u^1(0)=u^2(0)=u_0\in C_{tem}$.
Let $u\equiv u^1-u^2$,
\be \label{TK}
T_K=\inf \big\{s\geq 0 : \sup_{y\in \re}(|u^1(s,y)|\vee |u^2(s,y)|)e^{-|y|}>K\big\}\wedge K,
\ee
for some constant $K>0$ and
\bq
S_0(\omega)=\{(t,x)\in [0,T_K]\times \re: u(t,x)=0\}.
\eq
Then at every $(t_0,x_0)\in S_0$, $u$ is H\"{o}lder continuous with exponent $\xi$ for any $\xi<1$.
\end{theorem}
Z{\"a}hle in \cite{Zahle2004} considered (\ref{SHE}) when $\dot{W}$ is an inhomogeneous
white noise based on $\mu(dx)\times dt$, where $\mu$ satisfy conditions which are slightly more general then (\ref{eta-poten})(i) and (ii).
The existence and uniqueness of a strong $\mathcal{C}(\re_+,\mathcal{C}_{tem})$ solution to (\ref{SHE}) when $\sigma$ is Lipschitz continuous and satisfies (\ref{grow}) was proved in \cite{Zahle2004}.
The H\"{o}lder continuity of the solutions to (\ref{SHE}) with inhomogeneous white noise based on $\mu(dx)\times dt$ as above
 was also derived under some more relaxed assumptions on $\sigma$. In fact, it was proved in \cite{Zahle2004} that the solution to (\ref{SHE}), when $\mu$ satisfies (\ref{eta-poten}), is H\"{o}lder continuous with any exponent $\xi<\eta/2$ in space and $\xi<\eta/4$ in time. \medskip \\
The rest of this paper is devoted to the proof of Theorem \ref{holder-conj}. In the end of Section \ref{Section-Proof2.3} we prove Theorem \ref{thm-reg}.
\section{Proof of Theorem \ref{holder-conj} } \label{PoofConj}
Let us introduce the following useful proposition.
\begin{proposition} \label{parameter}
Let $u_0\in C_{tem}$. Let $\sigma$ be a continuous function satisfying (\ref{grow}). Then any solution $u \in C(\re_{+},C_{tem})$ to (\ref{SHE}) satisfies the following property. For any $T,\lam>0$ and $p\in (0,\infty)$,
\bd
E\bigg(\sup_{0\leq t \leq T}\sup_{x\in \re}|u(t,x)|^pe^{-\lam|x|}\bigg) < \infty.
\ed
\end{proposition}
The proof of Proposition \ref{parameter} follows the same lines as the proof of Theorem 1.8(a) in \cite{MPS06}. It uses the factorization method developed by Da Prato, Kwapien and Zabczyk in \cite{DaPratoKwapZab87}. In fact, in our case, the calculations become simpler because of the orthogonality of the white noise. Since the proof of Proposition \ref{parameter} is straightforward and technical, it is omitted.
\paragraph{Proof of Theorem \ref{holder-conj}}
The proof follows the similar lines as the proof of Theorem 1.2 in \cite{MP09}.
In what follows let $u^1$ and $u^2$ be two solutions of (\ref{SHE}) on $(\Omega,\mathcal{F},\mathcal{F}_t,P)$ with sample paths in $\mathcal{C}(\re_{+},\mathcal{C}_{tem})$ $P$-a.s., with the same initial condition $u^1(0)=u^2(0)=u_0\in C_{tem}$ and the same white noise.
By Proposition 4.4 in \cite{Zahle2004}, (\ref{MSHE}) is equivalent to the distributional form of (\ref{SHE}). That is, for $i=1,2$ and for every $\phi\in C_c^{\infty}(\re)$:
\bn \label{mildForm}
\langle u^i(t),\phi \rangle &=& \langle u_0,\phi \rangle + \frac{1}{2}\int_{0}^{t}\langle u^i(s),\frac{1}{2}\Delta \phi \rangle ds \nonumber \\
&&+ \int_{0}^{t}\int_{\re} \sigma(s,x,u^i(s,x))\phi(x)W(ds , dx), \ \ \forall t \geq 0, \ P-\rm{a.s.}
\en
Let $R_0=\bar R_0(K)$ and $R_1=\bar R_1 +\bar R_2$. By the same truncation argument as in Section 2 of \cite{MP09}, it is enough to prove Theorem \ref{holder-conj} with the following condition instead of (\ref{HolSigmaCon}): \\\\
There are $R_0,R_1>1$ so that for all $t>0$ and all $(x,u,u')\in \re^3$,
\be \label{HolSigmaCon2}
|\sigma(t,x,u) - \sigma(t,x,u')| \leq R_0e^{R_1|x|}|u-u'|^\gamma.
\ee
We use the following definitions and notations from \cite{MPS06}.
Let
\be \label{an0}
a_n=e^{-n(n+1)/2}
\ee
so that
\be \label{an}
a_{n+1} = a_ne^{-n-1} = a_n a_n^{n/2}.
\ee
Define functions $\psi\in \mathcal{C}_c^\infty(\re)$ such that $\supp(\psi_n) \subset (a_n/2,a_{n-1}/2)$,
\be \label{psi1}
0\leq \psi_{n}(x) \leq \frac{2}{nx}, \ \ \forall x\in \re,
\ee
and
\be
\int_{a_n/2}^{a_{n-1}/2}\psi_n(x)dx=1.
\ee
Finally, set
\be
\phi_{n}(x) = \int_{0}^{|x|}\int_{0}^{y}\psi_n(z)dzdy.
\ee
Note that $\phi_n(x) \uparrow |x|$ uniformly in $x$ and $\phi_n(x)\in \mathcal{C}^{\infty}(\re)$. We also have
\be
\phi'_n(x)=\textrm{sign}(x)\int_{0}^{|x|}\psi_n(y)dy,
\ee
\be
\phi''_n(x)=\psi_n(|x|).
\ee
Thus,
\bn \label{phi-tag-bound}
|\phi_n'(x)|\leq 1,  \ \forall x\in \re, \ n\in \mathds{N}.
\en
and for any function $h$ which is continuous at zero
\bd
\lim_{n\rr\infty} \int_{\re}\phi''_n (x)h(x)dx = h(0).
\ed
Define
\bd
u\equiv u^1-u^2.
\ed
Let $\< \cdot, \cdot \> $ denote the scalar product on $L^2(\re)$. Let $m\in \mathds{N}$ and recall that $G_t(x)$ denotes the  heat kernel.
Apply It\^{o}'s formula to the semimartingales $ \< u^{i}_t,G_{m^{-2}}(x-\cdot) \> =G_{m^{-2}}u^{i}_t(x), \ i=1,2$ in (\ref{mildForm}) to get
\bn \label{ito1}
&&\phi_n(G_{m^{-2}}u^{i}_t(x))  \nonumber \\
&&=  \int_{0}^{t}\int_{\re}\phi'_n(G_{m^{-2}}u^{i}_s(x))(\sigma(s,y,u^1(s,y))-\sigma(s,y,u^2(s,y))) G_{m^{-2}}(x-y)  W(ds , dy) \nonumber \\
&&+ \int_{0}^{t}\phi'_n(G_{m^{-2}}u^{i}_s(x))\langle u_s,\frac{1}{2}\Delta G_{m^{-2}}(x-\cdot)  \rangle ds \nonumber \\
&&+ \frac{1}{2}\int_{0}^{t}\int_{\re}\psi_n(|G_{m^{-2}}u^{i}_s(x)|)(\sigma(s,y,u^1(s,y))-\sigma(s,y,u^2(s,y)))^2 G_{m^{-2}}(x-y) ^2 \mu(dy)ds.
\en
\begin{remark}
We would like to emphasize that we use here a very specific mollifier $G_{m^{-2}}$. It is different from $\Phi^m$ that was used in \cite{MP09}, since $\Phi^m$ in \cite{MP09} has compact support and $G_{m^{-2}}$ is the heat kernel that has an unbounded support. This choice of mollifier will help us later on.
\end{remark}
Fix $t_0\in(0,\infty)$ and let us integrate (\ref{ito1}) with respect to the $x$-variable, against another test function $\Psi\in \mathcal{C}_c([0,t_0]\times \re)$. Choose $K_1\in \mathds{N}$ large enough so that for $\lam =1$,
\be \label{K1}
\|u_0 \|_\lam< K_1 \ \ \textrm{and } \Gamma \equiv \{x:\Psi_s(x)>0 \textrm{ for some } s\leq t_0 \} \in (-K_1,K_1).
\ee
Now apply the stochastic Fubini's Theorem (see Theorem 2.6 in \cite{Walsh}) and Proposition 2.5.7 in \cite{Perk2002} to \\ $\< \phi_n(G_{m^{-2}}u_t(\cdot)),\Psi_t \> $ as in \cite{MP09} to get,
\bn \label{ito2}
&& \< \phi_n(G_{m^{-2}}u_t(\cdot)),\Psi_t \>  \nonumber \\ \nonumber
&& =  \int_{0}^{t}\int_{\re} \< \phi'_n(G_{m^{-2}}u_s(\cdot))G_{m^{-2}}(\cdot-y),\Psi_s \> (\sigma(s,y,u^1(s,y))-\sigma(s,y,u^2(s,y)))  W(ds , dy) \nonumber \\
&&\quad + \int_{0}^{t} \< \phi'_n(G_{m^{-2}}u_s(\cdot))\langle u_s,\frac{1}{2}\Delta G_{m^{-2}}(\cdot) \rangle,\Psi_s \> ds \nonumber \\
&& \quad + \frac{1}{2}\int_{0}^{t}\int_{\re}\int_{\re}\psi_n(|G_{m^{-2}}u_s(x)|)(\sigma(s,y,u^1(s,y))-\sigma(s,y,u^2(s,y)))^2 G_{m^{-2}}(x-y)^2 \mu(dy)\Psi_s(x)dx ds \nonumber \\
&& \quad +\int_{0}^{t} \< \phi_n( G_{m^{-2}}u_s(\cdot)) , \dot{\Psi}_s \> ds \nonumber \\
&&=:I_1^{m,n}(t)+ I_2^{m,n}(t)+ I_3^{m,n}(t)+ I_4^{m,n}(t).
\en
Note that for $i=2,4$, $I_i^{m,n}(t)$ look exactly like the terms considered in $\cite{MP09}$. The only difference is that
we chose here the heat kernel as a mollifier instead of the compact support mollifier that was chosen in \cite{MP09}. For $I^{m,n}_1,I^{m,n}_3$, the expressions are different since here we use the inhomogeneous white noise based on the measure $\mu(dy)ds$. \\\\
\paragraph{Notation.} We fix the following positive constants satisfying
\bn \label{eps1}
0< \eps_1 < \frac{1}{100}\big(\gamma-1+\frac{\eta}{2(\eta+1)}\big), \ \ 0 < \eps_0 <\frac{\eta^2\eps_1}{100}.
\en
Let $\al\in [1/2,1)$ be some fixed constant (to be chosen later). Set $m_n= a_{n-1}^{-\al-\eps_0}=\exp\{(\al+2\eps_0)(n-1)n/2\}$, for $n\in \mathds{N}$. Note that for $I_3^{m_{n+1},n+1}$ we may assume $|x|\leq K_1$ by (\ref{K1}). Recall that $T_K$ was defined in (\ref{TK}). If $s\leq T_k$, then we have
\be
|u^i(s,y)|\leq Ke^{|y|}, \ \ i=1,2.
\ee
If $K'=Ke^{K_1+1}$, by (\ref{HolSigmaCon2}), (\ref{psi1}), $\supp(\psi_n) \subset (a_n/2,a_{n-1}/2)$, $\sup_{x\in \re}|G_{a_n^{2(\al+\eps_0)}}(x)|\leq m_{n+1}/\sqrt{2\pi}$, (\ref{an0}) and (\ref{an}), we get for all $t\in [0,t_0]$,
\bn
&&I_3^{m_{n+1},n+1}(t\wedge T_K) \nonumber \\
&& \leq C(R_0) \frac{ m_{n+1}}{(n+1)}\int_{0}^{t\wedge T_K}\int_{\re}\int_{\re}|\langle u_s,G_{a_n^{2(\al+\eps_0)}}(x-\cdot) \rangle|^{-1}\mathds{1}_{\{a_{n+1}/2 \leq |\langle u_s,G_{a_n^{2(\al+\eps_0)}}(x-\cdot) \rangle|\leq a_n/2\}}e^{2\gamma R_1|y|}|u(s,y)|^{2\gamma} \nonumber \\ \nonumber \\
&& \quad \times G_{a_n^{2(\al+\eps_0)}}(x-y) \mu(dy)\Psi_s(x)dx ds \nonumber \\
&& \leq C(R_0)  a_n^{-\al-\eps_0}2a_{n+1}^{-1}\int_{0}^{t\wedge T_K}\int_{\re}\int_{\re}\mathds{1}_{\{a_{n+1}/2 \leq |\langle u_s,G_{a_n^{2(\al+\eps_0)}}(x-\cdot) \rangle|\leq a_n/2\}}e^{2R_1|y|}|u(s,y)|^{2\gamma} \nonumber \\ \nonumber \\
&&  \quad \times G_{a_n^{2(\al+\eps_0)}}(x-y) \mu(dy)\Psi_s(x)dx ds \nonumber \\
&& \leq  C(R_0)  a_n^{-1-\al-\eps_0-2/n}\int_{0}^{t\wedge T_K}\int_{\re}\int_{\re}\mathds{1}_{\{a_{n+1}/2 \leq |\langle u_s,G_{a_n^{2(\al+\eps_0)}}(x-\cdot) \rangle|\leq a_n/2\}}e^{2R_1|y|}|u(s,y)|^{2\gamma} \nonumber \\ \nonumber \\
&& \quad \times G_{a_n^{2(\al+\eps_0)}}(x-y) \mu(dy)\Psi_s(x)dx ds. \nonumber \\
\en
Define
\bn \label{I-n}
&&I^n(t)\nonumber \\
&&:= a_n^{-1-\al-\eps_0-2/n}\int_{0}^{t\wedge T_K}\int_{\re}\int_{\re}\mathds{1}_{ |\langle u_s,G_{a_n^{2(\al+\eps_0)}}(x-\cdot) \rangle|\leq a_n/2\}}e^{2R_1|y|}|u(s,y)|^{2\gamma}G_{a_n^{2(\al+\eps_0)}}(x-y) \mu(dy)\Psi_s(x)dx ds. \nonumber \\
\en
The following proposition is crucial for the proof of Theorem \ref{holder-conj}.
\paragraph{Notation:} Denote by
$ \mathds{N}^{\geq K_1} = \{K_1,K_1+1,...\}$.
\begin{proposition} \label{PropStop}
Suppose $\{U_{M,n,K}: \ M,n,K \in \mathds{N}, K\geq K_1 \}$ are $\mathcal{F}_t$ stopping times such that for each $K\in \mathds{N}^{\geq K_1}$,
\bn \label{H1}
U_{M,n,K}\leq T_K, \ U_{M,n,K} \uparrow T_K \textrm{ as } M\rr \infty \textrm{ for each } n \textrm{ and } \nonumber \\
\lim_{M\rr \infty } \sup_{n}P(U_{M,n,K}<T_K)=0,
\en
and
\bn \label{H2}
\textrm{For all } M\in \mathds{N} \textrm{ and } t_0>0, \ \lim_{n\rr\infty}E(I^n(t_0\wedge U_{M,n,K}))=0.
\en
Then, the conclusion of Theorem \ref{holder-conj} holds.
\end{proposition}
\begin{proof}
The proof of Proposition \ref{PropStop} follows the same lines as the proof of Proposition 2.1 in \cite{MP09}. Fix an arbitrary $K\in \mathds{N}^{\geq K_1}$, $t_0>0$ and $t\in[0,t_0]$.
Denote by
\be
Z_n(t)=\int_{\re}\phi_n( \< u_t, G_{a_n^{2(\al+\eps_0)}}(x-\cdot) \> ) \Psi_t(x)dx.
\ee
From Lemma 6.2(ii) in \cite{Shi94} we have
\bn \label{shilem}
\sup_{t\in [0,T]}\sup_{y\in \re}e^{-\lam|y|}\int_{\re}G_{t}(y-z)e^{\lam|z|}dy<C(\lam), \  \forall t\in(0,T], \ \lam\in \re.
\en
By similar lines as in the proof of Proposition 2.1 in \cite{MP09} and with (\ref{shilem}) we have
\be \label{Zn}
0\leq Z_n(t\wedge T_K)  \leq 2Ke^{K_1+1}C_{\ref{Zn}}(\Psi).
\ee
Denote by $g_{m_{n+1},n}(s,y)= \< \phi'_n(G_{a_n^{2(\al+\eps_0)}}u_s(\cdot) )G_{a_n^{2(\al+\eps_0)}}(\cdot-y),\Psi_s \> $. From (\ref{phi-tag-bound}) we have
\be \label{gn}
|g_{m_{n+1},n}(s,y)| \leq \< G_{a_n^{2(\al+\eps_0)}}(\cdot-y),\Psi_s \> .
\ee
We will use the following lemma to bound $E(\< I_1^{m_n,n} \> _{t\wedge T_K})$.
\begin{lemma} \label{new-int-bound}
For every $K>0$ we have
\bn  \label{new-int-bound2}
\sup_{\eps\in(0,1]} \int_{\re}\int_{\re} e^{|y|}G_{\eps}(z-y)\mathds{1}_{\{|z|\leq K\}}\mu(dy)dz <\infty.
\en
\end{lemma}
The proof of Lemma \ref{new-int-bound} is given at the end of Section \ref{Sec.Heat-Kernel-Bounds}. \medskip \\
By (\ref{grow}), (\ref{gn}), Jensen's inequality, Proposition \ref{parameter}, and Lemma \ref{new-int-bound} we have
\bn \label{I1}
E(\< I_1^{m_{n+1},n} \> _{t\wedge T_K}) &=&  E\bigg(\int_{0}^{t\wedge T_K}\int_{\re} (g_{m_{n+1},n}(s,y))^2(\sigma(s,y,u^1(s,y))-\sigma(s,y,u^2(s,y)))^2\mu(dy) ds \bigg) \nonumber \\
&\leq&C(T)\int_{0}^{t\wedge T_K}\int_{\re} ( \< G_{a_n^{2(\al+\eps_0)}}(\cdot-y),\Psi_s \> )^2e^{|y|}E(e^{-|y|}(|u^1(s,y)|^2+|u^2(s,y)|^2))\mu(dy) ds  \nonumber \\
&\leq&C(T)\int_{0}^{t}\int_{\re}\int_{\re} e^{|y|}G_{a_n^{2(\al+\eps_0)}}(z-y)\Psi^2_s(z)\mu(dy)dz ds  \nonumber \\
&\leq&C(T)\|\Psi\|^2_{\infty}\int_{0}^{t}\int_{\re}\int_{\re} e^{|y|}G_{a_n^{2(\al+\eps_0)}}(z-y)\mathds{1}_{\{|z|\leq K+1\}}\mu(dy)dz ds  \nonumber \\
&\leq& C(T,\|\Psi\|_{\infty},K).
\en
Hence
\be
\{I_1^{m_{n+1},n}(s): \ s\leq t_0 \} \textrm { is an } \mathds{L}^2- \textrm{bounded sequence of } \mathds{L}^2- \textrm{martingales}.
\ee
We handle $I_2^{m_{n+1},n}$ similarly as in \cite{MP09}. The argument in  \cite{MP09} is based on the proof of Lemma 2.2(b) in \cite{MPS06}.
It is shown in  \cite{MPS06} that $I_2^{m_{n+1},n}$ has the following decomposition
\bn
I_2^{m_{n+1},n}&=&I_2^{m_{n+1},n,1}+I_2^{m_{n+1},n,2},
\en
where
\bn
&&I_2^{m_{n+1},n,1}(t) \nonumber \\
&&:=-\frac{1}{2}\int_{0}^{t}\int_{\re}  \psi_n(|\langle u_s,G_{a_n^{2(\al+\eps_0)}}(\cdot-x)\rangle|)\bigg(\frac{\partial}{\partial x}\langle u_s,G_{a_n^{2(\al+\eps_0)}}(\cdot-x) \rangle\bigg)^2\Psi_s(x) dx ds \nonumber \\
&& \quad +\frac{1}{2}\int_{0}^{t}\int_{\re}  \psi_n(\langle u_s,G_{a_n^{2(\al+\eps_0)}}(\cdot-x) \rangle)\frac{\partial}{\partial x}(\langle u_s, G_{a_n^{2(\al+\eps_0)}}(\cdot-x) \rangle)\langle u_s, G_{a_n^{2(\al+\eps_0)}}(\cdot-x) \rangle \frac{\partial}{\partial x}\Psi_s(x) dx ds, \nonumber
\en
and
\bn
I_2^{m_{n+1},n,2}(t)&:=&\int_{0}^{t}\int_{\re} \phi'_n(\langle u_s,G_{a_n^{2(\al+\eps_0)}}(\cdot-x) \rangle)\langle u_s, G_{a_n^{2(\al+\eps_0)}}(\cdot-x) \rangle \frac{1}{2}\Delta\Psi_s(x) dx ds. \nonumber
\en
Repeat the same steps as in the proof of Lemma 2.2(b) in \cite{MPS06} to get that for any stopping time $T$
\bn
I_2^{m_{n+1},n,2}(t\wedge T) \rr \int_{0}^{ t\wedge T} \int_{\re}|u(s,x)|\frac{1}{2}\Delta\Psi_s(x)dxds, \ \ \textrm{in } \mathds{L}^1 \textrm{ as } n\rr \infty,
\en
and
\bn
I_2^{m_{n+1},n,1}(s) \leq \frac{a_n}{n}C(\Phi), \ \ \forall  s\leq t_0,  \ n\geq 1.
\en
By the same steps as in Lemma 2.2(c) in \cite{MPS06}, we get that for any stopping time $T$
\bn
I_4^{m_{n+1},n}(t\wedge T) \rr \int_{0}^{ t\wedge T} \int_{\re}|u(s,x)|\dot{\Psi}_s(x)dxds, \ \ \textrm{in } \mathds{L}^1 \textrm{ as } n\rr \infty.
\en
The rest of the proof is identical to the proof of Proposition 2.1 in \cite{MP09}, hence it is omitted.
\end{proof}
\\\\
The rest of this work is devoted to the proof of Proposition \ref{PropStop}. The proof of Proposition \ref{PropStop} is long and involved. In the next section we provide short heuristics for this proof.
\section{Heuristics for the Proof of Proposition \ref{PropStop}} \label{Sec-Heuristics}
We have shown in Section \ref{PoofConj} that the proof of Theorem \ref{holder-conj} depends only on the verification of the results of Proposition \ref{PropStop}. In this section we give a heuristic explanation for this proof.
We use the argument from Section 2 of \cite{MP09} and, in fact, we adapt it for the case of inhomogeneous white noise. \\ \\
For the sake of simplicity neglect $\eps_0$, the exponent $e^{2R_1|y|}$ in $I_n$, in the calculations of this section. We also replace $a_n/2$ with $a_n$ in the indicator function in $I_n$. For the same reasons we also replace the mollifier function in $I_n$ from $G_{a_n^{2(\al+\eps_0)}}(x-\cdot)$ to the following compact support mollifier, which was used in Section 2 of \cite{MP09}.
Let $\Phi\in\mathcal{C}_c^\infty(\re)$ satisfy $0\leq \Phi \leq 1$ $\supp(\Phi)\in (-1,1)$ and $\int_{\re}\Phi(x)dx=1$, and set $\Phi_x^m(y)=m\Phi(m(x-y))$.
Our goal is to show that
\bn \label{heu1}
&& I^n(t)\nonumber \\
&&\approx a_n^{-1-\al}\int_{0}^{t}\int_{\re}\int_{\re}\mathds{1}_{ |\langle u_s,\Phi_x^m \rangle|\leq a_n\}}|u(s,y)|^{2\gamma}\Phi^{m_{n+1}}_x(y) \Psi_s(x)\mu(dy)dx ds \nonumber \\
&&\rr 0 \textrm{ as } n\rr \infty.
\en
The following discussion is purely formal. To simplify the exposition we assume for that $u'$ (the spatial derivative of $u$) exists.
The key to derive (\ref{heu1}) is to control $u$ near its zero set. We will show in Section \ref{Sec-Local-bound} a rigorous analog for the following statement
\bn \label{heu2}
\gamma>1-\frac{\eta}{2(\eta+1)}  \Rightarrow &&  u'(s,\cdot) \textrm{ is } \xi-\textrm{ H\"{o}lder continuous on } \{x:u(s,x)\approx u'(s,x)\approx 0\},   \nonumber \\
 && \forall \xi<\eta.
\en
For the rest of this section we assume that (\ref{heu2}) holds. \medskip \\
Let us expand $I^n$ in (\ref{heu1}) to get,
\bn \label{heu3}
&&  I^n(t)\nonumber \\
&&\approx a_n^{-1-\al}\sum_{\beta}\int_{0}^{t}\int_{\re}\int_{\re}\mathds{1}_{ \{ |u(s,x)|\leq a_n, u'(s,x)\approx \pm a_n^{\beta}\}}|u(s,y)|^{2\gamma}\Phi^{m_{n+1}}_x(y) \Psi_s(x)\mu(dy)dx ds \nonumber \\
&&\equiv \sum_{\beta}I_{\beta}^n(t),
\en
where $\sum_{\beta}$ is a summation over a finite grid $\beta_i\in[0,\bar\beta]$ where $\bar\beta$ will be specified later. The notation $u'(s,x)\approx \pm a_n^{\beta_i}$ refers to partition of the space-time to sets where $|u'(s,x)|\in [a_n^{\beta_{i+1}},a_n^{\beta_i}]$. In what follows, assume that $u'(s,x)\approx a_n^{\beta_i}$, (the negative values are handled in the same way).
Let us expand $u(s,y)$ in (\ref{heu3}) to a Taylor series. The fact that
\bn\label{heu3.1}
\supp{(\Phi_{x}^{m_{n+1}})}\in[x-a_n^{\al},x+a_n^{\al}]
\en
and (\ref{heu2}) leads to,
\bn \label{heu4}
|u(s,y)|&\leq& |u(s,x)|+(|u'(s,x)|+M|x-y|^{\xi})|x-y| \nonumber \\
&\leq& a_n+a_n^{\beta+\al}+Ma_n^{(\xi+1)\al} \nonumber \\
&\leq& a_n^{(\beta\wedge \al\xi)+\al}.
\en
Here (\ref{heu2}) and a comparison of the first and third summands in the second line of (\ref{heu4}) leads to the following choice
\bn \label{alpha_0}
\al=\frac{1}{\eta+1}.
\en
From now on we assume (\ref{alpha_0}).
From (\ref{heu3}) and (\ref{heu4}) we have
\bn \label{heu5}
 I_{\beta}^n(t)&\leq &
 Ca_n^{-1-\al+2\gamma[(\beta\wedge \al\xi)+\al]}\int_{0}^{t}\int_{\re}\int_{\re}\mathds{1}_{ \{ |u(s,x)|\leq a_n, u'(s,x)\approx \pm a_n^{\beta}\}}\Phi^{m_{n+1}}_x(y) \Psi_s(x)\mu(dy)dx ds. \nonumber \\
 \en
Recall that $\cardim{(\mu)}=\eta$. Hence by Definition \ref{def-cardim}, there exists a set $A\subset \re$, such that $\mu(A^c)=0$ and $\dim_H(A)=\eta$. From this, (\ref{heu3.1}), the fact that $\supp(\Psi_s)\subset [-K_1,K_1]$, and the definitions of the Hausdorff measure and the Hausdorff dimension in Sections 2.1 and 2.2 in \cite{falconer}, we can construct a cover $\mathds{V}^{\beta}$  of $A \cap [-K_1,K_1]$  with at most $N(\beta)\leq K_1a_n^{-\beta-\eps}$ balls of diameter $3a_n^{\beta/\eta}$. Here $\eps>0$ is arbitrarily small. We conclude that the integration with respect to the $x$-axis could be done over the set $\mathds{V}^{\beta}$, instead of the whole real line.
\paragraph {Notation:} For a set $B\subset \re$, let $|B|$ denote the Lebesgue measure of $B$. \medskip \\
From the discussion above, we have
\bn \label{lebV}
|\mathds{V}^{\beta}|\leq   3K_1a_n^{-\beta+\beta/\eta-\eps}.
\en
This cover $\mathds{V}^{\beta}$ will help us to improve the argument from \cite{MP09} in the case of inhomogeneous noise, by taking into account that the noise "lives" on a "smaller" set. \medskip \\
Note that by the definitions of $\Phi_{x}^{m_{n+1}}$ and the fact that $\mu \in M^{\eta}_f (\re)$ we have
\bn \label{heu7}
\int_{\re}\Phi_{x}^{m_{n+1}}(y)\mu(dy)&\leq& a_n^{-\al} \int_{\re}\mathds{1}_{\{|x-y|\leq a_n^{\al}\}}|x-y|^{\eta-\eps}|x-y|^{-\eta+\eps}\mu(dy)\nonumber \\
&\leq& a_n^{-\al(1-\eta+\eps)}.
\en
From (\ref{heu5}) and (\ref{heu7}) we have
\bn \label{heu8}
 I_{\beta}^n(t)&\leq &
 Ca_n^{-1-\al+2\gamma[(\beta\wedge \al\xi)+\al]}\int_{0}^{t}\int_{\mathds{V}^{\beta}}\int_{\re}\mathds{1}_{ \{ |u(s,x)|\leq a_n, u'(s,x)\approx \pm a_n^{\beta}\}}\Phi^{m_{n+1}}_x(y) \Psi_s(x)\mu(dy)dx ds \nonumber \\
 &\leq& a_n^{-1-\al+2\gamma[(\beta\wedge \al\xi)+\al]-\al(1-\eta+\eps)}\int_{0}^{t}\int_{\mathds{V}^{\beta}}\mathds{1}_{ \{ |u(s,x)|\leq a_n, u'(s,x)\approx \pm a_n^{\beta}\}} \Psi_s(x)dx ds.
\en
Our goal is to show that the right hand side of (\ref{heu8}) goes to zero as $n\rr \infty$, for any $\beta\geq 0$. The value $\beta=0$ is a little bit special, so we will concentrate here on $\beta>0$. Fix some $\hat \beta>0$ whose value will be verified in the end of the section.
For $\beta \geq \bar\beta$, we get from (\ref{lebV}) and (\ref{heu8}),
\bn \label{heu9}
 I_{\bar \beta}^n(t)
 &\leq& C t a_n^{-1-\al+2\gamma[(\beta\wedge \al\xi)+\al]-\al(1-\eta+\eps)}|\mathds{V}^{\beta}| \nonumber \\
 &\leq& C t a_n^{-1-\al+2\gamma[(\bar\beta\wedge \al\xi)+\al]-\al(1-\eta+\eps)}|\mathds{V}^{\bar\beta}| \nonumber \\
&\leq& C K_1t a_n^{-1-\al+2\gamma[(\bar\beta\wedge \al\xi)+\al]-\al(1-\eta+\eps)-\bar\beta+\bar\beta/\eta-\eps}.
\en
Consider the case where $0<\beta<\bar\beta$. Let
\bn
S_n(s)=\{x\in [-K_1,K_1]: |u(s,x)|<a_n, u'(s,x)\geq a_n^{\beta} \}.
\en
From (\ref{heu2}) we have that if $n$ is large enough, then for every $x\in S_n(s)$, $u'(s,y)\geq a_n^{\beta}/2$ if $|y-x|\leq M^{-1} a_n^{\beta/\xi}$ ($M$ is from (\ref{heu4})). By the Fundamental Theorem of Calculous we get,
\bn
u(s,y)>a_n, \ \forall 4a_n^{1-\beta} <|y-x|\leq M^{-1} a_n^{\beta/\xi}.
\en
The covering argument in \cite{MP09} suggests that $|S_n(s)|$ can be covered by $K_1a_n^{-\beta/\xi}$ disjoint balls of diameter $3a_n^{1-\beta}$. Since we are interested to bound $ I_\beta^n(t)$, then from (\ref{heu8}) it is sufficient to cover $S_n(s)\cap \mathds{V}^{\beta}$. We will assume for now that $a_n^{\beta/\eta}\geq a_n^{1-\beta}$ for $\beta\in[0,\bar \beta]$. This inequality will be verified when we will fix all our constants at the end of this section. By the last inequality and the construction of $\mathds{V}^{\beta}$ and $S_n(s)$, we get that $S_n(s)\cap \mathds{V}^{\beta}$ can be covered with $K_1a_n^{-\beta-\eps}$ balls of diameter $3a_n^{1-\beta}$. Then from the discussion above we have for $0<\beta<\bar\beta$,
\bn \label{heu10}
|S_n(s)\cap \mathds{V}^{\beta}|\leq C(K_1)a_n^{-\beta-\eps}a_n^{1-\beta}.
\en
From (\ref{heu8}) and (\ref{heu10}) we get,
\bn \label{heu11}
 I_{\beta}^n(t)
 &\leq& C(M,K_1,t)a_n^{-1-\al+2\gamma[(\beta\wedge \al\xi)+\al]-\al(1-\eta+\eps)-\beta-\eps+1-\beta}.
\en
Recall that $\eps>0$ is arbitrarily small.
Hence from (\ref{heu9}) and (\ref{heu11}) we get that $\lim_{n\rr \infty}I^{n}_\beta =0$ if
\bn \label{heu12}
\gamma[(\beta\wedge \al\eta)+\al]>\frac{(2-\eta)\al+2\beta}{2}, \ \forall \beta\leq \bar \beta,
\en
and
\bn \label{heu13}
\gamma[(\bar\beta\wedge \al\eta)+\al]>\frac{(2-\eta)\al+\bar\beta+1-\bar\beta/\eta}{2}.
\en
Now recall that $\gamma>1-\frac{\eta}{2(\eta+1)}$. In this case the choice $\bar\beta=1-\bar\beta/\eta$ together with $\bar\beta=\al\eta$ are optimal for (\ref{heu12}) and (\ref{heu13}).
Since $\al=\frac{1}{1+\eta}$ was fixed, we get that $\bar\beta=\frac{\eta}{\eta+1}$ and then by substituting $\bar\beta =\frac{\eta}{\eta+1}$, $\al=\frac{1}{\eta+1}$ in (\ref{heu12}) and (\ref{heu13}) we get  $\lim_{n\rr \infty}  I_{\beta}=0$ if (\ref{optsol}) holds.
\begin{remark} \label{remark-covering}
Note that (\ref{heu8}) corresponds to (2.58) in Section 2 of \cite{MP09}. In the case of SPDE driven by homogeneous white noise, the integral in
(2.58) is bounded by the constant times Lebesgue measure of the covering of $S_n$. In (\ref{heu8}) we improved the bound on the integral by using a covering of a smaller set $S_n(s)\cap \mathds{V}^{\beta}$. The improved upper bound allowed us to get (\ref{optsol}) as a condition for uniqueness in the case of SPDE driven by white noise based on a fractal measure. A direct implementation of the cover from Section 2 of \cite{MP09} to our case would give us the condition $\gamma>1-\eta/4$ which is more restrictive than (\ref{optsol}).
\end{remark}

\section{Verification of the Hypotheses of Proposition \ref{PropStop}}
This section is devoted to the proof of hypotheses of Proposition \ref{PropStop}. The proof follows the same lines as the proof of Proposition 2.1 in Section 3 of \cite{MP09}.
Let $u^1$ and $u^2$ be as in Section \ref{PoofConj}. We assume also the hypothesis of Theorem \ref{holder-conj} and (\ref{HolSigmaCon2}).  \\\\
Denote by
\be \label{D-notat}
D(s,y) = \sigma(s,y,u^1(s,y))-\sigma(s,y,u^2(s,y)).
\ee
From (\ref{MSHE}) and (\ref{D-notat}) we have
\be \label{ud}
u(t,x)=\int_{0}^{t}\int_{\re}G_{t-s}(y-x)D(s,y)W(ds,dy), \ \ P-\textrm{a.s. for all } (t,x).
\ee
By (\ref{HolSigmaCon2}) we have
\be \label{Dbound}
|D(s,y)| \leq R_0(T)e^{R_1|y|}|u(s,y)|^\gamma.
\ee
Let $\dl \in (0,1]$. Recall that $G_t(\cdot)$ was defined as the heat kernel. Denote by
\be \label{u1}
u_{1,\dl}(t,x) = G_{\dl}(u_{(t-\dl)_{+}})(x)
\ee
and
\be\label{u2}
u_{2,\dl}(t,x) =u(t,x)- u_{1,\dl}(t,x).
\ee
Note that by the same argument as in Section 3 of \cite{MP09}, both $u_{1,\dl}$ and $ u_{2,\dl}$ have sample paths in $\mathcal{C}(\re_+,\mathcal{C}_{tem})$. From  (\ref{ud}), (\ref{u1}) we have
\be
u_{1,\dl}(t,x)=\int_{\re}\bigg(\int_{0}^{(t-\dl)_{+}}\int_{\re}G_{(t-\dl)_{+}-s}(y-z)D(s,y)W(ds,dy)\bigg)G_\dl(z-x)dz.
\ee
By the stochastic Fubini's theorem we get \label{check conditions}
\be \label{u11}
u_{1,\dl}(t,x)=\int_{0}^{(t-\dl)_{+}}\int_{\re}G_{t-s}(y-x)D(s,y)W(ds,dy), \ \ P-\textrm{a.s. for all } (t,x)\in \re_+\times \re.
\ee
From (\ref{u2}) and (\ref{u11}) we have
\be \label{u2exp}
u_{2,\dl}(t,x)=\int_{(t-\dl)_{+}}^{t}\int_{\re}G_{t-s}(y-x)D(s,y)W(ds,dy), \ \ P-\textrm{a.s. for all } (t,x)\in \re_+\times \re.
\ee
Let $\epsilon \in (0,\dl)$. We denote by
\be \label{tilde-u2exp}
\tilde u_{2,\dl}(t,\epsilon,x)=\int_{(t+\epsilon-\dl)_{+}}^{t}\int_{\re}G_{t+\epsilon-s}(y-x)D(s,y)W(ds,dy), \ \ P-\textrm{a.s. for all } (t,x)\in \re_+\times \re.
\ee
Again that by (\ref{ud}), (\ref{u11}) and Fubini's Theorem we have
\be \label{decomp-mol}
 \< u_t,G_{\epsilon}(x-\cdot) \> =  u_{1,\dl}(t+\epsilon,x)+ \tilde u_{2,\dl}(t,\epsilon,x).
\ee
\begin{remark} \label{Remark-Molifier}
The proof of Proposition \ref{PropStop} relies on the regularity of $u$ at space-time points where $ G_{\epsilon}u_t(x)\leq a_n$ (see Section \ref{Sec-Heuristics} for heuristics). In \cite{MP09} a sufficient regularity for $u$ was obtained by the following decomposition
\bn \label{decom}
u(t,x)=u_{1,\dl}(t,\cdot)+u_{2,\dl}(t,\cdot),
\en
and by deriving the regularity of $u_{1,\dl_n}(t,\cdot)$ and $u_{2,\dl_n}(t,\cdot)$ (see Lemmas 6.5 and 6.7 in \cite{MP09}). In our case we could not show that $u_{2,\dl_n}(t,\cdot)$ is "regular enough" to get the criterion (\ref{optsol}) for uniqueness. Therefore, we had to change the argument from \cite{MP09}. The crucial observation is that the integral at the right hand side of (\ref{I-n}) is taken over points where $ \< u_t,G_{a_n^{2(\al+\eps_0)}}(x-\cdot) \> \leq a_n$. Therefore by the decomposition in (\ref{decom}), to bound (\ref{I-n}) it is enough to use the regularity of $u_{1,\delta_n}(s+\eps_n,x),\ \tilde u_{2,\delta_n}(s,\eps_n,x)$ (for particular $\dl_n,\eps_n$). It turns out that $\tilde u_{2,\dl}(t,\epsilon,x)$ is "regular enough" so that the proof of Proposition \ref{PropStop} goes through, and as a result we get the condition (\ref{optsol}) for the pathwise uniqueness.
\end{remark}
We adopt the following notation from Section 3 of \cite{MP09}.
\paragraph{Notation.}
If $s,t,\dl\geq 0$ and $x\in \re$, let $\mathbb{G}_{\dl}(s,t,x)=G_{(t-s)_{+}+\dl}(u_{(s-\dl)_{+}})(x)$ and $F_\dl(s,t,x)=-\frac{d}{dx}\mathbb{G}_{\dl}(s,t,x) \equiv -\mathbb{G}'_{\dl}(s,t,x)$, if the derivative exists. \\\\
We will need the following Lemma.
\begin{lemma}\label{Lem.Fdl}
$\mathbb{G}'_{\dl}(s,t,x)$ exists for all $(s,t,x) \in \re^2_+\times \re$, is jointly continuous in $(s,t,x)$, and satisfies
\bn \label{Fdl2}
F_{\dl}(s,t,x)=\int_{0}^{(s-\dl)_{+}}\int_{\re}G'_{(t\vee s)-r}(y-x)D(r,y)W(dr,dy), \textrm{ for all } s\in\re_+, \nonumber \\ P-\textrm{a.s. for all } (t,x)\in \re_+\times \re.
\en
\end{lemma}
The proof of Lemma \ref{Lem.Fdl} is similar to the proof of Lemma 3.1 in \cite{MP09} and hence is omitted.
\begin{remark} \label{remark32}
Since $\mathbb{G}_{\dl}(t,t,x)=u_{1,\dl}(t,x)$, as a special case of Lemma \ref{Lem.Fdl} we get that $u'_{1,\dl}(t,x)$ is a.s. jointly continuous and satisfies
\be \label{Fdl}
u'_{1,\dl}(t,x)=-\int_{0}^{(t-\dl)_{+}}\int_{\re}G'_{t-s}(y-x)D(s,y)W(ds,dy),\ \ P-\textrm{a.s. for all } (t,x)\in \re_+\times \re.
\ee
\end{remark}
Now we define some constants that are used repeatedly throughout the proof.
\paragraph{Notation.}
We introduce the following grid of $\beta$. Let
\be \label{L-def}
L=L(\eps_0,\eps_1)=\bigg\lfloor\big(\frac{\eta}{\eta+1}-6\eta\eps_1\big)\frac{1}{\eps_0}\bigg\rfloor,
\ee
and
\bn \label{beta-def}
\beta_i&=&i\eps_0\in\bigg[0,\frac{\eta}{\eta+1}-6\eta\eps_1\bigg], \ \ \alpha_i=2\bigg(\frac{\beta_i}{\eta}+\eps_1\bigg) \in \bigg[0,\frac{2}{\eta+1}\bigg], \ \ i=0,...,L, \nonumber \\
\beta_{L+1} &=&\frac{\eta}{\eta+1} - \eta\eps_1.
\en
Note that $\beta = \beta_i, \ \ i=0,...L+1,$ satisfies
\be \label{beta-up-lim}
0\leq \beta \leq \frac{\eta}{\eta+1} - \eta\eps_1.
\ee
\begin{definition} \label{Xn-def}
For $(t,x)\in \re_+\times \re$,
\bn
\hat{x}_n(t,x)(\omega)&=&\inf\{y\in [x-a_n^{\al},x+a_n^{\al}]:|u(t,y)|=\inf\{|u(t,z)|:|z-x|\leq a_n^{\al} \} \} \nonumber \\
 &\in& [x-a_n^{\al},x+a_n^{\al}].
\en
\end{definition}
In what follows we introduce some notation which is relevant to the support of the measure $\mu$.
\paragraph{Notation} We fix $K_0\in \mathds{N}^{\geq K_1}$. Let $\rho\in(0,1-\al]$.
Since $\cardim(\mu)=\eta$, then, by Definition \ref{def-cardim}, there exists a set $A\subset \re $ such that $\mu(A^c) =0$ and $\dim_H(A)=\eta$. Then there exists $N(n,K_0,\eta,\rho,\eps_0)\in\mathds{N}$ such that
\bn \label{cov0}
N(n,K_0,\eta,\rho,\eps_0)\leq 2K_0a_n^{-\rho-\eps_0},
\en
and a covering, $\{\tilde V^{n,\rho}_i\}_{i=1}^{N(n,K_0,\eta,\rho,\eps_0)}$,
such that $\tilde V^{n,\rho}_i$ is an open interval of length bounded by $|\tilde V^{n,\rho}_i|\leq a_n^{\rho/\eta}$, for all   $i=1,\dots, N(n,K_0,\eta,\rho,\eps_0)$ and
\bn \label{cov1}
(A\cap [-K_0,K_0])\subset \cup_{i=1}^{N(n,K_0,\eta,\rho,\eps_0)}\tilde V^{n,\rho}_i.
\en
We refer to Sections 2.1 and 2.2 in \cite{falconer}, for the definitions of the Hausdorff measure and the Hausdorff dimension (respectively). Note that (\ref{cov0}) and (\ref{cov1}) follow immediately from those definitions.
We define the following extension of the covering above. We denote by $V^{n,\rho}_i=\{x:|x-y| \leq  a_n^{\rho/\eta} \textrm{ for some } y\in \tilde V^{n,\rho}_i\}$. Note that $| V^{n,\rho}_i|\leq 3a_n^{\rho/\eta}$, for all  $i=1,\dots, N(n,K_0,\eta,\rho,\eps_0)$.
From (\ref{cov0}) and (\ref{cov1}) we get
\bn\label{Cover-Bound}
|V^{n,\rho}_i|N(n,K_0,\eta,\rho,\eps_0)&\leq& 6K_0a_n^{\rho/\eta}a_n^{-\rho-\eps_0}\nonumber \\
&=&6K_0a_n^{(1/\eta-1)\rho-\eps_0}.
\en
We denote by $\mathds{V}^{n,\eta,\rho,\eps_0}=\cup_{i=1}^{N(n,K_0,\eta,\rho,\eps_0)}V^{n,\rho}_i$.
We will use the covering $\mathds{V}^{n,\eta,\rho,\eps_0}$ repeatedly throughout this proof.
 \\\\
We define the following sets.
If $s\geq 0$ set
\bn
J_{n,0}(s)&=&\big\{x:|x|\leq K_0, | \< u_s,G_{a_n^{2\al+2\eps_0}}(x-\cdot) \> | \leq \frac{a_n}{2}, u'_{1,a^{2\al}_n}(s,\hat{x}_n(s,x))\geq \frac{a_n^{\eps_0}}{4}\big\} \nonumber \\
J_{n,L}(s)&=&\big\{x:|x|\leq K_0, | \< u_s,G_{a_n^{2\al+2\eps_0}}(x-\cdot) \> | \leq \frac{a_n}{2}, u'_{1,a^{2\al}_n}(s,\hat{x}_n(s,x))\in [0,\frac{a_n^{\beta_L}}{4}]\big\}. \nonumber \\
\en
For $i=1,...,L-1$ set
\bn
J_{n,i}(s)=\big\{x\in\mathds{V}^{n,\eta,\beta_i,\eps_0}: | \< u_s,G_{a_n^{2\al+2\eps_0}}(x-\cdot) \> | \leq \frac{a_n}{2}, u'_{1,a^{2\al}_n}(s,\hat{x}_n(s,x))\in [\frac{a_n^{\beta_{i+1}}}{4},\frac{a_n^{\beta_i}}{4}]\big\}. \nonumber \\
\en
If $t_0>0$ as in (\ref{H2}) and $i=0,...,L$, define
\be
J_{n,i}=\big\{(s,x):0\leq s, \ \ x\in J_{n,i}(s)\big\},
\ee
We also define
\bn \label{cover1}
\hat J_{n}(s)=\big\{x: |x|\leq K_0, \textrm{ and } x \not \in \cup_{i=1}^{L-1}\mathds{V}^{n,\eta,\beta_i,\eps_0} \big\}.
\en
Note that
\bq
\big\{x:|x|\leq K_0, \ u'_{1,a^{2\al}_n}(s,\hat{x}_n(s,x))\geq 0, \  | \< u_s,G_{a_n^{2\al+2\eps_0}}(x-\cdot) \> | \leq \frac{a_n}{2}\big\} \subset \big(\cup_{i=0}^{L}J_{n,i}(s)\big)\cup \hat J_{n}(s), \ \forall s \geq 0.
\eq
If $0\leq t\leq t_0$, let
\be \label{tildeIn}
\hat I^n(t)=a_n^{-1-\al-\eps_0-\frac{2}{n}}\int_{0}^{t}\int_{\re}\int_{\re}\mathds{1}_{\hat J_{n}(s)}(x)e^{2 R_1|y|}|u(s,y)|^{2\gamma}
G_{a_n^{2\al+2\eps_0}}(x-y)\Psi_s(x)\mu(dy)dx ds,
\ee
and
\be \label{Iin}
I_i^n(t)=a_n^{-1-\al-\eps_0-\frac{2}{n}}\int_{0}^{t}\int_{\re}\int_{\re}\mathds{1}_{J_{n,i}(s)}(x)e^{2 R_1|y|}|u(s,y)|^{2\gamma}
G_{a_n^{2\al+2\eps_0}}(x-y)\Psi_s(x)\mu(dy)dx ds.
\ee
The following lemma gives us an essential bound on the heat kernel.
\begin{lemma}  \label{LemmaNewBound2}
Let $T>0$. Let $\mu\in M_f^\eta(\re)$ for some $\eta\in(0,1)$. There is a $C_{\ref{LemmaNewBound2}}(\eta,\nu_1,\lam)>0$ so that for any $\nu_1\in(0,1/2)$
\bn
\int_{\re}e^{\lam|x-y|}G_{t}(x-y)\mathds{1}_{\{|x-y|>t^{1/2-\nu_1}\}}\mu(dy) \leq C_{\ref{LemmaNewBound2}}(\eta,\nu_1,\lam)\exp\{t/4-t^{-2\nu_1}/32\}, \ \forall x\in \re, \ t\in[0,T],\ \lam>0. \nonumber
\en
\end{lemma}
The proof of Lemma \ref{LemmaNewBound2} uses the same estimates as the proof of Lemma 4.4(b) from \cite{MP09} (see the equation after (4.20) in \cite{MP09}), and hence it is omitted. \medskip \\
The following lemma is one of the ingredients in the proof of (\ref{H2}) in Proposition \ref{PropStop}.
\begin{lemma} \label{Lemma-tildeJ}
\bn  \label{til0}
\lim_{n \rr \infty} \hat I^n(t\wedge T_K)=0, \ \ \forall \ t\leq t_0.
\en
\end{lemma}
\begin{proof}
Let $y\in A\cap [-K_1,K_1]$, where $A$ was defined before (\ref{cov0}). Note that by (\ref{alpha_0}), (\ref{beta-def}) and the construction of $\mathds{V}^{n,\eta,\beta_i,\eps_0}$, if $x\in \hat J_{n}(s)$ then $|x-y|> a_n^{\al}$.
Use this and Lemma \ref{LemmaNewBound2} to get
\bn \label{til1}
\hat I^n(t\wedge T_K)&\leq& a_n^{-1-\al-\eps_0-\frac{2}{n}}\int_{0}^{t_0}\int_{\re}\int_{\re}\mathds{1}_{\hat J_{n}(s)}(x)e^{2 R_1|y|}
e^{2\gamma|y|}G_{a_n^{2\al+2\eps_0}}(x-y)\Psi_s(x)\mu(dy)dxds \nonumber \\
&\leq& a_n^{-1-\al-\eps_0-\frac{2}{n}}\int_{0}^{t_0}\int_{\re}\int_{\re}\mathds{1}_{\{|x-y|> a_n^{2\al}\}}e^{2 (R_1+2)|y|} G_{a_n^{2\al+2\eps_0}}(x-y)\Psi_s(x)\mu(dy)dxds \nonumber \\    \nonumber \\
&\leq & C_{\ref{LemmaNewBound2}}(\eta,\eps_0,R_1)a_n^{-1-\al-\eps_0-\frac{2}{n}}\int_{0}^{t_0}\int_{\re}e^{-a_n^{-2\eps_0}/32}\Psi_s(x)dxds \nonumber \\
&\leq & C(\eta,\eps_0,R_1,t_0,K_1,\|\Psi \|_{\infty})a_n^{-1-\al-\eps_0-\frac{2}{n}}e^{-a_n^{-2\eps_0}/32}.
\en
From (\ref{an0}) and (\ref{til1}) we get (\ref{til0}).
\end{proof}
\\\\
Let
\bn \label{Iplus0}
&&I_{+}^n(t)  \\
&&=a_n^{-1-\al-\eps_0-\frac{2}{n}}\int_{0}^{t_0}\int_{\re}\int_{\re}\mathds{1}_{\{ u'_{1,a^{2\al}_n}(s,\hat{x}_n(s,x))\geq 0\}}\mathds{1}_{\{| \< u_s,G_{a_n^{2\al+2\eps_0}}(x-\cdot) \> | \leq a_n/2 \} }e^{2 R_1|y|}|u(s,y)|^{2\gamma}
\nonumber \\
&&{} \ \ \times G_{a_n^{2\al+2\eps_0}}(x-y)\Psi(x)\mu(dy)dx ds. \nonumber
\en
Then, to prove Proposition \ref{PropStop}, it suffices to construct the sequence of stopping times \\ $\{U_{M,n}\equiv U_{M,n,K_0}:M,n\in \mathds{N}\}$ satisfying (\ref{H1}) and
\bn \label{Iplus}
\textrm{for each } M\in \mathds{N},  \ \ \lim_{n\rr\infty} E(I_{+}^n(t_0\wedge U_{M,n}))=0.
\en
Note that (\ref{Iplus}) implies (\ref{H2}) by symmetry (replace $u_1$ with $u_2$).
By (\ref{cover1})--(\ref{tildeIn}) and (\ref{Iplus0}) we have
\be \label{Iplusbound}
I_{+}^n\leq \sum_{i=0}^{L}I_i^n(t)+\hat I_n(t), \ \ \forall \ t\leq t_0.
\ee
Therefore, from (\ref{Iplusbound}) and Lemma \ref{Lemma-tildeJ}, to prove (\ref{H2}) it is enough to show that for $i=0,...,L$,
\bn \label{Iplus2}
\textrm{for all } M\in \mathds{N},  \ \ \lim_{n\rr\infty} E(I_{i}^n(t_0\wedge U_{M,n}))=0.
\en
\paragraph{Notation.} Let $\bar{l}(\beta)=a_n^{\beta/\eta+5\eps_1}$.
Now introduce the following related sets:
\bn
\tilde{J}_{n,0}(s)&=&\big\{x \in[-K_0,K_0]:  | \< u_s,G_{a_n^{2\al+2\eps_0}}(\cdot-x) \> | \leq \frac{a_n}{2}, u'_{1,a_n^{\alpha_0}}(s+a_n^{2\al+2\eps_0},x')\geq \frac{a_n^{\beta_1}}{16}, \nonumber \\
&&\ \forall x'\in [x-5\bar{l}(\beta_0),x+5\bar{l}(\beta_0)],\nonumber \\
&&\ |\tilde u_{2,a_n^{\alpha_0}}(s,a_n^{2\al+2\eps_0},x')-\tilde u_{2,a_n^{\alpha_0}}(s,a_n^{2\al+2\eps_0},x'')| \leq 2^{-75}a_n^{\beta_1}(|x'-x''|\vee a_n), \nonumber \\
&& \forall x'\in [x-4a_n^{\al},x+4a_n^{\al}], \ x''\in [x'-\bar{l}(\beta_0),x'+\bar{l}(\beta_0)] , \textrm{ and }  \nonumber \\
&&\int_{\re}e^{2 R_1|y|}|u(s,y)|^{2\gamma}G_{a_n^{2\al+2\eps_0}}(x-y)\mu(dy) \leq C(R_1,K_0,\eta,\eps_0)a_n^{2\gamma\al-\al(1-\eta)-2\eps_0} \nonumber \big\},
\en
\bn \label{JtildeL}
\tilde{J}_{n,L}(s)&=&\big\{x \in[-K_0,K_0]: | \< u_s,G_{a_n^{2\al+2\eps_0}}(\cdot-x)  \> | \leq \frac{a_n}{2}, u'_{1,a_n^{\alpha_L}}(s+a_n^{2\al+2\eps_0},x')\leq a_n^{\beta_L},  \\ &&
\forall x'\in [x-5\bar{l}(\beta_L),x+5\bar{l}(\beta_L)],\nonumber \\
&&\ |\tilde u_{2,a_n^{\alpha_L}}(s,a_n^{2\al+2\eps_0},x')- \tilde u_{2,a_n^{\alpha_L}}(s,a_n^{2\al+2\eps_0},x'')| \leq 2^{-75}a_n^{\beta_{L+1}}(|x'-x''|\vee a_n), \nonumber \\
&& \forall x'\in [x-4a_n^{\al},x+4a_n^{\al}], \ x''\in [x'-\bar{l}(\beta_L),x'+\bar{l}(\beta_L)]  , \textrm{ and }  \nonumber \\
&&\int_{\re}e^{2 R_1|y|}|u(s,y)|^{2\gamma}G_{a_n^{2\al+2\eps_0}}(x-y)\mu(dy) \leq C(R_1,K_0,\eta,\eps_0)a_n^{2\gamma(\al+\beta_L)-\al(1-\eta)-2\eps_0} \big\}, \nonumber
\en
and for $i\in\{1,...,L-1\}$,
\bn \label{Jtildei}
\tilde{J}_{n,i}(s)&=&\big\{x \in\mathds{V}^{\eta,\beta_i,\eps_0}: | \< u_s,G_{a_n^{2\al+2\eps_0}}(\cdot-x) \> | \leq \frac{a_n}{2}, u'_{1,a_n^{\alpha_i}}(s+a_n^{2\al+2\eps_0},x')\in [a_n^{\beta_{i+1}}/16,a_n^{\beta_i}],   \\
&& \forall x'\in [x-5\bar{l}(\beta_i),x+5\bar{l}(\beta_i)],\nonumber \\
&&\ |\tilde u_{2,a_n^{\alpha_i}}(s,a_n^{2\al+2\eps_0},x')-\tilde u_{2,a_n^{\alpha_i}}(s,a_n^{2\al+2\eps_0},x'')| \leq 2^{-75}a_n^{\beta_{i+1}}( |x'-x''|\vee a_n), \nonumber \\
&& \forall x'\in [x-4a_n^{\al},x+4a_n^{\al}], \ x''\in [x'-\bar{l}(\beta_i),x'+\bar{l}(\beta_i)]  , \textrm{ and }  \nonumber \\
&&\int_{\re}e^{2 R_1|y|}|u(s,y)|^{2\gamma}G_{a_n^{2\al+2\eps_0}}(x-y)\mu(dy) \leq C(R_1,K_0,\eta,\eps_0)a_n^{2\gamma(\al+\beta_i)-\al(1-\eta)-2\eps_0} \big\}. \nonumber
\en
Finally for $0\leq i \leq L$, set
\be
\tilde{J}_{n,i}=\{(s,x):s\geq 0, x\in\tilde{J}_{n,i}(s)\}.
\ee
\paragraph{Notation}
Denote by
$n_M(\eps_1) = \inf\{n\in \mathds{N}:  a^{\eps_1}_n \leq 2^{-M-8}\}$, $n_0(\eps_1,\eps_0)= \sup\{n\in \mathds{N}:  a_n^{\al} < 2^{-a_n^{-\eps_0\eps_1/4}}\}$, where $\sup\emptyset =1$ and
\bn \label{n1}
 n_{1}(\eps_0,K)= \inf\bigg\{n\in\mathds{N}:a_n\int_{-a_n^{-\eps_0}}^{a_n^{-\eps_0}}G_{1}(y)dy
-2K\int_{a_n^{-\eps_0}}^{\infty}e^{|y|}G_{1}(y)dy > \frac{a_n}{2}\bigg\}.
\en
The following proposition corresponds to Proposition 3.3 in \cite{MP09}. We will prove this proposition in Section \ref{Sec-Proof-Prop3.3}.
\begin{proposition}\label{PropStopTimes}
$\tilde{J}_{n,i}(s)$ is a compact set for all $s\geq 0$. There exist stopping times \\ $\{U_{M,n}=U_{M,n,K_0}:M,n\in\mathds{N}\}$, satisfying (\ref{H1}) and $n_2(\eps_0,\al,\gamma,\eta,K,R_1)\in \mathds{N}$ such that for $i\in \{0,1,...,L\}$, $\tilde{J}_{n,i}(s)$ contains $J_{n,i}(s)$ for all $0\leq s \leq U_{M,n}$ and
\be \label{nM}
n > n_M(\eps_1) \vee n_0(\eps_0,\eps_1)\vee n_{1}(\eps_0,K)\vee n_2(\eps_0,\al,\gamma,\eta,K,R_1).
\ee
\end{proposition}
Throughout the rest of this section we assume that that the parameters $M,n\in \mathds{N}$ satisfy (\ref{nM}). \\\\
The following Lemma corresponds to Lemma 3.4 in \cite{MP09}.
\begin{lemma} \label{Lem3.4MP}
Assume $i\in\{0,...,L\}$, $x\in\tilde{J}_{n,i}(s)$ and $|x-x'|\leq 4a_n^{\al}$.
\begin{itemize}
  \item[\bf{(a)}] If $i>0$, then
  \bd |\< u_s,G_{a_n^{2\al+2\eps_0}}(\cdot-x'') \> -\< u_s,G_{a_n^{2\al+2\eps_0}}(\cdot-x') \> | \leq 2a_n^{\beta_i}(|x'-x''|\vee a_n), \ \ \forall |x''-x'|\leq \bar{l}_n(\beta_i).
  \ed
  \item [\bf{(b)}] If $i<L$, and $  a_n\leq |x''-x'|\leq \bar{l}_n(\beta_i)$, then
  \bd
  \< u_s,G_{a_n^{2\al+2\eps_0}}(\cdot-x'') \> -\< u_s,G_{a_n^{2\al+2\eps_0}}(\cdot-x') \>
  \left\{
   \begin{array}{cc}
    \geq 2^{-5}a_n^{\beta_{i}}(x''-x') & \textrm{if } x''\geq x', \\
    \leq 2^{-5}a_n^{\beta_{i}}(x''-x') & \textrm{if } x'\geq x''. \\
     \end{array}
     \right.
  \ed
\end{itemize}
\end{lemma}
\begin{proof}
(a) For $n,i,s,x,x'',x'$ as in (a) we have
\bn
|x'-x|\vee|x''-x'|\leq 2\bar l_n(\beta_i).
\en
From the definition of $\tilde J_{n,i}$, (\ref{decomp-mol}) and the mean value theorem we have
\bn
|\< u_s,G_{a_n^{2\al+2\eps_0}}(\cdot-x'') \> -\< u_s,G_{a_n^{2\al+2\eps_0}}(\cdot-x') \> |&\leq& |u_{1,a_n^{\alpha_i}}(s+a_n^{2\al+2\eps_0},x'')-u_{1,a_n^{\alpha_i}}(s+a_n^{2\al+2\eps_0},x')| \nonumber \\
&&+|\tilde u_{2,a_n^{\alpha_i}}(s,a_n^{2\al+2\eps_0},x'')-\tilde u_{2,a_n^{\alpha_i}}(s,a_n^{2\al+2\eps_0},x')| \nonumber \\
&\leq&a_n^{\beta_i}|x'-x''|+2^{-75}a_n^{\beta_i}(|x'-x''|\vee a_n) \nonumber \\
&\leq& 2a_n^{\beta_i}(|x'-x''|\vee a_n).
\en
(b) For $n,i,s,x,x'',x'$ as in (b) with $x''\geq x'$ we have
\bn
\< u_s,G_{a_n^{2\al+2\eps_0}}(\cdot-x'') \> -\< u_s,G_{a_n^{2\al+2\eps_0}}(\cdot-x') \> &=& u_{1,a_n^{\alpha_i}}(s+a_n^{2\al+2\eps_0},x'')-u_{1,a_n^{\alpha_i}}(s+a_n^{2\al+2\eps_0},x') \nonumber \\
&&+\tilde u_{2,a_n^{\alpha_i}}(s,a_n^{2\al+2\eps_0},x'')-\tilde u_{2,a_n^{\alpha_i}}(s,a_n^{2\al+2\eps_0},x') \nonumber \\
&\geq& \frac{a_n^{\beta_i}}{16}|x'-x''|-2^{-75}a_n^{\beta_i}|x'-x''| \nonumber \\
&\geq& \frac{a_n^{\beta_i}}{32}|x'-x''|.
\en
\end{proof}
\paragraph{Notation.} Let $l_n(\beta_i)=65 a_n^{1-\beta_{i+1}}$.
The following Lemma corresponds to Lemmas 3.6 in \cite{MP09}.
\begin{lemma}\label{Lem3.6MP}
If $i\in\{0,...,L\}$, then
\be
l_n(\beta_i)< a_n^{\al} <\frac{1}{2}\bar{l}_n(\beta_i).
\ee
\end{lemma}
The proof of Lemma \ref{Lem3.6MP} in our case is similar to the proof of Lemma 3.6 in \cite{MP09}, hence it is omitted.
\begin{lemma} \label{LebJ}
\begin{itemize}
\item [\bf{(a)}] For all $s\geq 0$,
\bd
|\tilde{J}_{n,0}(s)|\leq 10K_0\bar{l}_n(\beta_0)^{-1}l_n(\beta_0).
\ed
\item [\bf{(b)}] For all $i\in\{1,...,L-1\}$ and $s\geq 0$,
\bd
|\tilde{J}_{n,i}(s)|\leq 10K_0\bar{l}_n(\beta_i)^{-\eta}l_n(\beta_i).
\ed
\end{itemize}
\end{lemma}
\begin{proof} The proof of (a) is follows the same lines as the proof of Lemma 3.7 in \cite{MP09} for the case where $i=0$. The proof (b) also follows the same lines as the proof of Lemma 3.7 in \cite{MP09} for the case where $i=\{1,...,L-1\}$. The major difference is that in our case $\tilde{J}_{n,i}(s)\subset \mathds{V}^{n,\eta,\beta_i,\eps_0}$, so instead of covering $\tilde{J}_{n,i}(s)$ with $\bar{l}_n(\beta_i)^{-1}$ balls as in \cite{MP09}, we can cover it with a smaller number of balls which is proportional to  $\bar{l}_n(\beta_i)^{-\eta}$.
\end{proof}
\paragraph{Proof of (\ref{H2}) in Proposition \ref{PropStop}.}
Recall that to prove (\ref{H2}) in Proposition \ref{PropStop}, it is enough to show (\ref{Iplus}), for $i=0,...,L$. In fact we will prove a stronger result. Recall that $n_M$, $n_0$, $n_1$, $n_2$ were defined before and in Proposition \ref{PropStopTimes}. We will show that for
\be \label{317}
n > n_M(\eps_1) \vee n_0(\eps_0,\eps_1)\vee \frac{2}{\eps_1}\vee n_{1}(\eps_0,K_0,K_1)\vee n_2(\eps_0,\al,\gamma,\eta,K_0,R_1).
\ee
\be \label{Iibound}
I_i^n(t_0\wedge U_{M,n}) \leq C(\eta,K_0,t_0,\|\Psi\|_{\infty})a_n^{\gamma-1+\frac{\eta}{2(\eta+1)}-13\eps_1},
\ee
which implies (\ref{H2}) since $\gamma > 1-\frac{\eta}{2(\eta+1)}+100\eps_1$.
By Proposition \ref{PropStopTimes}, $\frac{2}{n}<\eps_1$ (by (\ref{317})) and (\ref{eps1}) we have
\bn \label{GenIn}
&& I_i^n(t_0\wedge U_{M,n}) \\
&&=a_n^{-1-\al-\eps_0-\frac{2}{n}}\int_{0}^{t_0\wedge U_{M,n}}\int_{\re}\int_{\re}e^{2R_1|y|}\mathds{1}_{J_{n,i}(s)}(x)|u(s,y)|^{2\gamma}
G_{a_n^{2\al+2\eps_0}}(x-y)\Psi_s(x)\mu(dy)dx ds \nonumber \\
&&\leq a_n^{-1-\al-2\eps_1}\int_{0}^{t_0}\int_{\re}\int_{\re}e^{2R_1|y|}\mathds{1}_{\{s< U_{M,n}\}}\mathds{1}_{\tilde{J}_{n,i}(s)}(x)G_{a_n^{2\al+2\eps_0}}(x-y)|u(s,y)|^{2\gamma}
\Psi_s(x)\mu(dy)dx ds. \nonumber
\en
Consider first the case where $i=0$. For $x\in \tilde J_{n,0}(s)$ we have
\bn \label{nb-b}
\int_{\re}e^{2R_1|y|}|u(s,y)|^{2\gamma}G_{a_n^{2\al+2\eps_0}}(x-y)\mu(dy) \leq C(R_1,K_0,\eta,\eps_0)a_n^{2\gamma\al-\al(1-\eta)-2\eps_0}.
\en
We get from (\ref{GenIn}), (\ref{nb-b}), Lemma \ref{LebJ}(a) and Lemma \ref{condmu}(a),
\bn \label{CIn0}
&& I_0^n(t_0\wedge U_{M,n}) \\
&&\leq C(R_1,K_0,\eta,\eps_0)a_n^{-1-\al-2\eps_1} a_n^{2\gamma\al-\al(1-\eta)-2\eps_0}\|\Psi\|_{\infty}\int_{0}^{t_0}\int_{\re}\mathds{1}_{\tilde{J}_{n,0}(s)}(x)\mathds{1}_{[-K_0,K_0]}(x)dx ds \nonumber \\
&& \leq C(R_1,K_0,\eta,\eps_0)a_n^{-1-\al-2\eps_1} a_n^{2\gamma\al-\al(1-\eta)-2\eps_0}\|\Psi\|_{\infty}\int_{0}^{t_0}\int_{\re}\mathds{1}_{\tilde{J}_{n,0}(s)}(x)
dx ds \nonumber \\
&&\leq  C(R_1,K_0,\eta,\eps_0)a_n^{-1-\al-2\eps_1} a_n^{2\gamma\al-\al(1-\eta)-2\eps_0}\|\Psi\|_{\infty}\int_{0}^{t_0}|\tilde{J}_{n,0}(s)|ds \nonumber \\
&&\leq  C(R_1,K_0,\eta,\eps_0,t_0,\|\Psi\|_{\infty})a_n^{-1-\al-2\eps_1} a_n^{2\gamma\al-\al(1-\eta)-2\eps_0}10K_0\bar{l}_n(\beta_0)^{-1}l_n(\beta_0). \nonumber
\en
From the definitions of  $\bar{l}_n(\beta_i),l_n(\beta_i)$ we get, 
\bn \label{I0est}
I_0^n(t_0\wedge U_{M,n}) &\leq&C(R_1,K_0,\eta,\eps_0,t_0,\|\Psi\|_{\infty})a_n^{-1-\al-2\eps_1} a_n^{2\gamma\al-\al(1-\eta)-2\eps_0}a_n^{-5\eps_1}65a_n^{1-\eps_0} \nonumber \\
&\leq&C(R_1,K_0,\eta,\eps_0,t_0,\|\Psi\|_{\infty})a_n^{\rho_{0}}.
\en
From (\ref{alpha_0}), (\ref{eps1}) and (\ref{optsol}) we have
\bn \label{I0est1}
\rho_{0}&=& -1-\al-2\eps_1+2\gamma\al+\al(\eta-1)-5\eps_1+1-3\eps_0 \nonumber \\
&\geq& 2\al\gamma-2\al+\al\eta-8\eps_1\nonumber \\
&\geq& \gamma-\frac{2\al-\al\eta}{2\al}-8\eps_1 \nonumber \\
&\geq& \gamma-1+\frac{\eta}{2}-8\eps_1.
\en
Consider now $i=\{1,...,L\}$. Assume $x\in \tilde{J}_{n,i}(s)$. 
Repeat the same steps as in (\ref{CIn0}) to get,
\bn \label{Ini1-L}
&& I_i^n(t_0\wedge U_{M,n}) \\
&&\leq C(R_1,K_0,\eta,\eps_0,t_0,\|\Psi\|_{\infty}) a_n^{-1-\al-2\eps_1}a_n^{2\gamma(\beta_i+\al)+\al(\eta-1)-2\eps_0}\int_{0}^{t_0}|\tilde{J}_{n,i}(s)|ds. \nonumber
\en
For $i=\{1,\cdots,L-1\}$ apply Lemma \ref{LebJ}(b) to (\ref{Ini1-L}) to get
\bn \label{Ini1-L2}
&& I_i^n(t_0\wedge U_{M,n})  \nonumber \\
&&\leq   C(K_0,\eta,\eps_0,t_0,\|\Psi\|_{\infty},\omega) a_n^{-1-\al-2\eps_1+2\gamma(\beta_i+\al)+\al(\eta-1)-2\eps_0}10K_0\bar{l}_n(\beta_i)^{-\eta}l_n(\beta_i) \nonumber \\
&&\leq  C(K_0,\eta,\eps_0,t_0,\|\Psi\|_{\infty},\omega) a_n^{-1-\al-2\eps_1+2\gamma(\beta_i+\al)+\al(\eta-1)-2\eps_0}a_n^{-\beta_i-5\eps_1}65 a_n^{1-\beta_{i+1}} \nonumber \\
&&\leq C(K_0,\eta,\eps_0,t_0,\|\Psi\|_{\infty},\omega)a_n^{\rho_{1,i}}.
\en
Use (\ref{alpha_0}), (\ref{eps1}), (\ref{optsol}) and (\ref{beta-def}) to get $\beta_i<1-\al$, and
\bn \label{rho1}
\rho_{1,i}&=&2\al\gamma-2\beta_i(1-\gamma)-(2-\eta)\al-3\eps_0-7\eps_1 \nonumber \\
&>& 2\al\gamma-2(1-\al)(1-\gamma)-(2-\eta)\al-3\eps_0-7\eps_1\nonumber \\
&\geq& 2\gamma -\frac{2\eta}{\eta+1}-\frac{2-\eta}{\eta+1}-15\eps_1 \nonumber \\
&\geq& \gamma  - 1+\frac{\eta}{2(\eta+1)}-8\eps_1.
\en
For $i=L$, we repeat the same steps as in (\ref{Ini1-L}) with $\rho = 1-\al=\eta\al$. We also use $\mathds{V}^{n,\eta,\al,\eps_0}$ to bound the integration region, (\ref{Cover-Bound}), (\ref{eps1}), (\ref{beta-up-lim}), (\ref{alpha_0}) and (\ref{optsol}) to get
\bn \label{Ini1-L3}
&&I_L^n(t_0\wedge U_{M,n}) \nonumber \\
&&\leq C(R_1,K_0,\eta,\eps_0,\|\Psi\|_{\infty})a_n^{-1-\al-2\eps_1+2\gamma(\beta_L+\al)+\al(\eta-1)-2\eps_0}\int_{0}^{t_0}\int_{\re}\mathds{1}_{\tilde{J}_{n,i}(s)}(x)dx ds  \nonumber \\
&&\leq C(R_1,K_0,\eta,\eps_0,\|\Psi\|_{\infty})a_n^{-1-\al-2\eps_1+ 2\gamma(\beta_L+\al)+\al(\eta-1)-2\eps_0}\int_{0}^{t_0}\sum_{i=1}^{N(n,\eta,K_0,\eta\al)}
|V^{n,\al}_i| ds  \nonumber \\
&&\leq C(R_1,K_0,\eta,\eps_0,\|\Psi\|_{\infty})t_0a_n^{(\eta-2)\al-1-3\eps_1+2\gamma(\beta_L+\al)}\sum_{i=1}^{N(n,\eta,K_0,\eta\al)}|V^{n,\al}_i| \nonumber \\
&&\leq C(R_1,K_0,\eta,\eps_0,t_0,\|\Psi\|_{\infty})a_n^{(\eta-2)\al-1-3\eps_1+2\gamma(\beta_L+\al)+(1-\eta)\al-\eps_0} \nonumber \\
&&\leq C(R_1,K_0,\eta,\eps_0,t_0,\|\Psi\|_{\infty})a_n^{-1-\al-3\eps_1+ 2\gamma-12\eps_1-\eps_0} \nonumber \\
&&\leq C(R_1,K_0,\eta,\eps_0,t_0,\|\Psi\|_{\infty})a_n^{\gamma-\frac{1+\al}{2}-13\eps_1}  \nonumber \\
&&= C(R_1,K_0,\eta,\eps_0,t_0,\|\Psi\|_{\infty})a_n^{\gamma-1+\frac{\eta}{2(\eta+1)}-13\eps_1}
\en
From (\ref{I0est}),(\ref{I0est1}), (\ref{Ini1-L2}),(\ref{rho1}) and (\ref{Ini1-L3}), it follows that,
\be \label{est2}
I_i^n(t_0\wedge U_{M,n}) \leq C(R_1,K_0,\eta,\eps_0,t_0,\|\Psi\|_{\infty}) a_n^{\gamma  -1+\frac{\eta}{2(1+\eta)}-13\eps_1}, \ \forall \ i=0,...,L.
\ee
\qed

\section{Some Integral Bounds for the Heat Kernel} \label{Sec.Heat-Kernel-Bounds}
In this section we introduce some integral bounds for the heat kernel. These bounds will be useful for the proofs in the following sections. The proofs of the following lemmas are straightforward and therefore they are omitted. First let us recall some useful lemmas from \cite{MP09}, \cite{rosen} and \cite{Zahle2004}.\\\\
For $0\leq p \leq 1$, $q\in\re$ and $0\leq \Delta_2\leq \Delta_1\leq t$, define
\bd
J_{p,q}(\Delta_1,\Delta_2,\Delta)=\int_{t-\Delta_1}^{t-\Delta_2}(t-s)^q\big(1\wedge\frac{\Delta}{t-s}\big)^p ds.
\ed
We will use Lemmas 4.1, 4.2 from \cite{MP09}.
\begin{lemma} \label{Lem4.1MP}
\begin{itemize}
\item [\bf{(a)}] If $q>p-1$, then
\bd
J_{p,q}(\Delta_1,\Delta_2,\Delta)\leq \frac{2}{q+1-p}(\Delta\wedge \Delta_1)^p\Delta_1^{q+1-p}.
\ed
\item [\bf{(b)}] If $-1<q<p-1$, then
\bn
J_{p,q}(\Delta_1,\Delta_2,\Delta)&\leq& ((p-1-q)^{-1}+(q+1)^{-1})\big[(\Delta\wedge \Delta_1)^{q+1}\mathds{1}_{\{\Delta_2\leq \Delta\}}  \nonumber  \\ \nonumber  \\
&&+(\Delta \wedge \Delta_1)^p\Delta_2^{q-p+1}\mathds{1}_{\{\Delta_2> \Delta\}}\big]   \nonumber  \\ \nonumber  \\
&\leq&((p-1-q)^{-1}+(q+1)^{-1})\Delta^p(\Delta\vee \Delta_2)^{q-p+1}. \nonumber
\en
\item [\bf{(c)}] If $q<-1$, then
\bd
J_{p,q}(\Delta_1,\Delta_2,\Delta)\leq 2|q+1|^{-1}(\Delta \wedge \Delta_2)^p\Delta_2^{q+1-p}.
\ed
\end{itemize}
\end{lemma}
 \medskip
Denote
\bd
G'_t(x)=\frac{\partial G_t(x)}{\partial x}.
\ed
\begin{lemma} \label{Lem4.2MP}
\bd
|G'_t(z)|\leq C_{\ref{Lem4.2MP}}t^{-1/2}G_{2t}(z).
\ed
\end{lemma}
The inequities in the following lemma were introduced as equations (2.4e) and (2.4f) in Section 2 of \cite{rosen}.
\begin{lemma} \label{heatest}
For any $0\leq \dl \leq 1$, there exists a constant $C_{\ref{heatest}}>0$ such that
\begin{itemize}
\item [\bf{(a)}] \be \label{rosest}
|G_{t}(x-y)-G_{t}(x'-y)| \leq C_{\ref{heatest}} |x-x'|^{\dl} t^{-(1+\dl)/2}\big(e^{-\frac{(x-y)^2}{2t}}+e^{-\frac{(x'-y)^2}{2t}}\big),
\ee
for all $t \geq 0$, $x,x'\in \re$. \\\\
\item [\bf{(b)}] \be|G'_{t}(x-y)-G'_{t}(x'-y)| \leq C_{\ref{heatest}} |x-x'|^{\dl} t^{-(2+\dl)/2}\big(e^{-\frac{(x-y)^2}{2t}}+e^{-\frac{(x'-y)^2}{2t}}\big),
\ee
for all $t \geq 0$, $x,x'\in \re$.
\end{itemize}
\end{lemma}
We will use the following bound upper bound on the exponential function.
\begin{lemma} \label{Lem-Ker-up-Bound}
Let $a>0$, then for any $\dl>0$, there exists a constant $C_{\ref{Lem-Ker-up-Bound}}(a,\dl)>0$ such that
\bd
e^{-\frac{x^2}{at}}\leq C_{\ref{Lem-Ker-up-Bound}}(a,\dl)t^{\dl/2}|x|^{-\dl},
\ed
for all $ t > 0$, $x\in \re$.
\end{lemma}
The proof of Lemma \ref{Lem-Ker-up-Bound} is trivial, hence it is omitted. \medskip \\
The following lemma puts together the results of Lemma 3.4(c) and Lemma 3.7 from \cite{Zahle2004}.
\begin{lemma} \label{condmu}
Let $\mu \in M^{\eta}_\sigma (\re)$ for some $\eta \in (0,1)$. Let $\varpi\in(0,\eta)$ and $T>0$.
Then for every $\lam \geq 0$, $r\in(0,2+\eta-\varpi)$ and $t\in[0,T]$,
\begin{itemize}
\item [\bf{(a)}] There exists a constant $C_{(\ref{mu1cond})}(r,\lam,T,\eta,\varpi)>0$ such that
\be \label{mu1cond}
e^{-\lam|x|}\int_{\re}e^{\lam|y|}G_{t-s}^r(x-y)\mu(dy) <\frac{C_{(\ref{mu1cond})}(r,\lam,T,\eta,\varpi)}{(t-s)^{(r-\eta+\varpi)/2}}, \ \forall  \ s\in[0,t),  \ \ x\in\re.
\ee
\item [\bf{(b)}] There exists a constant $C_{(\ref{mu2cond})}(r,\lam,T,\eta,\varpi)>0$ such that
\be \label{mu2cond}
\sup_{x\in \re} e^{-\lam|x|}\int_{0}^{t}\int_{\re}e^{\lam|y|}G_{t-s}^r(x-y)\mu(dy)ds <C_{(\ref{mu2cond})}(r,\lam,T,\eta,\varpi), \ \forall  t\in[0,T].
\ee
\item [\bf{(c)}] For every $\dl\in (0,\eta-\varpi)$, there exists a constant $C_{(\ref{mu3cond})}(\dl,T,\eta,\varpi,\lam)>0$ such that
\bn \label{mu3cond}
&&\int_{0}^{t\vee t'} \int_{\re} e^{\lam|y|}(G_{t-s}(x-y)-G_{t'-s}(x'-y))^2 \mu(dy)ds \nonumber \\
&& \leq C_{(\ref{mu3cond})}(\dl,T,\eta,\varpi,\lam)(|t-t'|^{\dl/2}+|x-x'|^\dl)e^{\lam|x|}e^{\lam|x-x'|}, \ \forall t,t' \in [0,T], \ x,x'\in \re, \ \lam>0.
\en
\item [\bf{(d)}] For every $a\in[0,2]$, $\eps>0$ and $\theta\geq 0$, there exists a constant $C_{(\ref{mu4cond})}(\theta,a,\eta,T,\eps)>0$ such that
\bn\label{mu4cond}
\int_{\re}e^{\theta|x-y|}(x-y)^{2\alpha}G_{2(t-r)}(y-x)^2\mu(dy) \leq C_{(\ref{mu4cond})}(\theta,a,\eta,T,\eps) (t-r)^{\alpha-1+\eta/2-\eps}, \ \forall t \in [0,T], \ x\in \re.
\en
\end{itemize}
\end{lemma}
\begin{proof}
(a) and (b) follow immediately from Lemma 3.4(c) in \cite{Zahle2004}.  (c) was introduced as Lemma 3.7 in \cite{Zahle2004}. The proof of (d) follows along the same lines as the proof of Lemma 3.4(c) in \cite{Zahle2004}, and hence it is omitted.
\end{proof}
\paragraph{Notation} Denote by \be \label{Notation-d}
d((t,x),(t',x'))=|t-t'|^{1/2}+|x'-x|.
\ee
The following lemma is a modification of Lemma 4.3 from \cite{MP09}.
\begin{lemma}  \label{Lemma-MP4.3}
Let $\mu\in M_f^\eta(\re)$ for some $\eta\in(0,1)$. Let $\eps>0$ be arbitrarily small.
\begin{itemize}
  \item [\bf{(a)}] Then, there exists a constant $C_{\ref{Lemma-MP4.3}}(\eta,\eps)>0$ such that $s<t\leq t',\ x,x'\in \re$,
\be
 \int_{\re} (G_{t'-s}(x'-y)-G_{t-s}(x-y))^2 \mu(dy) \leq C_{\ref{Lemma-MP4.3}}(\eta,\eps)|t-s|^{\eta/2-1-\eps}\bigg[1\wedge \frac{d((t,x),(t',x'))^2}{t-s}\bigg].\nonumber
\ee
  \item [\bf{(b)}]  For any $R>2$ there is a $C_{\ref{Lemma-MP4.3}}(R,\eta,\eps)>0$ so that for any $0\leq p,r \leq R, \\  \nu_0,\nu_1\in(1/R,1/2)$, $0\leq s \leq t \leq t' \leq R, \ x,x' \in \re$,
\bn
 &&\int_{\re}e^{r|x-y|}|x-y|^p(G_{t'-s}(x'-y)-G_{t-s}(x-y))^2 \mathds{1}_{\{|x-y|>(t'-s)^{1/2-\nu_0}\vee 2|x'-x|\}}\mu(dy) \nonumber \\
 &&\leq C_{\ref{Lemma-MP4.3}}(R,\eta,\eps,\nu_0,\nu_1)|t-s|^{\eta/2-1-\eps}\exp\{-\nu_1(t'-s)^{-2\nu_0}/32\}\nonumber \\
  &&{}\times \bigg[1\wedge \frac{d((t,x),(t',x'))^2}{t-s}\bigg]^{1-(\nu_1/2)}. \nonumber
\en
\end{itemize}
\end{lemma}
The following Lemma is a modification of Lemma 4.4 from \cite{MP09}.
\begin{lemma}  \label{Lemma-MP4.4}
Let $\mu\in M_f^\eta(\re)$ for some $\eta\in(0,1)$. Let $\eps>0$ be arbitrarily small.
\begin{itemize}
  \item [{\bf (a)}] Then, there exists a constant $C_{\ref{Lemma-MP4.4}}(\eta,\eps)>0$ such that $s<t\leq t',\ x,x'\in \re$,
\be
 \int_{\re} (G'_{t'-s}(x'-y)-G'_{t-s}(x-y))^2 \mu(dy) \leq C_{\ref{Lemma-MP4.4}}(\eta,\eps)|t-s|^{\eta/2-\eps-2}\bigg[1\wedge \frac{d((t,x),(t',x'))^2}{t-s}\bigg].\nonumber
\ee
  \item [{\bf (b)}] For any $R>2$ there is a $C_{\ref{Lemma-MP4.4}} (R,\nu_0,\nu_1,\eta,\eps)>0$ so that for any $0\leq p,r \leq R, \\ \nu_0,\nu_1\in(1/R,1/2)$, $0\leq s \leq t \leq t' \leq R, \ x,x' \in \re$,
\bn
 &&\int_{\re}e^{r|x-y|}|x-y|^p(G'_{t'-s}(x'-y)-G'_{t-s}(x-y))^2 \mathds{1}_{\{|x-y|>(t'-s)^{1/2-\nu_0}\vee 2|x'-x|\}}\mu(dy) \nonumber \\
 &&\leq C_{\ref{Lemma-MP4.4}}(R,\nu_0,\nu_1,\eta,\eps)|t-s|^{\eta/2-2-\eps}\exp\{-\nu_1(t'-s)^{-2\nu_0}/64\} \nonumber \\
 &&\times \bigg[1\wedge \frac{d((t,x),(t',x'))^2}{t-s}\bigg]^{1-(\nu_1/2)}. \nonumber
\en
\end{itemize}
\end{lemma}
The following lemma follows from Lemmas \ref{condmu}(a) and \ref{Lemma-MP4.3}.
\begin{lemma} \label{lemma-new2}
Let $\mu\in M_f^\eta(\re)$ for some $\eta\in(0,1)$. Let $\lam>0$ and let $\eps>0$ be arbitrarily small. There is a $C_{\ref{lemma-new2}}(\eta,\eps,\lam)>0$ such that for any  $0 \leq s< t\leq t',\ x,x'\in \re$,
\bn  \label{rt2222}
\int_{\re}e^{\lam|y|}[G_{t-s}(x'-y)-G_{t'-s}(x-y)]^2\mu(dy)ds &\leq& C_{\ref{lemma-new2}}(\eta,\eps,\lam)|t-s|^{\eta/2-\eps-2}d((t,x),(t',x'))^2e^{2\lam|x|}e^{2\lam|x'-x|}.  \nonumber
\en
\end{lemma}
\medskip
Now we are ready to give a proof to Lemma \ref{new-int-bound}.
\paragraph{Proof of Lemma \ref{new-int-bound}}
Let $K>0$. Note that
\bn \label{zzz}
\int_{\re}\int_{\re}e^{|y|}G_{\eps}(z-y)\mathds{1}_{\{|z|\leq K\}}\mu(dy)dz &=&\int_{\re}\int_{\re}e^{|y|}G_{\eps}(z-y)\mathds{1}_{\{|z|\leq K\}}\mathds{1}_{\{|y-z|\leq K+1\}}\mu(dy)dz  \nonumber \\
&&+\int_{\re}\int_{\re}e^{|y|}G_{\eps}(z-y)\mathds{1}_{\{|z|\leq K\}}\mathds{1}_{\{|y-z|> K+1\}}\mu(dy)dz \nonumber \\
&:=&I_1(K,\eps)+ I_2(K,\eps), \ \forall \eps\in(0,1].
\en
For $I_1(K,\eps)$ we get that
\bn \label{zzz1}
I_1(K,\eps)&\leq & C(K)\int_{\re}\int_{\re}G_{\eps}(z-y)dz\mu(dy)  \nonumber \\
&\leq& C(K,\mu(\re)),  \ \forall \eps\in(0,1].
\en
From Lemma \ref{LemmaNewBound2} we get for $I_2(K,\eps)$,
\bn \label{zzz2}
I_2(K,\eps) &\leq & C(K)\int_{\re}\mathds{1}_{\{|z|\leq K\}}\int_{\re}e^{|y|}G_{\eps}(z-y)\mathds{1}_{\{|y-z|> \eps^{1/4}\}}\mu(dy)dz \nonumber \\
&\leq &C(\eta,K)\int_{\re}\mathds{1}_{\{|z|\leq K\}}e^{-\eps/4-\eps^{-1/2}/32}dz \nonumber \\
&\leq &  C(\eta,K) , \ \forall \eps\in(0,1].
\en
From (\ref{zzz})--(\ref{zzz2}) we get (\ref{new-int-bound2}). \qed
\section{Local Bounds on the Difference of Solutions} \label{Sec-Local-bound}
This section is devoted to establishing local bounds on the difference of two solutions of (\ref{SHE}). Theses bounds are crucial for the construction of the stopping times in Proposition \ref{PropStopTimes}. In this section we assume again that $u_1,u_2$ are two solutions of (\ref{SHE}). We denote by $u=u^1-u^2$ and we assume the hypotheses of Theorem \ref{holder-conj} and (\ref{HolSigmaCon2}). Our argument follows the same lines as the argument in Section 5 of \cite{MP09}.
\paragraph{Notation}
Let
\bn \label{Def-Z-K-N-Set}
Z(K,N)(\omega)&=&\{(t,x)\in [0,T_K]\times[-K,K]: \textrm{ There is a } (\hat{t}_0,\hat{x}_0)\in[0,T_K]\times \re,  \\ && \textrm{ such that } \nonumber d((t,x),(\hat{t}_0,\hat{x}_0))\leq 2^{-N} \textrm{ and } |\tilde{u}(\hat{t}_0,\hat{x}_0)|\leq 2^{-N\xi} \}.
\en
For all $K,N,n\in \mathds{N}$ and $\beta \in (0,\frac{\eta}{\eta+1}]$,
\bn
Z(N,n,K,\beta)&=&\{(t,x)\in[0,T_K]\times [-K,K] : \textrm{there is a } (\hat{t}_0,\hat{x}_0)\in [0,T_K]\times \re \nonumber \\ && \textrm{ such that } d((\hat{t}_0,\hat{x}_0),(t,x))\leq 2^{-N}, \nonumber \\ && \ |u(\hat{t}_0,\hat{x}_0)|\leq a_n\wedge (a_n^{1-\al}2^{-N}),   \textrm{ and } |u'_{1,a^{2\al}_n}(\hat{t}_0,\hat{x}_0)|\leq a_n^{\beta}\},  \nonumber
\en
and for $\beta =0$ define $Z(N,n,K,0)(\omega)=Z(N,n,K)(\omega)$ as above, but with the condition on $|u'_{1,a_n}(\hat{t}_0,\hat{x}_0)|\leq a_n^{\beta}$ omitted.   \\ \\
Recall that we fixed $\eta\in(0,1)$ and $\gamma\in \big(1-\frac{\eta}{2(\eta+1)},1\big)$. Let
\bn \label{gamma-m}
\gamma_m=\frac{(\gamma-1+\eta/2)(1-\gamma^m)}{1-\gamma} + 1, \ \ \tilde{\gamma}_m=\gamma_m\wedge (1+\eta).
\en
From (\ref{gamma-m}) we get
\bn \label{gamma-rec}
\gamma_{m+1}=\gamma\gamma_m+\eta/2, \ \ \gamma_0=1.
\en
Note that $\gamma_m$ increases to $\gamma_\infty=\frac{(\gamma-1+\eta/2)}{1-\gamma} + 1=\frac{\eta}{2(1-\gamma)}>\eta+1$ (the last inequality follows by (\ref{optsol})) and therefore we can define a finite natural number, $\bar{m}>1$, by
\bn \label{m-bar-def}
\bar{m}=\min\{m:\gamma_{m+1}>\eta+1\}= \min\{m:\gamma\gamma_{m}>1+\frac{\eta}{2}\}.
\en
Note that
\be \label{gamma-m-up-lim}
\tilde{\gamma}_{\bar{m}+1}=\eta+1.
\ee
\begin{remark}
 In this section we often use the constrain $\gamma>1-\eta/2$. Note that it holds trivially for $\gamma$ satisfying (\ref{optsol}).
\end{remark}
\begin{definition}
A collection of $[0,\infty]$-valued random variables, $\{N(\alpha): \alpha \in A\}$, is stochastically bounded uniformly in $\alpha$ if
\bn
\lim_{M\rr \infty} \sup_{\alpha \in A} P(N(\alpha)\geq M)=0.
\en
\end{definition}
The following condition will be the goal of this section. Recall that $K_1$ was chosen to satisfy (\ref{K1}).
\paragraph{Property $(P_m)$.} For $m\in\mathds{Z}_{+}$ we let $(P_m)$ denote the following property:
\bn
&&\textrm{For any } n\in \mathds{N}, \xi, \eps_0 \in(0,1), K\in \mathds{N}^{\geq K_1} \textrm{ and } \beta\in [0,\frac{\eta}{\eta+1}], \textrm{ there is an } \nonumber \\
&&N_1(\omega) = N_1(m,n,\xi,\eps_0,K,\beta,\eta) \textrm{ in } \mathds{N} \textrm{ a.s. such that for all } N\geq N_1,   \nonumber \\
&& \textrm{if } (t,x)\in Z(N,n,K,\beta), \ t^{'}\leq T_K \textrm{ and } d((t,x),(t',x')) \leq 2^{-N}, \nonumber \\
&& \textrm{ then } |u(x',t')|\leq a_n^{-\eps_0}2^{-N\xi}\big[(a_n^{\al}\vee 2^{-N})^{\tilde{\gamma}_m-1}+a_n^{\beta}\mathds{1}_{\{m>0\}}\big].\nonumber \\
&&\textrm{Moreover } N_1 \textrm{ is stochastically bounded uniformly in } (n,\beta).
\en
The goal of this section is to prove the following proposition.
\begin{proposition} \label{Prop-Induction}
For any $m\leq\bar{m}+1$, $(P_m)$ holds.
\end{proposition}
\begin{remark}
We will prove Proposition \ref{Prop-Induction} by induction.
\end{remark}
We introduce the following theorem, that will help us to prove $(P_0)$.
\begin{theorem} \label{Theorem2.3MP-P0-ind}
Assume the same assumptions as in Theorem \ref{holder-conj} except now allow $\gamma\geq1-\eta/2$. For each $K\in\mathds{N}$ and $\xi\in(0,1)$ there is an $N_0(\xi,K,\omega)\in \mathds{N}$ a.s. such that for all natural numbers $N\geq N_0$ and all $(t,x)\in Z(N,K)$,
\bn \label{res-Thm-2.3mps}
d((t',x'),(t,x))\leq 2^{-N} \textrm{ and } t'\leq T_K \textrm{ implies } |u(t',x')-u(t,x)|\leq 2^{-N\xi}. \nonumber
\en
\end{theorem}
The proof of Theorem \ref{Theorem2.3MP-P0-ind} is given in Section \ref{Section-Proof2.3}. Theorem \ref{Theorem2.3MP-P0-ind} was proved in \cite{MP09} for the case of homogeneous white noise.
\paragraph{Proof of $(P_0)$}
The proof of $(P_0)$ is similar to the proof of $(P_0)$ in Section 5 of \cite{MP09}, just replace Theorem 2.3 in \cite{MP09} with Theorem \ref{Theorem2.3MP-P0-ind}. Exactly as in the proof of $(P_0)$ in Section 5 of \cite{MP09}, we get that $N_1=N_1(0,\xi,K)$. That is, $N_1$ does not depend on $(n,\beta)$.
\qed \\\\
To carry out the induction we first use $(P_m)$ to get a local modulus of continuity for $F_\dl$ (as in Section 5 of \cite{MP09}). \\\\
Recall the $F_\dl$ was given by (\ref{Fdl2}) where
\bd
D(r,y)=\sigma(r,y,u^1(r,y))-\sigma(r,y,u^2(r,y)).
\ed
From (\ref{Fdl2}) we get for $s\leq t \leq t'$ and $s'\leq t'$
\bn \label{FdlIncrement}
|F_\dl(s,t,x)-F_\dl(s',t',x')| &\leq & |F_\dl(s,t',x')-F_\dl(s',t',x')|+|F_\dl(s,t',x')-F_\dl(s,t,x)|  \\
&=&\bigg|\int_{(s-\dl)^{+}}^{(s'-\dl)_{+}}\int_{\re}G'_{t'-r}(y-x')D(r,y)W(dr,dy)\bigg|   \nonumber \\
&&+\bigg|\int_{0}^{(s-\dl)_{+}}\int_{\re}\big(G'_{t'-r}(y-x')-G'_{t-r}(y-x)\big)D(r,y)W(dr,dy)\bigg|. \nonumber
\en
From (\ref{FdlIncrement}) and (\ref{Dbound}) we realize that to get the bound on $|F_\dl(s,t,x)-F_\dl(s',t',x')|$ we may use the bounds on the following square functions
\bn \label{Qdef}
Q_{T,\dl}(s,s',t',x')&=&\int_{(s\wedge s' -\dl)^{+}}^{(s\vee s' -\dl)^{+}}\int_{\re}G'_{t'-r}(y-x')^{2}e^{2R_1|y|}|u(r,y)|^{2\gamma}\mu(dy)dr,  \nonumber \\  \nonumber \\
Q_{S,1,\dl,\nu_0}(s,t,x,t',x')&=&\int_{0}^{(s-\dl)^{+}}\int_{\re}\mathds{1}_{\{|x-y|>(t'-r)^{1/2-\nu_0}\vee 2|x'-x|\}}\nonumber \\  \nonumber \\
&&\times(G'_{t'-r}(x'-y)-G'_{t-r}(x-y))^2e^{2R_1|y|}|u(r,y)|^{2\gamma}\mu(dy)dr,  \nonumber \\  \nonumber \\
Q_{S,2,\dl,\nu_0}(s,t,x,t',x')&=&\int_{0}^{(s-\dl)^{+}}\int_{\re}\mathds{1}_{\{|x-y| \leq (t'-r)^{1/2-\nu_0}\vee 2|x'-x|\}}\\  \nonumber \\
&& \times(G'_{t'-r}(x'-y)-G'_{t-r}(x-y))^2e^{2R_1|y|}|u(r,y)|^{2\gamma}\mu(dy)dr, \nonumber
\en
for $\nu_0\in(0,1/2)$, $\dl\in (0,1]$ and $s\leq t \leq t', s'\leq t'$. \\\\
We will use the following lemmas to bound the terms in (\ref{Qdef}).
\begin{lemma} \label{LemmaBound-u}
Let $0\leq m \leq \bar{m}+1$ and assume $(P_m)$. For any $n,\xi,\eps_0,K$ and $\beta$ as in $(P_m)$, if \\ $\bar{d}_N=d((s,y),(t,x))\vee2^{-N}$ and $\sqrt{C_{\ref{LemmaBound-u}}(\omega)}=(4a_n^{-\eps_0}+2^{2N_1(\omega)}2Ke^K)$, then for any $N\in \mathds{N}$, on
\bn \label{rd1}
\{\omega:(t,x)\in Z(N,n,K,\beta), N>N_{1}(m,n,\xi,\eps_0,K,\beta)\},
\en
we have
\bn \label{rd4}
|u(s,y)|\leq \sqrt{C_{\ref{LemmaBound-u}}(\omega)}e^{|y-x|}\bar{d}_N^{\xi}\big[(a_n^{\al}\vee \bar{d}_N)^{\tilde{\gamma}_m-1}+\mathds{1}_{\{m>0\}}a_n^{\beta}\big], \ \forall s\leq T_K, \ y\in\re. \nonumber \\
\en
\end{lemma}
\begin{proof}
The proof of Lemma \ref{LemmaBound-u} is similar to the proof of Lemma 5.2 in \cite{MP09}. Assume $N,\omega,t,x$ are as in (\ref{rd1}). \\
Case 1. $d\equiv d((s,y),(t,x))\leq 2^{-N_1}$.
If $d>2^{-N}$ choose $N_1\leq N' <N$ so that $2^{-N'-1}\leq d\leq 2^{-N'}$, and if $d\leq 2^{-N}$ set $N'=N$. Then $(t,x)\in Z(N',n,K,\beta),\ d\leq 2^{-N'}\leq 2^{-N}\vee 2d \leq 2\bar{d}_N$. We get from $(P_m)$ for $s\leq T_K$,
\bn \label{rd2}
|u(s,y)| &\leq& a_n^{-\eps_0}2^{-N'\xi}\big[(a_n^{\al}\vee 2^{-N'})^{\tilde{\gamma}_m-1}+\mathds{1}_{\{m>0\}}a_n^{\beta}\big] \nonumber \\
&\leq & 4a_n^{-\eps_0}(\bar{d}_N)^{\xi}\big[(a_n^{\al}\vee \bar{d}_N)^{\tilde{\gamma}_m-1}+\mathds{1}_{\{m>0\}}a_n^{\beta}\big].
\en
Case 2.  $d> 2^{-N_1}$. Since $K \geq K_1$, we get for $s\leq T_k$
\bn \label{rd3}
|u(s,y)| \leq 2Ke^{|y|} &\leq& 2Ke^{|y|}(d 2^{N_1})^{\xi+\tilde{\gamma}_m-1} \nonumber \\
&\leq& 2Ke^{K}e^{|y-x|}2^{2N_1}(\bar{d}_N)^{\xi+\tilde{\gamma}_m-1}.
\en
From (\ref{rd2}) and (\ref{rd3}) we get (\ref{rd4}).
\end{proof}
\begin{remark} \label{Remark5.3}
If $m=0$ we may set $\eps_0=0$ in the above and $N_1$ does not depend on $(n,\eps_0,\beta)$ by the proof of $(P_0)$.
\end{remark}
\begin{lemma} \label{Lemma-Q1}
For all $K\in \mathds{N}^{\geq K_1}, R>2/\eta$ there exists $C_{\ref{Lemma-Q1}}(K,R,R_1,\nu_0,\eta) > 0$ and an $N_{\ref{Lemma-Q1}}=N_2(K,\omega) \in \mathds{N}$ a.s. such that for all $\nu_0,\nu_1 \in (1/R,1/2), \dl\in(0,1],\beta\in[0,\frac{\eta}{1+\eta}]$ and $N,n\in \mathds{N}$, for any $(t,x)\in \re_{+}\times \re$, on
\bd
\{\omega:(t,x)\in Z(N,n,K,\beta), N>N_{\ref{Lemma-Q1}}\},
\ed
\bd
Q_{S,1,\dl,\nu_0}(s,t,x,t',x')\leq C_{\ref{Lemma-Q1}}2^{4N_{\ref{Lemma-Q1}}} \big[d^{2-\nu_1}+(d\wedge\sqrt{\dl})^{2-\nu_1}\dl^{-2+\eta/2-\eps_0}(d\wedge1)^{4\gamma}\big], \ \ \forall s\leq t \leq t', \ x'\in \re.
\ed
Here $d=d((t',x'),(t,x))$.
\end{lemma}
\begin{proof}
The proof is almost similar to the proof of Lemma 5.4 in \cite{MP09}. The only difference is that we use Lemma \ref{Lemma-MP4.4}(b) instead of Lemma 4.4(b) from \cite{MP09}. Therefore, we get the exponent $-2+\eta/2-\eps_0$ instead of $-3/2$ for $\dl$.
\end{proof}
\begin{lemma} \label{Lemma-Q2}
Let $0\leq m \leq \bar{m}+1$ and assume $(P_m)$. For any $K\in \mathds{N}^{\geq K_1}, R>2/\eta, n\in \mathds{N}, \eps_0\in (0,1)$, and $\beta\in[0,\frac{\eta}{\eta+1}]$, there exists $C_{\ref{Lemma-Q2}}(K,R,R_1,\nu_0,\eta) > 0$ and an \\ $N_{\ref{Lemma-Q2}}=N_{\ref{Lemma-Q2}}(m,n,\eps_0,K,\beta,\eta,\omega) \in \mathds{N}$ a.s. such that for any $\nu_1 \in (1/R,1/2), \nu_0\in(0,\eta_1/32), \dl\in[a_n^{2\al},1], N\in \mathds{N}$ and $(t,x)\in \re_{+}\times \re$, on
\bd
\{\omega:(t,x)\in Z(N,n,K,\beta), N>N_{\ref{Lemma-Q2}}\},
\ed
\bn
Q_{S,2,\dl,\nu_0}(s,t,x,t',x')&\leq& C_{\ref{Lemma-Q2}}\big[a_n^{-2\eps_0}+2^{4N_{\ref{Lemma-Q2}}}\big] \big[d^{2-\nu_1}(\bar{\dl}_N^{(\gamma\tilde{\gamma}_m-2+\eta/2-\eps_0)\wedge 0}+a_n^{2\beta\gamma}\bar{\dl}_N^{\gamma-2+\eta/2-\eps_0}\big) \nonumber \\
&&+(d\wedge\sqrt{\dl})^{2-\nu_1}\dl^{-2+\eta/2-\eps_0}(\bar{d}_N^{2\gamma\tilde{\gamma}_m}+a_n^{2\beta\gamma}\bar{d}_N^{2\gamma}\big)\big], \nonumber \\
&& \ \forall s\leq t \leq t', \ |x'|\leq K+1.\nonumber
\en
Here $d=d((t',x'),(t,x)), \bar{d}_N=d\vee2^{-N}$ and $\bar{\dl}_N=\dl\vee\bar{d}_N^2$. Moreover, $N_{\ref{Lemma-Q2}}$ is stochastically bounded uniformly in $(n,\beta)$.
\end{lemma}
\begin{proof}
The proof is similar to the proof of Lemma 5.5 in \cite{MP09}. First we use Lemma \ref{LemmaBound-u} to bound $|u(r,y)|$ in the integral defining $Q_{S,2,\dl,\nu_0}$. Since we may assume $s\geq \dl$, and from the assumptions we know that $\dl \geq a_n^{2\al}$, therefore we have $d((r,y),(t,x)))\geq a_n^{\al}$. Hence, when we use (\ref{rd4}) to bound $|u(r,y)|$, we may drop the max with $a_n^{\al}$ in (\ref{rd4}). Then we proceed as in Lemma 5.5 in \cite{MP09}.  The only difference is that we use Lemma \ref{Lemma-MP4.4}(a) instead of Lemma 4.4(a) from \cite{MP09}. Therefore, we get the power $-2+\eta/2-\eps_0$ instead of $-3/2$ for $\dl$.
\end{proof}

\begin{lemma} \label{Lemma-QT}
Let $0\leq m \leq \bar{m}+1$ and assume $(P_m)$. For any $K\in \mathds{N}^{\geq K_1}$, $R>2/\eta$, $n\in\mathds{N}, \eps_0\in(0,1)$, and $\beta\in[0,\frac{\eta}{\eta+1}]$, there is a $C_{\ref{Lemma-QT}}(K,R_1,\mu(\re),\eta)>0$ and \\ $N_{\ref{Lemma-QT}}=N_{\ref{Lemma-QT}}(m,n,R,\eps_0,K,\beta)(\omega)\in \mathds{N}$ a.s. such that for any $\nu_1\in(1/R,1/2),\dl\in[a_n^{2\al},1], N\in \mathds{N}$ and $(t,x)\in \re_{+}\times \re$, on
\bn \label{QTDL1}
\{\omega:(t,x)\in Z(N,n,K,\beta), N>N_{\ref{Lemma-QT}}\},
\en
\bn \label{QTDL}
Q_{T,\dl}(s,s',t',x')&\leq& C_{\ref{Lemma-QT}}\big[a_n^{-2\eps_0}+2^{4N_{\ref{Lemma-Q2}}}\big]|s-s'|^{1-\nu_1/2} \big[\bar{\dl}_N^{(\gamma\tilde{\gamma}_m-2+\eta/2-\eps_0)\wedge 0}+a_n^{2\beta\gamma}\bar{\dl}_N^{\gamma-2+\eta/2-\eps_0} \nonumber \\
&&+\mathds{1}_{\{\dl<\bar{d}_N^2\}}\dl^{-2+\eta/2-\eps_0}(\bar{d}_N^{2\gamma\tilde{\gamma}_m}+a_n^{2\beta\gamma}\bar{d}_N^{2\gamma}\big)\big], \nonumber \\
&& \ \forall s\leq t \leq t', \ s'\leq t' \leq T_K,  \ |x'|\leq K+1.
\en
Here $d=d((t',x'),(t,x)), \bar{d}_N=d\vee2^{-N}$ and $\bar{\dl}_N=\dl\vee\bar{d}_N^2$. Moreover, $N_{\ref{Lemma-QT}}$ is stochastically bounded uniformly in $(n,\beta)$.
\end{lemma}
The proof of Lemma \ref{Lemma-QT} follows the same lines as the proof of Lemma 5.6 in \cite{MP09}. \\\\
\begin{proof}
Let $\xi=1-((2R)^{-1}\wedge \frac{\eta^2}{2(\eta+1)})$ and define $N_{\ref{Lemma-QT}}=N_1(m,n,\xi,\eps_0,K,\beta)$ so that $N_{\ref{Lemma-QT}}$ is stochastically bounded uniformly in $(n,\beta)$, immediately from $(P_m)$. Assume that $s\vee s'\equiv \bar{s} \geq \dl$, otherwise, $Q_{T,\dl}(s,s',t',x')\equiv 0$. Denote by $\underline{s}=s\wedge s'$. We use Lemma \ref{LemmaBound-u} to bound $|u(r,y)|$ in the integrand of $Q_{T,\dl}$ and the maximum with $a_n^{\al}$ can be ignored since $a_n^{\al} \leq \sqrt{t'-r}$ in the calculations below. We argue as in the proof of Lemma 5.6 in \cite{MP09}, that for $\omega$ as in (\ref{QTDL1}) and $s,t,s',t',x'$ as in (\ref{QTDL}) we have
\bn \label{ds0}
Q_{T,\dl}(s,s',t',x') &\leq& C_{\ref{LemmaBound-u}}\int_{(\underline{s}-\dl)^{+}}^{\bar{s}-\dl}\int_{\re}G'_{t'-r}(y-x')^2e^{2R_1K}e^{2(R_1+1)|x-y|}
[2^{-N}\vee(\sqrt{t'-r}+|y-x|)]^{2\gamma \xi} \nonumber \\
&&  \times \big[2^{-N}\vee(\sqrt{t'-r}+|y-x|)^{\tilde{\gamma}_{m}-1}+a_n^{\beta}\big]^{2\gamma}\mu(dy)dr.
\en
Note that
\bn \label{ds1}
2^{-N}\vee(\sqrt{t'-r}+|y-x|)&\leq& \big(2^{-N}\vee|x-x'|\big) + \sqrt{t'-r} +|y-x'| \nonumber \\
&\leq& \bar{d}_N + \sqrt{t'-r} +|y-x'|,
\en
and
\bn \label{ds2}
e^{(2R_1+1)|y-x|}\leq C(K,R_1)e^{(2R_1+1)|y-x'|}.
\en
Apply Lemma \ref{Lem4.2MP}, (\ref{ds1}) and (\ref{ds2}) to (\ref{ds0}) to get
\bn \label{ds3}
&&Q_{T,\dl}(s,s',t',x')  \nonumber \\
&&\leq C_{\ref{LemmaBound-u}} C(K,R_1)\int_{(\underline{s}-\dl)^{+}}^{\bar{s}-\dl}\int_{\re}(t'-r)^{-1}G_{2(t'-r)}(y-x')^2e^{2(R_1+1)|y-x'|}
[\bar{d}_N^{2\gamma\xi}+(t'-r)^{\gamma\xi}+|y-x'|^{2\gamma \xi}] \nonumber \\
&&\quad \times \big[\bar{d}_N^{2\gamma(\tilde{\gamma}_{m}-1)}+(t'-r)^{\gamma(\tilde{\gamma}_{m}-1)}+|y-x'|^{2\gamma(\tilde{\gamma}_{m}-1)}+a_n^{2\beta\gamma}\big]\mu(dy)dr.
\en
Apply Lemma \ref{condmu}(d) with $a=\gamma\xi$ and $a=\gamma(\tilde \gamma_m-1)$ to (\ref{ds3}) to get
\bn \label{ds6}
&&Q_{T,\dl}(s,s',t',x')  \nonumber \\ \nonumber \\
&&\leq C(K,R_1,\eta,K)\int_{(\underline{s}-\dl)^{+}}^{\bar{s}-\dl}(t'-r)^{-2+\eta/2-\eps_0}
[\bar{d}_N^{2\gamma\xi}+(t'-r)^{\gamma\xi}] \nonumber \\ \nonumber \\
&&\quad\times \big[\bar{d}_N^{2\gamma(\tilde{\gamma}_{m}-1)}+(t'-r)^{\gamma(\tilde{\gamma}_{m}-1)}+a_n^{2\beta\gamma}\big]dr \nonumber \\ \nonumber \\
&&\leq C(K,R_1,\eta)\int_{(\underline{s}-\dl)^{+}}^{\bar{s}-\dl}\mathds{1}_{\{r\leq t'-\bar{d}_N^2\}}(t'-r)^{-2+\eta/2-\eps_0}
\big[(t'-r)^{\gamma(\tilde{\gamma}_{m}+\xi-1)}+a_n^{2\beta\gamma}(t'-r)^{\gamma\xi}\big]dr \nonumber \\  \nonumber \\
&&\quad+ C(K,R_1,\eta)\int_{(\underline{s}-\dl)^{+}}^{\bar{s}-\dl}\mathds{1}_{\{r> t'-\bar{d}_N^2\}}(t'-r)^{-2+\eta/2-\eps_0}dr
[\bar{d}_N^{2\gamma(\tilde{\gamma}_{m}+\xi-1)}+a_n^{2\beta\gamma}\bar{d}_N^{2\gamma\xi}] \nonumber \\ \nonumber \\
&&=: C(K,R_1,\eta)(J_1+J_2).
\en
Note that
\bn \label{ds7}
&&\int_{(\underline{s}-\dl)^{+}}^{\bar{s}-\dl}\mathds{1}_{\{r> t'-\bar{d}_N^2\}}(t'-r)^{-2+\eta/2-\eps_0}dr \nonumber \\
&&\leq
\mathds{1}_{\{\dl<\bar{d}_N^2\}}[(t'-\bar{s}+\dl)^{-2+\eta/2-\eps_0}|s'-s|\wedge \frac{1}{1-\eta/2+\eps_0}(t'-\bar{s}+\dl)^{-1+\eta/2-\eps_0}]  \nonumber \\
&&\leq  \frac{1}{1-\eta/2-\eps_0}\mathds{1}_{\{\dl<\bar{d}_N^2\}}\dl^{-2+\eta/2-\eps_0}(|s'-s|\wedge \dl).
\en
From (\ref{ds6}) and (\ref{ds7}) we get
\bn \label{ds8}
J_2 &\leq & (K,R_1,\eta)\mathds{1}_{\{\dl<\bar{d}_N^2\}}\dl^{-2+\eta/2-\eps_0}(|s'-s|\wedge \dl)\bar{d}_N^{(-2\gamma(1-\xi))}[\bar{d}_N^{2\gamma\tilde{\gamma}_{m}}+a_n^{2\beta\gamma}\bar{d}_N^{2\gamma}]\nonumber \\
&\leq & (K,R_1,\eta)\mathds{1}_{\{\dl<\bar{d}_N^2\}}\dl^{-2+\eta/2-\eps_0}(|s'-s|\wedge \dl)^{1-\nu_1/2}[\bar{d}_N^{2\gamma\tilde{\gamma}_{m}}+a_n^{2\beta\gamma}\bar{d}_N^{2\gamma}],
\en
where we have used the fact that $\gamma(1-\xi)\leq 1-\xi \leq (2R)^{-1} \leq \nu_1/2$. \medskip \\
For $J_1$, let $p=\gamma(\tilde{\gamma}_m+\xi-1)-(2-\eta/2+\eps_0)$ or $\gamma\xi-(2-\eta/2+\eps_0)$ for $0\leq m-1 \leq \bar{m}$. Recall that $\tilde{\gamma}_m \in [1,1+\eta]$, and $\eta\in(0,1)$. By our choice of $\xi$ and from the bounds on $\gamma,\tilde{\gamma}_m,\eta$ we get
\bn \label{b1}
p&\geq& \gamma\xi-(2-\eta/2+\eps_0) \nonumber \\
&\geq & -2+\frac{\eta}{2}-\eps_0+\gamma\big(1-\frac{\eta^2}{2(\eta+1)}\big) \nonumber \\
&\geq & -2+\frac{\eta}{2}-\eps_0+\gamma-\frac{\eta^2}{2(\eta+1)} \nonumber \\
&> & -1+10\eps_1,
\en
where we have used (\ref{eps1}) in the last inequality. From the same bounds on $\gamma,\tilde{\gamma}_m,\eta,\xi,$ we also get
\bn  \label{b2}
p &\leq& \gamma(\tilde{\gamma}_m+\xi-1)-(2-\eta/2+\eps_0)  \nonumber \\
&\leq & \gamma(1+\eta) -2+\frac{\eta}{2}-\eps_0   \nonumber \\
&\leq & -1+\frac{3}{2}\eta-\eps_0 \nonumber  \\
&< & \frac{1}{2}.
\en
From (\ref{b1}) and (\ref{b2}) we get
\bd
p\in \bigg(10\eps_1-1,\frac{1}{2}\bigg).
\ed
Denote by $p'=p\wedge 0$ and let $0\leq \eps \leq -p'$. Since $t'\leq t_0$ and $p'\in[0,1-10\eps_1)$, we get
\bn \label{ds9}
I(p)&:=&\int_{(\underline{s}-\dl)^{+}}^{\bar{s}-\dl}\mathds{1}_{\{r\leq t'-\bar{d}_N^2\}}(t'-r)^p dr \nonumber \\
&= &\int_{(\underline{s}-\dl)^{+}}^{\bar{s}-\dl}\mathds{1}_{\{r\leq t'-\bar{d}_N^2\}}(t'-r)^{p\wedge0} (t'-r)^{p\vee 0} dr \nonumber \\
&\leq &\int_{(\underline{s}-\dl)^{+}}^{\bar{s}-\dl}\mathds{1}_{\{r\leq t'-\bar{d}_N^2\}}(t'-r)^{p'} (1+\sqrt{t'}) dr \nonumber \\
&\leq &C(K)\int_{(\underline{s}-\dl)^{+}}^{\bar{s}-\dl}\mathds{1}_{\{r\leq t'-\bar{d}_N^2\}}(t'-r)^{p'}  dr \nonumber \\
&\leq &C(K)\min\bigg\{|s'-s|(\bar{d}_N^{2p'}\wedge(t'-\bar s +\dl)^{p'}),\int_{0}^{|s'-s|}u^{p'}du\bigg\} \nonumber \\
&\leq &C(K,\eps_1)\min\bigg\{|s'-s|(\bar{d}_N^{2p'}\wedge(t'-\bar s +\dl)^{p'}),|s'-s|^{p'+1}\bigg\} \nonumber \\
&\leq &C(K,\eps_1)\min\bigg\{|s'-s|\bar{\dl}_N^{p'},|s'-s|^{p'+1}\bigg\} \nonumber \\
&\leq &C(K,\eps_1)|s'-s|^{p'+1}\min\bigg\{\bigg(\frac{|s'-s|}{\bar{\dl}_N}\bigg)^{-p'},1\bigg\} \nonumber \\
&\leq &C(K,\eps_1)|s'-s|^{p'+1}\bigg(\frac{|s'-s|}{\bar{\dl}_N}\bigg)^{-p'-\eps}\nonumber \\
&\leq &C(K,\eps_1)|s'-s|^{1-\eps}(\bar{\dl}_N)^{\eps+p'}.
\en
Define $q=p+\gamma(1-\xi)$, so that $q=\gamma\tilde{\gamma}_m-2+\eta/2-\eps_0$ or $q=\gamma-(2-\eta/2+\eps_0)$. We distinct between two cases as follows. \\\\
The first case is $q\leq 0$. Then $p'=p<0$. Choose $\eps=\gamma(1-\xi)\leq(2R)^{-1}<\nu_1/2$, then $\eps+p'=q\leq 0$. Thus we can use (\ref{ds9}) with this $\eps$ to get
\bn \label{ds10}
I(p)&\leq & C(K,\eps_1)|s'-s|^{1-\eps}(\bar{\dl}_N)^{q}\nonumber \\
&\leq&  C(K,\eps_1)|s'-s|^{1-\nu_1/2}(\bar{\dl}_N)^{q}.
\en
The second case is $q>0$. Then $p'=(q-\gamma(1-\xi))\wedge 0\geq -\gamma(1-\xi)$. Choose $\eps=-p'\leq \gamma(1-\xi)\leq(2R)^{-1}<\nu_1/2$ and again we can apply (\ref{ds9}) to get
\bn \label{ds11}
I(p)&\leq & C(K,\eps_1)|s'-s|^{1-\eps}\nonumber \\
&\leq&  C(K,\eps_1)|s'-s|^{1-\nu_1/2}.
\en
From (\ref{ds10}) and (\ref{ds11}) we get
\bn \label{ds12}
I(p)&\leq&  C(K,\eps_1)|s'-s|^{1-\nu_1/2}(\bar{\dl}_N)^{q\wedge 0}.
\en
From (\ref{ds6}) and (\ref{ds12}) we get
\bn\label{ds13}
J_1\leq C(K,\eps_1)|s'-s|^{1-\nu_1/2}\big[\bar{\dl}_N^{(\gamma\tilde{\gamma}_m-2+\eta/2-\eps_0)\wedge 0} + a_n^{2\beta\gamma} \bar{\dl}_N^{(\gamma -2+\eta/2-\eps_0)\wedge 0}\big].
\en
From (\ref{ds6}), (\ref{ds8}) and (\ref{ds13}), (\ref{QTDL}) follows.
\end{proof}
\paragraph{Notation:}
Denote
\bd
d((s,t,x),(s',t',x'))=\sqrt{|s-s'|}+\sqrt{|t-t'|}+|x-x'|,
\ed
and
\bn \label{Delta-u1-tag}
&&\bar{\Delta}_{u_1'}(m,n,\alpha,\eps_0,2^{-N})  \\
&&= a_n^{-2\eps_0}\big[a_n^{-\alpha(1-\eta/4)}2^{-N\gamma \tilde{\gamma}_m}+(a_n^{\alpha/2}\vee 2^{-N})^{(\gamma_{m}\gamma-2+\eta/2)\wedge 0}+a_n^{-\alpha(1-\eta/4)+\beta\gamma}(a_n^{\alpha/2}\vee 2^{-N})^\gamma\big]. \nonumber
\en
\begin{proposition} \label{probIncFdl}
Let $0 \leq m \leq \bar{m}+1$ and assume $(P_m)$. For any $n\in \mathds{N},\nu_1\in(0,\eta/2), \eps_0 \in (0,1), \\ K\in \mathds{N}^{\geq K_1}, \alpha \in [0,2\al]$, and $\beta\in [0,\frac{\eta}{\eta+1}]$, there exists $N_{\ref{probIncFdl}}=N_{\ref{probIncFdl} }(m,n,\nu_1,\eps_0,K,\alpha, \beta, \eta, \mu(\re))(\omega)$  in $\mathds{N}^{\geq2}$  a.s. such that for all $N\geq N_{\ref{probIncFdl}}, (t,x)\in Z(N,n,K,\beta),$ $\ s\leq t, \ s'\leq t' \leq  T_K$, if $d((s,t,x),(s',t',x'))\leq 2^{-N}$,
then,
\bn
|F_{a_n^{\alpha}}(s,t,x)-F_{a_n^{\alpha}}(s',t',x')|\leq 2^{-86} d((s,t,x),(s',t',x'))^{1-\nu_1}\bar{\Delta}_{u_1'}(m,n,\alpha,\eps_0,2^{-N}). \nonumber
\en
Moreover $N_{\ref{probIncFdl}}$ is stochastically bounded uniformly in $(n,\alpha,\beta)$.
\end{proposition}
\begin{proof}
The proof of Proposition \ref{probIncFdl} is similar to the proof of Proposition 5.8 in \cite{MP09}. Let $R=33/(\nu_1\eta)$ and choose $\nu_0\in(R^{-1},\nu_1/32)$. Let $d=d((t,x),(t',x'))$ and $\bar d_N=d\vee2^{-N}$, and
\bn
Q_{a_n^{\alpha}}(s,t,x,s',x',t')  =Q_{T,a_n^{\alpha}}(s,s',t',x')+\sum_{i=1}^{2}Q_{S,i,a_n^{\alpha},\nu_0}(s,t,x,t',x').
\en
We use Lemmas \ref{Lemma-Q1}-\ref{Lemma-QT} to bound $Q_{a_n^{\alpha}}(s,t,x,s',x',t')$. We get that there exists $C(K,\nu_1)$ and
$N_2(m,n,\nu_1,\eps_0,K,\beta,\eta)$ stochastically bounded uniformly in $(n,\beta)$, such that for all $N\in \mathds{N}$ and $(t,x)$, on \bn
\{\omega:(t,x)\in Z(N,n,K+1,\beta), N\geq N_2\},  \nonumber
\en
\bn
&&R_0^{\gamma}Q_{a_n^{\alpha}}(s,t,x,s',x',t') ^{1/2} \nonumber \\
&&\leq C(K,\nu_1)\big[a_n^{-\eps_0}+2^{2N_2}\big] (d+\sqrt{|s'-s|})^{1-\nu_1/2}\big\{(a_n^{\alpha/2}\vee\bar d_N)^{(\gamma\tilde{\gamma}_m-2+\eta/2-\eps_0)\wedge 0}+a_n^{\beta\gamma}(a_n^{\alpha/2}\vee\bar d_N)^{\gamma-2+\eta/2-\eps_0}\big) \nonumber \\
&&\quad +a_n^{-1+\eta/4-\eps_0/2}(\bar{d}_N^{\gamma\tilde{\gamma}_m}+a_n^{\beta\gamma}\bar{d}_N^{\gamma}\big)\big\},  \ \forall s\leq t \leq t' , \ s'\leq t'\leq T_K, \ |x'|\leq K+2. \nonumber
\en
The rest of the proof follows the proof of Proposition 5.8 in \cite{MP09} after (5.37) there.
Briefly, we use Dubins-Schwarz theorem to bound $|F_{a_n^{\alpha}}(s,t,x)-F_{a_n^{\alpha}}(s',t',x')|$.
We get that the power of $\alpha_n$ in $\bar{\Delta}_{u_1'}(m,n,\alpha,\eps_0,2^{-N})$ changes from $-3/4\alpha$ in \cite{MP09} to $-(1-\eta/4)\alpha$ and the power of $(a_n^{\alpha/2}\vee 2^{-N})$ changes to $(\gamma\gamma_m-2+\eta/2)\wedge 0$ instead of $(\gamma_{m+1}-2)\wedge 0$ in \cite{MP09}.
\end{proof}
\\\\
From Remark \ref{remark32} we have $F_{\dl}(t,t,x)=-u'_{1,\dl}(t,x)$. Hence the following corollary follows immediately from Proposition \ref{probIncFdl}.
\begin{corollary} \label{corollary-u-tag-reg}
Let $0\leq m \leq \bar{m} +1$ and assume $(P_m)$. Let $n,\nu_1,\eps_0,K,\alpha$ and $\beta$ as in Proposition \ref{probIncFdl}. For all $N\geq N_{\ref{probIncFdl}}$, $(t,x)\in Z(N,n,K,\beta)$ and $t'\leq T_K$,  if $d((t,x),(t',x'))\leq 2^{-N}$ then
\bn
|u'_{1,a_n^\alpha}(t,x)-u'_{1,a_n^\alpha}(t',x')|\leq 2^{-85}d((t,x),(t',x'))^{1-\nu_1}\bar{\Delta}_{u_1'}(m,n,\alpha,\eps_0,2^{-N}).
\en
\end{corollary}
We would like to bound $|u'_{1,\dl}-u'_{1,a^{2\al}_n}|$. Let $\dl\geq a_n^{2\al}$ and $s=t-\dl+a_n^{2\al}$. It is easy to check that
\bn \label{u-tag-F-an}
u'_{1,\dl}(t,x)=-F_{a^{2\al}_n}(t-\dl+a_n^{2\al},t,x).
\en
The following lemmas will be used to bound $|F_{a^{2\al}_n}(s,t,x)-F_{a^{2\al}_n}(t,t,x)|$, where the hypothesis $\sqrt{t-s}\leq 2^{-N}$  from Proposition \ref{probIncFdl} is weakened considerably (see Proposition \ref{Prop5.11}).
\begin{lemma} \label{Lemma-QTan}
Let $0\leq m \leq \bar{m}+1$ and assume $(P_m)$. For any $K\in \mathds{N}^{\geq K_1}$, $R>2/\eta$, $n\in\mathds{N}, \eps_0\in(0,1)$, and $\beta\in[0,\frac{\eta}{\eta+1}]$, there is a $C_{\ref{Lemma-QTan}}(K,R_1,\mu(\re),\eta)>0$ and \\ $N_{\ref{Lemma-QTan}}=N_{\ref{Lemma-QTan}}(m,n,R,\eps_0,K,\beta)(\omega)\in \mathds{N}$ a.s. such that for any $\nu_1\in(1/R,\eta/2), N\in \mathds{N}$ and $(t,x)\in \re_{+}\times \re$, on
\bd \label{Z-SET2}
\{\omega:(t,x)\in Z(N,n,K,\beta), N>N_{\ref{Lemma-QTan}}\},
\ed
\bn \label{QTan}
Q_{T,a_n^{2\al}}(s,t,t,x)&\leq& C_{\ref{Lemma-QTan}}a_n^{-\eps_0}\big[a_n^{-2\eps_0}+2^{4N_{\ref{Lemma-QTan}}}\big] \big\{|t-s|^{1-\nu_1/4} \big[((t-s)\vee a_n^{2\al})^{\gamma \tilde{\gamma}_m-2+\eta/2-\eps_0}\nonumber \\
&&+a_n^{2\beta\gamma}((t-s)\vee a_n^{2\al})^{\gamma-2+\eta/2-\eps_0}\big]\nonumber \\
&&+\mathds{1}_{\{a_n^{2\al}<2^{-2N}\}}((t-s)\wedge a_n^{2\al}) a_n^{-2+\eta/2-\eps_0}2^{N\nu_1/2}(2^{-2N\gamma\tilde{\gamma}_m}+a_n^{2\beta\gamma}2^{-2N\gamma}\big)\big\}, \nonumber \\
&& \ \forall s\leq t.
\en
Moreover $N_{\ref{Lemma-QTan}}$ is stochastically bounded uniformly in $(n,\beta)$.
\end{lemma}
\begin{proof}
The proof follows the same lines as the proof of Lemma 5.10 in \cite{MP09}. Fix $\theta\in (0,\eta/2-\eps_0)$ such that $\gamma-\theta>1-\eta/2-\eps_0$.
Let $\xi=1-((4\gamma R)^{-1}\wedge \theta)$ and define $N_{\ref{Lemma-QTan}}=N_1(m,n,\xi(R),\eps_0,K,\beta)$. Then, we get from $(P_m)$ that $N_{\ref{Lemma-QTan}}$ is stochastically bounded uniformly in $(n,\beta)$. Assume that $t\geq a_n^{2\al}$, otherwise $Q_{T,a_n^{2\al}}(s,t,t,x)\equiv 0$. Now repeat the same steps as in Lemma \ref{Lemma-QT}, here we have $x'=x$, $s'=t'=t$ and $\dl=a_n^{2\al}$.
We get that
\bn \label{sa2}
J_1&=&\int_{(s-a_n^{2\al})^{+}}^{t-a_n^{2\al}}\mathds{1}_{\{r< t-2^{-2N}\}}\big[(t-r)^{\gamma(\tilde{\gamma}_{m}+\xi-1)-2+\eta/2-\eps_0}
+a_n^{2\beta\gamma}(t-r)^{\gamma\xi-2+\eta/2-\eps_0}\big]dr,  \nonumber  \\
J_2&=&\int_{(s-a_n^{2\al})^{+}}^{t-a_n^{2\al}}\mathds{1}_{\{r\geq t-2^{-2N}\}}(t-r)^{-2+\eta/2-\eps_0}dr2^{-2N\gamma\xi}[2^{-2N\gamma(\tilde{\gamma}_m-1)}+a_n^{2\beta\gamma}],
\en
and
\be \label{sa2.2}
Q_{T,a_n}(s,t,t,x)\leq C_{\ref{LemmaBound-u}} C(K,R_1,\mu(\re),\eta)[J_1+J_2].
\ee
Repeat the same steps as in (\ref{ds7}) and (\ref{ds8}) with $d_N=2^{-N}, \dl=a_n^{2\al}, s'=t$, and use the fact that \\ $\gamma(1-\xi(R))\leq (4R)^{-1}\wedge \theta\leq \nu_1/4$ to get,
\bn \label{sa3}
J_2 &\leq & 2\cdot\mathds{1}_{\{a_n^{2\al}<2^{-2N}\}}a_n^{-2+\eta/2-\eps_0}(|t-s|\wedge a_n^{2\al})2^{2N\gamma(1-\xi)}[2^{-2N\gamma\tilde{\gamma}_{m}}+a_n^{2\beta\gamma}2^{-2N\gamma}] \nonumber\\
&\leq& 2\cdot\mathds{1}_{\{a_n^{2\al}<2^{-2N}\}}a_n^{-2+\eta/2-\eps_0}(|t-s|\wedge a_n^{2\al})2^{N\nu_1/2}[2^{-2N\gamma\tilde{\gamma}_{m}}+a_n^{2\beta\gamma}2^{-2N\gamma}]. \nonumber\\
\en
For $J_1$, let $p=\gamma(\tilde{\gamma}_m+\xi-1)-2+\eta/2-\eps_0$ or $p=\gamma\xi-2+\eta/2-\eps_0$. Recall that $\eta\in(0,1),\gamma \in (0,1-\eta/2-\eps_0+\theta)$ and $R>2/\eta$. Therefore $\xi=1-((4\gamma R)^{-1}\wedge \theta) \geq 1- \theta $ and
\bn
p\in \bigg( \theta(\eta/2-\eps_0-\theta)-1, \frac{1}{2}\bigg).
\en
If we consider the case of $p\geq 0$ and $p<0$ separately we get, as in the proof of Lemma 5.10 in \cite{MP09}, that
\bn \label{sa4}
\int_{(s-a_n^{2\al})^{+}}^{(t-a_n^{2\al})}(t-r)^{p}dr \leq C(\eta)(t-s)\big((t-s)\vee a_n^{2\al})^p.
\en
Use (\ref{sa4}) and the fact that $\gamma(1-\xi(R))\leq(4R)^{-1}\leq \nu_1/4$ to get
\bn  \label{sa5}
J_1&\leq& C(\eta)a_n^{-\eps_0}(t-s)\big[((t-s)\vee a_n^{2\al})^{\gamma(\tilde{\gamma}_{m}+\xi-1)-2+\eta/2}
+a_n^{2\beta\gamma}((t-s)\vee a_n^{2\al})^{\gamma\xi-2+\eta/2-\eps_0}\big] \nonumber  \\
&\leq& C(\eta)a_n^{-2\eps_0}(t-s)^{1-\gamma(1-\xi)}\big[((t-s)\vee a_n^{2\al})^{\gamma\tilde{\gamma}_{m}-2+\eta/2-\eps_0}
+a_n^{2\beta\gamma}((t-s)\vee a_n^{2\al})^{\gamma-2+\eta/2-\eps_0}\big]  \nonumber  \\
&\leq& C(\eta,K)a_n^{-\eps_0}(t-s)^{1-\nu_1/4}\big[((t-s)\vee a_n^{2\al})^{\gamma\tilde{\gamma}_{m}-2+\eta/2-\eps_0}
+a_n^{2\beta\gamma}((t-s)\vee a_n^{2\al})^{\gamma-2+\eta/2-\eps_0}\big].   \nonumber  \\
\en
From (\ref{sa2}), (\ref{sa2.2}), (\ref{sa3}) and (\ref{sa5}) we get (\ref{QTan}).
\end{proof}
\begin{proposition}\label{Prop5.11}
Let $0\leq m\leq \bar{m}+1$ and assume $(P_m)$. For any $n\in \mathds{N}, \nu_1\in(0,\eta/2), \eps_0\in (0,1),K\in \mathds{N}^{\geq K_1}$ and $\beta\in[0,\frac{\eta}{\eta+1}]$, there is an $N_{\ref{Prop5.11}}=N_{\ref{Prop5.11}}(m.n,\nu_1,\eps_0,K,\beta,\eta, \mu(\re))(\omega)\in \mathds{N}$ a.s. such that for all $N\geq N_{\ref{Prop5.11}}, (t,x)\in Z(N,n,K,\beta), s\leq t$ and $\sqrt{t-s}\leq N^{-4/\nu_1}$ implies that
\bn \label{5.11res}
|F_{a_n^{2\al}}(s,t,x)-F_{a_n^{2\al}}(t,t,x)|&\leq& 2^{-81}a_n^{-3\eps_0}\bigg\{2^{-N(1-\nu_1)}(a_n^{\al}\vee2^{-N})^{(\gamma\gamma_m-2+\eta/2)\wedge 0}  \nonumber  \\
&&+2^{N\nu_1}a_n^{-1+\eta/4+\al}\bigg(\frac{2^{-N}}{a_n^{3\al-\al\eta/2-1+\eta/4}}+1\bigg)\big(2^{-N\gamma\tilde{\gamma}_m}+a_n^{\beta\gamma}
(a_n^{\al}\vee2^{-N})^{\gamma}\big)\nonumber  \\
&& +(t-s)^{(1-\nu_1)/2}\big((\sqrt{t-s}\vee a_n^{\al})^{\gamma\tilde{\gamma}_m-2+\eta/2}\nonumber  \\
&&+a_n^{\beta\gamma}(\sqrt{t-s}\vee a_n^{\al})^{\gamma-2+\eta/2}\big)\bigg\}.
\en
Moreover $N_{\ref{Prop5.11}}$ is stochastically bounded, uniformly in $(n,\beta)$.
\end{proposition}
\begin{proof}
The proof is similar to the proof of Proposition 5.11 in \cite{MP09}. First we bound $R_0^\gamma Q_{T,a^{2\al}_n}(s,t,t,x)^{1/2}$ as follows. Let $N_2(m,n,\nu_1,\eps_0,K,\beta)=\frac{8}{\nu_1}[N_{\ref{Lemma-QTan}}+\tilde N_0(K)]$, where $\tilde N_0(K)\in \mathds{N}$ is large enough such that
\bn
C_{\ref{Lemma-QTan}}R_0^\gamma\big[a_n^{-2\eps_0}+2^{4N_{\ref{Lemma-QTan}}}\big] 2^{-\nu_1 N_2/4}&\leq & C_{\ref{Lemma-QTan}}R_0^\gamma \big[a_n^{-2\eps_0}+2^{4N_{\ref{Lemma-QTan}}}\big] 2^{-2N_{\ref{Lemma-QTan}}-2\tilde N_0(K)}  \nonumber \\
&\leq & 2^{-100}a_n^{-\eps_0}.
\en
On $$\{\omega : (t,x)\in Z(N,n,K,\beta), N\geq N_{\ref{Lemma-QTan}}(m,n,2/\nu_1,\eps_0,K,\beta)\}, $$
we have
\bn
R_0^\gamma Q_{T,a^{2\al}_n}(s,t,t,x)^{1/2} &\leq& (\sqrt{t-s})^{1-\nu_1/2}\Delta_1(m,n,\sqrt{t-s}\vee a_n^{\al})+2^{N\nu_1/2}\Delta_2(m,n,2^{-N}), \nonumber \\  &&  s\leq t, \ \sqrt{t-s}\leq 2^{-N_2},
\en
where
\bn
\Delta_1(m,n,\sqrt{t-s}\vee a_n^{\al})&:=&2^{-100}a_n^{-3\eps_0}\big\{(\sqrt{t-s}\vee a_n^{\al})^{\gamma\tilde{\gamma}_m-2+\eta/2}+a_n^{\beta\gamma}(\sqrt{t-s}\vee a_n^{\al})^{\gamma-2+\eta/2}\big\}, \nonumber  \\ \nonumber  \\
\Delta_2(m,n,2^{-N})&:=&2^{-100}a_n^{-3\eps_0-1+\eta/4+\al}\big(2^{-N\gamma\tilde{\gamma}_m}+a_n^{\beta\gamma}
2^{-N\gamma}\big). \nonumber
\en
The only difference from the proof in \cite{MP09} is that we use Lemma \ref{Lemma-QTan} instead of Lemma 5.10 in \cite{MP09}. This gives the values $\Delta_1,\Delta_2$ above.
The rest of the proof is similar to the proof of Proposition 5.11 in \cite{MP09}. We use Dubins-Schwarz theorem and
Proposition \ref{probIncFdl} with
\bn
&&\bar{\Delta}_{u_1'}(m,n,2\al,\eps_0,2^{-(N-1)}) \nonumber \\
&&= a_n^{-3\eps_0}\big[a_n^{-2\al(1-\eta/4)}2^{-N\gamma \tilde{\gamma}_m}+(a_n^{\al}\vee 2^{-N})^{(\tilde{\gamma}_{m}\gamma-2+\eta/2)\wedge 0}+a_n^{-2\al(1-\eta/4)+\beta\gamma}(a_n^{\al}\vee 2^{-N})^\gamma\big]. \nonumber
\en
to get
\bn \label{nb1}
&&|F_{a_n^{2\al}}(s,t,x)-F_{a_n^{2\al}}(t,t,x)| \nonumber \\
&& \leq 2^{-81}2^{-N(1-\nu_1)}a_n^{-3\eps_0}\big[a_n^{-2\al(1-\eta/4)}2^{-N\gamma \tilde{\gamma}_m}+(a_n^{\al}\vee 2^{-N})^{(\tilde{\gamma}_{m}\gamma-2+\eta/2)\wedge 0}+a_n^{-2\al(1-\eta/4)+\beta\gamma}(a_n^{\al}\vee 2^{-N})^\gamma\big]. \nonumber \\
&&\quad+2^{-99}a_n^{-3\eps_0}(t-s)^{1-\nu_1}\big\{(\sqrt{t-s}\vee a_n^{\al})^{\gamma\tilde{\gamma}_m-2+\eta/2}+a_n^{\beta\gamma}(\sqrt{t-s}\vee a_n^{\al})^{\gamma-2+\eta/2}\big\} \nonumber  \\
&&\quad+2^{-98}a_n^{-3\eps_0-1+\eta/4+\al}2^{N\nu_1}\big(2^{-N\gamma\tilde{\gamma}_m}+a_n^{\beta\gamma}
2^{-N\gamma}\big).
\en
Note that
\bn \label{nb2}
2^{-N(1-\nu_1)}a_n^{-2\al(1-\eta/4)}+2^{N\nu_1}a_n^{-1+\eta/4+\al}\leq 2^{N\nu_1}a_n^{-1+\eta/4+\al}\bigg(\frac{2^{-N}}{a_n^{3\al-\al\eta/2-1+\eta/4}}+1\bigg).
\en
From (\ref{nb1}) and (\ref{nb2}) we get (\ref{5.11res}).
\end{proof}
\\\\
We would like to bound the increment of $\mathbb{G}_{\alpha_n}$ with $\alpha\in [0,2\al]$. We argue, like in  Lemma \ref{Lem.Fdl} for $F_{\dl}$, to get
\bn
\mathbb{G}_{\dl}(s,t,x)=\int_{0}^{(s-\dl)^{+}}\int_{\re}G_{(t\vee s)-r}(y-x)D(r,y)W(dr,dy), \ \ \textrm{for all} \ s  \ \textrm{a.s.} \  \textrm{for all} \  (t,x). \nonumber
\en
We need to bound $\mathbb{G}_{a_n^{\alpha}}(s,t,x)-\mathbb{G}_{a_n^{\alpha}}(t,t,x)$ with $\alpha\in [0,2\al]$, so we repeat the same process that led to Proposition \ref{probIncFdl}. The difference is that now we deal with the Gaussian densities $G_{t-r}$ instead of the derivatives $G'_{t-r}$. Recall that Proposition \ref{probIncFdl} followed from Lemmas \ref{Lemma-Q1}--\ref{Lemma-QT}. In order to bound the increment of $G_{a_n^{\al}}$, one needs analogues to the above lemmas with Gaussian densities $G_{t-r}$ instead of the derivative $G'_{t-r}$. The proofs of these lemmas and the proposition follows the same lines as the proof of Lemmas \ref{Lemma-Q1}--\ref{Lemma-QT} and Proposition \ref{probIncFdl}, therefore they are omitted. Here is the final statement.
\begin{proposition} \label{Prop5.12}
Let $0\leq m \leq \bar{m}+1$ and assume $(P_m)$. For any $n\in \mathds{N}, \nu_1\in (0,\eta/2), \\ \eps_0\in (0,1), K\in \mathds{N}^{\geq K_1}, \alpha \in [0,2\al]$, and $\beta\in [0,\frac{\eta}{\eta+1}]$, there is an \\ $N_{\ref{Prop5.12}}=N_{\ref{Prop5.12}}(m,n,\nu_1,\eps_0,K,\alpha, \beta, \eta, \mu(\re))(\omega)$ in $\mathds{N}$ a.s.
such that for all $N\geq N_{\ref{Prop5.12}}, \\ (t,x)\in Z(N,n,K,\beta), s\leq t$ and $\sqrt{t-s}\leq 2^{-N}$,
\bn
|\mathbb{G}_{a_n^{\alpha}}(s,t,x)-\mathbb{G}_{a_n^{\alpha}}(t,t,x)|&\leq& 2^{-92}(t-s)^{(1-\nu_1)/2}a_n^{-3\eps_0}a_n^{-\alpha(1/2-\eta/4)} \nonumber \\
&&\times \big[(a_n^{\alpha/2}\vee2^{-N})^{\gamma\tilde{\gamma}_m}+a_n^{\beta\gamma}(a_n^{\alpha/2} \vee 2^{-N})^{\gamma}\big]. \nonumber
\en
\end{proposition}
We will use the modulus of continuity of $u'_{1,a_n^{\alpha}}$ from Corollary \ref{corollary-u-tag-reg} to get the modulus of continuity for $u_{1,a_n^{\alpha}}$.
\paragraph{Notation.} Define
\bn
\bar{\Delta}_{u_1}(m,n,\alpha,\eps_0,2^{-N},\eta)&=&a_n^{-3\eps_0-(1-\eta/4)\alpha}\big[a_n^{\beta}a_n^{(1-\eta/4)\alpha}
+a_n^{\beta \gamma} (a_n^{\alpha/2}\vee 2^{-N})^{\gamma+1}   \nonumber \\
&&+(a_n^{\alpha/2}\vee 2^{-N})^{\gamma\tilde{\gamma}_m+1}+\mathds{1}_{\{m\geq \bar{m}\}}a_n^{\alpha(1-\eta/4)}(a_n^{\alpha/2}\vee 2^{-N})^{\eta}\big].  \nonumber
\en
If $\nu>0$ let $N'_{\ref{Prop5.13}}(\nu)$ be the smallest natural number such that $2^{1-N}\leq N^{-4/\nu}$ whenever $N>N'_{\ref{Prop5.13}}(\nu)$.
\begin{proposition} \label{Prop5.13}
Let $0\leq m \leq \bar{m}+1$ and assume $(P_m)$. For any $n\in \mathds{N}, \nu_1\in (0,\gamma-1+\eta/2),$ \\ $\eps_0,\eps_1\in (0,1), K\in \mathds{N}^{\geq K_1}, \alpha \in [0,2\al]$, and $\beta\in [0,\frac{\eta}{\eta+1}]$, there is an $N_{\ref{Prop5.13}}=N_{\ref{Prop5.13}}(m,n,\nu_1,\eps_0,K,\alpha, \beta, \eta, \mu(\re))(\omega)$ in $\mathds{N}$ a.s.
such that for all $N\geq N_{\ref{Prop5.13}}$, $n,\alpha$ satisfying
\bn \label{lm1}
a_n^{2\al}\leq 2^{-2(N_{\ref{Prop5.11}}(m,n,\nu_1/2,\eps_0,K,\beta,\eta,\mu(\re))+1)}\wedge 2^{-2(N'_{\ref{Prop5.13}}(\frac{\nu_1\eps_1}{2\al})+1)}, \textrm{ and } \alpha \geq \eps_1,
\en
$(t,x)\in Z(N,n,K,\beta), \ t'\leq T_K,$ if $d((t,x),(t',x'))\leq 2^{-N}$, then,
\bn
|u_{1,a_n^{\alpha}}(t,x)-u_{1,a_n^{\alpha}}(t',x')|\leq 2^{-90}d((t,x),(t',x'))^{1-\nu_1}\bar{\Delta}_{u_1}(m,n,\alpha,\eps_0,2^{-N},\eta).
\en
Moreover $N_{\ref{Prop5.13}}$ is stochastically bounded uniformly in $n\in \mathds{N}, \alpha \in[0,2\al]$ and $\beta \in [0,\frac{\eta}{\eta+1}]$.
\end{proposition}
\begin{remark}
Although $n$ appears in both sides of (\ref{lm1}), the fact that $N_{\ref{Prop5.11}}$ is stochastically bounded ensures that (\ref{lm1})
holds for infinitely many $n$.
\end{remark}
\begin{proof}
The proof of Proposition \ref{Prop5.13} follows the same lines as the proof of Proposition 5.13 in \cite{MP09}.\\
Let
\bn  \label{lm-1}
N''_{\ref{Prop5.13}}(m,n,\nu_1,\eps_0,K,\alpha,\beta,\eta,\mu(\re)) &=& ((2N_{\ref{probIncFdl}})(m,n,\nu_1/2,\eps_0,K+1,\alpha,\beta,\eta,\mu(\re)) \nonumber \\
&&\vee N_{\ref{Prop5.12}}(m,n,\nu_1,\eps_0,K+1,\alpha,\beta,\eta,\mu(\re))+1.
\en
Hence $N''_{\ref{Prop5.13}}$ is stochastically bounded in $(n,\alpha,\beta)$. Assume (\ref{lm1}) and \
\bn \label{lm0}
N\geq N''_{\ref{Prop5.13}}, (t,x)\in Z(N,n,K,\beta), \ t'\leq T_K, \textrm{ and } d((t,x),(t',x'))\leq 2^{-N}.
\en
As in the proof of Proposition 5.13 in \cite{MP09}, $(t',x')\in Z(N-1,n,K+1,\beta)$ and we may assume $t'\leq t$. \medskip \\
Recall that
\bn \label{lm2}
\mathbb{G}_{a_n^{\alpha}}(t',t,x)=G_{t-t'+a_n^{\alpha}}(u(t'-a_n^{\alpha},\cdot)(x)=G_{t-t'}(u_{1,a_n^{\alpha}}(t',\cdot))(x).
\en
From (\ref{lm2}) we get that the increment of $u_{1,a_n^{\alpha}}$ can be bounded by
\bn \label{lm-22}
|u_{1,a_n^{\alpha}}(t',x')-u_{1,a_n^{\alpha}}(t,x)| &\leq&
|u_{1,a_n^{\alpha}}(t',x')-u_{1,a_n^{\alpha}}(t',x)| \nonumber \\
&&+ |u_{1,a_n^{\alpha}}(t',x)-G_{t-t'}(u_{1,a_n^{\alpha}}(t',\cdot))(x)| \nonumber \\
&&+|\mathbb{G}_{a_n^{\alpha}}(t',t,x)- \mathbb{G}_{a_n^{\alpha}}(t,t,x)| \nonumber \\
&&\equiv T_1+T_2+T_3.
\en
For $T_1$ let $(\hat{t}_0,\hat{x}_0)$ to be as in the definition of $(t,x)\in Z(N,n,K,\beta)$. Let $y$ be between $x'$ and $x$. Therefore, $d((t',y),(t,x))\leq 2^{-N}$ and $d((\hat{t}_0,\hat{x}_0),(t,x))\leq 2^{-N}$.
Use Corollary \ref{corollary-u-tag-reg} twice, with $\nu_1/2$ in place of $\nu_1$, (\ref{u-tag-F-an})  and the definition of $(\hat{t}_0,\hat{x}_0)$ in $Z(N,n,K,\beta)$, to get
\bn \label{lm3}
|u'_{1,a_n^{\alpha}}(t',y)| &\leq& |u'_{1,a_n^{\alpha}}(t',y)-u'_{1,a_n^{\alpha}}(t,x)| \nonumber \\
&&+|u'_{1,a_n^{\alpha}}(t,x)-u'_{1,a_n^{\alpha}}(\hat{t}_0,\hat{x}_0)| \nonumber \\
&&+|u'_{1,a_n^{\alpha}}(\hat{t}_0,\hat{x}_0)-u'_{1,a_n^{2\al}}(\hat{t}_0,\hat{x}_0)| \nonumber \\
&&+|u'_{1,a_n^{2\al}}(\hat{t}_0,\hat{x}_0)| \nonumber \\
&\leq& 2^{-84}2^{-N(1-\nu_1/2)}\bar{\Delta}_{u_1'}(m,n,\alpha,\eps_0,2^{-N}) \nonumber \\
&&+|F_{a_n^{2\al}}(\hat{t}_0-a_n^{\alpha}+a_n^{2\al},\hat{t}_0,\hat{x}_0 )-F_{a_n^{2\al}}(\hat{t}_0,\hat{t}_0,\hat{x}_0)| \nonumber \\
&&+a_n^{\beta}.
\en
We would like to bound the increment of $F_{a_n^{2\al}}$ in (\ref{lm3}) with the use of Proposition \ref{Prop5.11}, but we need some adjustments first.
Choose $N'$ such that
\bn \label{lm4}
2^{-N'-1}\leq a_n^{\al} \leq 2^{-N'}.
\en
From (\ref{lm1}) we have
\bn \label{lm5}
a_n^{\al} \leq 2^{-N_{\ref{Prop5.11}}(m,n,\nu_1/2,\eps_0,K,\beta,\eta,\mu(\re))-1}.
\en
From (\ref{lm4}), (\ref{lm5}) we get
\bn  \label{lm6}
N'\geq N_{\ref{Prop5.11}}(m,n,\nu_1/2,\eps_0,K,\beta,\eta,\mu(\re)).
\en
Also from  (\ref{lm1}) and (\ref{lm4}) we have
\bn\label{lm7}
2^{-N'-1}\leq a_n^{\al} \leq 2^{-N'_{\ref{Prop5.13}}(\frac{\nu_1\eps_1}{2\al})-1}.
\en
Hence $N' \geq N'_{\ref{Prop5.13}}(\frac{\nu_1\eps_1}{2\al})$ and by (\ref{lm4}), our choice of $\alpha \geq \eps_1$ and the perviously introduced notation for $N'_{\ref{Prop5.13}}(\frac{\nu_1\eps_1}{2\al})$ we have
\bn \label{lm8}
a_n^{\alpha/2}\leq 2^{-\frac{N'\alpha}{2\al}}\leq  2^{-\frac{N' \eps_1}{2\al}}\leq N'^{-\frac{4 \eps_1}{2\al}\frac{2\al}{\nu_1 \eps_1}}= N'^{-\frac{4}{\nu_1}}.
\en
By the definition of $(\hat{t}_0,\hat{x}_0)$ in $Z(N,n,K,\beta)$ and (\ref{lm4}) we have
\bn \label{lm9}
|u(\hat{t}_0,\hat{x}_0)|\leq a_n=a_n\wedge (a_n^{1-\al} 2^{-N'}),
\en
and therefore, $(\hat{t}_0,\hat{x}_0)\in Z(N',n,K,\beta)$. From (\ref{lm6}) and (\ref{lm8}) we get that the assumptions of Proposition \ref{Prop5.11} hold with $N'$ instead of $N$, $(\hat{t}_0,\hat{x}_0)$ instead of $(t,x)$, $\nu_1/2$ instead of
$\nu_1$ and $s=\hat{t}_0-a_n^{\alpha}+a_n^{2\al}$. Hence from Proposition \ref{Prop5.11} and the simple inequality $a_n^{\al}\leq a_n^{\alpha/2},\ $
one easily gets
\bn \label{lm10}
&&|F_{a_n^{2\al} }(\hat{t}_0-a_n^{\alpha}+a_n^{2\al} ,\hat{t}_0,\hat{x}_0)-F_{a_n^{2\al} }(\hat{t}_0,\hat{t}_0,\hat{x}_0)| \nonumber  \\
&&\leq 2^{-81}a_n^{-3\eps_0}\bigg\{2^{-N'(1-\nu_1/2)}(a_n^{\al} \vee2^{-N'})^{(\gamma{\gamma}_{m}-2+\eta/2)\wedge 0}  \nonumber  \\
&&\quad+2^{N'\nu_1/2}a_n^{-1+\eta/4+\al}\bigg(\frac{2^{-N'}}{a_n^{3\al-\al\eta/2-1+\eta/4}}+1\bigg)\big(2^{-N'\gamma\tilde{\gamma}_m}+a_n^{\beta\gamma}
(a_n^{\al} \vee2^{-N'})^{\gamma}\big)\nonumber  \\
&& \quad +(a_n^{\alpha}-a_n^{2\al} )^{(1-\nu_1/2)/2}\big((\sqrt{a_n^{\alpha}-a_n^{2\al} }\vee a_n^{\al} )^{\gamma\tilde{\gamma}_m-2+\eta/2}+a_n^{\beta\gamma}(\sqrt{a_n^{\alpha}-a_n^{2\al} }\vee a_n^{\al} )^{\gamma-2+\eta/2}\big)\bigg\} \nonumber  \\
&&\leq 2^{-78}a_n^{-3\eps_0}\big\{a_n^{\al(1-\frac{\nu_1}{2})}a_n^{\al[(\gamma{\gamma}_m-2+\eta/2)\wedge 0]}  \quad +a_n^{-\al(1-\eta/2+\nu_1/2)}\big(a_n^{\al\gamma\tilde{\gamma}_m}+a_n^{\beta\gamma}
a_n^{\al\gamma}\big)\nonumber  \\
&&\quad +a_n^{\frac{\alpha}{2}(1-\frac{\nu_1}{2})}\big(a_n^{\frac{\alpha}{2}(\gamma\tilde {\gamma}_m-2+\frac{\eta}{2})}+a_n^{\beta\gamma}a_n^{\frac{\alpha}{2}(\gamma-2+\frac{\eta}{2})}\big)\big\}
\en
Since $a_n^{\al}\leq a_n^{\alpha/2}$, the middle term in (\ref{lm10}) is bounded by the third term and we get
\bn \label{lm11}
&&|F_{a_n^{2\al}}(\hat{t}_0-a_n^{\alpha}+a_n^{2\al},\hat{t}_0,\hat{x}_0)-F_{a_n^{2\al}}(\hat{t}_0,\hat{t}_0,\hat{x}_0)|  \\
&&= 2^{-77}a_n^{-3\eps_0}\big\{a_n^{\al(1-\frac{\nu_1}{2})}a_n^{\al[(\gamma{\gamma}_m-2+\frac{\eta}{2})\wedge 0]} +a_n^{-\frac{\alpha}{2}(1-\eta/2+\nu_1/2)}\big(a_n^{\frac{\alpha}{2}\gamma\tilde{\gamma}_m}
+a_n^{\beta\gamma}a_n^{\frac{\alpha}{2}\gamma}\big)\big\}. \nonumber
\en
Recall from (\ref{gamma-rec}) that $ \gamma_m \geq 1$. Hence by our assumptions on $\nu_1$ ($\nu_1\in(0,\gamma-1+\eta/2)$) and (\ref{gamma-rec}) we get that
\bn \label{lm12.11}
\gamma\gamma_m\geq 1-\frac{\eta}{2}+\nu_1.
\en
From (\ref{lm12.11}) we immediately get
\bn \label{lm12}
1-\frac{\nu_1}{2}+\gamma\gamma_{m}-2+\frac{\eta}{2}\geq \frac{\nu_1}{2}>0.
\en
From (\ref{lm12}) we have
\bn\label{lm13}
a_n^{\al[1-\frac{\nu_1}{2}+(\gamma{\gamma}_{m}-2+\frac{\eta}{2})\wedge0]} \leq (a_n^{\alpha/2}\vee 2^{-N})^{1-\frac{\nu_1}{2}+(\gamma\gamma_m-2+\eta/2)\wedge0}.
\en
Recall the definition of $\bar{\Delta}_{u_1'}(m,n,\alpha,\eps_0,2^{-N})$ given in (\ref{Delta-u1-tag}).
Apply (\ref{lm11}) and (\ref{lm13}) to (\ref{lm3}) to get
\bn \label{lm14}
|u'_{1,a_n^{\alpha}}(t',y)|
&\leq& 2^{-84}2^{-N(1-\nu_1/2)}a_n^{-2\eps_0}\big[a_n^{-\alpha(1-\eta/4)}2^{-N\gamma \tilde{\gamma}_m}\nonumber \\
&&+(a_n^{\alpha/2}\vee 2^{-N})^{(\gamma\gamma_m-2+\eta/2)\wedge 0}+a_n^{-\alpha(1-\eta/4)+\beta\gamma}(a_n^{\alpha/2}\vee 2^{-N})^\gamma\big] \nonumber \\
&&+2^{-77}a_n^{-3\eps_0}\big\{(a_n^{\alpha/2}\vee 2^{-N})^{1-\frac{\nu_1}{2}+(\gamma\gamma_m-2+\eta/2)\wedge0} +a_n^{-\frac{\alpha}{2}(1-\eta/2+\nu_1/2)}\big(a_n^{\frac{\alpha}{2}\gamma\tilde{\gamma}_m}
+a_n^{\beta\gamma}a_n^{\frac{\alpha}{2}\gamma}\big)\big\} \nonumber \\
&&+a_n^{\beta} \nonumber \\
&\leq & 2^{-77}a_n^{-3\eps_0} \bigg\{2^{-N(1-\nu_1/2)}(a_n^{\alpha/2}\vee 2^{-N})^{(\gamma\gamma_m-2+\eta/2)\wedge 0}\nonumber \\
&&+2^{-N(1-\nu_1/2)}a_n^{-\alpha(1-\eta/4)}\big[2^{-N\gamma \tilde{\gamma}_m}+a_n^{\beta\gamma}(a_n^{\alpha/2}\vee 2^{-N})^\gamma\big] \nonumber \\
&&+(a_n^{\alpha/2}\vee 2^{-N})^{1-\frac{\nu_1}{2}+(\gamma\gamma_m-2+\eta/2)\wedge0} +a_n^{\frac{\alpha}{2}(1-\frac{\nu_1}{2})}a_n^{\alpha(-1+\frac{\eta}{4})}\big(a_n^{\frac{\alpha}{2}\gamma\tilde{\gamma}_m}
+a_n^{\beta\gamma}a_n^{\frac{\alpha}{2}\gamma}\big)\bigg\} \nonumber \\
&&+a_n^{\beta} \nonumber \\
&\leq & 2^{-77}a_n^{-3\eps_0}(a_n^{\alpha/2}\vee2^{-N})^{(1-\nu_1/2)} \bigg\{(a_n^{\alpha/2}\vee 2^{-N})^{(\gamma\gamma_m-2+\eta/2)\wedge 0}\nonumber \\
&&+a_n^{-\alpha(1-\eta/4)}\big[2^{-N\gamma \tilde{\gamma}_m}+a_n^{\beta\gamma}(a_n^{\alpha/2}\vee 2^{-N})^\gamma\big] \nonumber \\
&&+(a_n^{\alpha/2}\vee 2^{-N})^{(\gamma\gamma_m-2+\eta/2)\wedge0} +a_n^{\alpha(-1+\frac{\eta}{4})}\big(a_n^{\frac{\alpha}{2}\gamma\tilde{\gamma}_m}
+a_n^{\beta\gamma}a_n^{\frac{\alpha}{2}\gamma}\big)\bigg\} \nonumber \\
&&+a_n^{\beta} \nonumber \\
&\leq & 2^{-77}a_n^{-3\eps_0}(a_n^{\alpha/2}\vee2^{-N})^{(1-\nu_1/2)} \bigg\{(a_n^{\alpha/2}\vee 2^{-N})^{(\gamma\gamma_m-2+\eta/2)\wedge 0}\nonumber \\
&&+a_n^{-\alpha(1-\eta/4)}\big[(a_n^{\alpha/2}\vee2^{-N})^{ \gamma\tilde{\gamma}_m}+a_n^{\beta\gamma}(a_n^{\alpha/2}\vee 2^{-N})^\gamma\big] \nonumber \\
&&+(a_n^{\alpha/2}\vee 2^{-N})^{(\gamma\gamma_m-2+\eta/2)\wedge0} +a_n^{-\alpha(1-\frac{\eta}{4})}\big((a_n^{\alpha/2}\vee2^{-N})^{ \gamma\tilde{\gamma}_m}
+a_n^{\beta\gamma}(a_n^{\alpha/2}\vee2^{-N})^{ \gamma}\big)\bigg\} \nonumber \\
&&+a_n^{\beta} \nonumber \\
&\leq & 2^{-76}a_n^{-3\eps_0}(a_n^{\alpha/2}\vee2^{-N})^{(1-\nu_1/2)} \bigg\{(a_n^{\alpha/2}\vee 2^{-N})^{(\gamma\gamma_m-2+\eta/2)\wedge 0}\nonumber \\
&&+a_n^{-\alpha(1-\eta/4)}(a_n^{\alpha/2}\vee2^{-N})^{ \gamma\tilde{\gamma}_m}+a_n^{-\alpha(1-\eta/4)}a_n^{\beta\gamma}(a_n^{\alpha/2}\vee 2^{-N})^\gamma \bigg\} \nonumber \\
&&+a_n^{\beta} \nonumber \\
&\equiv& 2^{-76} \tilde{\Delta}_{u_1}(m,n,\alpha,\eps_0,\nu_1,a_n^{\alpha/2}\vee 2^{-N},\eta)+a_n^\beta.
\en
From (\ref{lm12}) and the assumption $\nu_1<\gamma-1+\eta/2$ we get that $1-\nu_1/2>0$ and hence $\tilde{\Delta}_{u_1}$ is monotone increasing in $a_n^{\alpha/2}\vee 2^{-N}$. From the Mean Value Theorem and (\ref{lm14}) we get
\bn \label{lm15}
T_1\leq \big[2^{-76} \tilde{\Delta}_{u_1}(m,n,\alpha,\eps_0,\nu_1,a_n^{\alpha/2}\vee 2^{-N},\eta)+a_n^\beta\big]|x-x'|.
\en
Recall that $t'\leq t $. Form (\ref{lm-1}) and (\ref{lm0}) we get that $N>N_{\ref{Prop5.12}}$ and $\sqrt{t'-t}\leq 2^{-N}$. Apply Proposition \ref{Prop5.12} to $T_3$ to get
\bn \label{lm16}
T_3\leq 2^{-92}(t'-t)^{(1-\nu_1)/2}a_n^{-3\eps_0}a_n^{-\alpha(1/2-\eta/4)} \big[(a_n^{\alpha/2}\vee2^{-N})^{\gamma\tilde{\gamma}_m}+a_n^{\beta\gamma}(a_n^{\alpha/2} \vee 2^{-N})^{\gamma}\big].
\en
The last term that we have to bound is $T_2$. Let $\{B(s):s\geq 0 \}$ be a one dimensional Brownian motion, starting at $x$ under $P_x$. We assume first that
\bn \label{lm17}
|B(t-t')-x|\leq 2^{-\frac{3}{2}N_{\ref{probIncFdl}}}.
\en
Recall that $N_{\ref{probIncFdl}}\geq2$. From (\ref{lm-1}) we get that that $N>2N_{\ref{probIncFdl}}$ and
\bn \label{lm18}
d((t',B(t-t')),(t,x))&\leq& \sqrt{t-t'}+2^{-\frac{3}{2}N_{\ref{probIncFdl}}} \nonumber \\
&\leq& 2^{-N}+2^{-\frac{3}{2}N_{\ref{probIncFdl}}} \nonumber \\
&\leq& 2^{-2N_{\ref{probIncFdl}}}+2^{-\frac{3}{2}N_{\ref{probIncFdl}}} \nonumber \\
&\leq& 2^{-N_{\ref{probIncFdl}}}.
\en
Define a random $N'\in \{N_{\ref{probIncFdl}},...,N\}$ by
\begin{enumerate}
  \item if $d((t',B(t-t')),(t,x)) \leq 2^{-N}$ \ then $N'=N$;
  \item if $d((t',B(t-t')),(t,x)) > 2^{-N}$ \  then $2^{-N'-1}<d((t',B(t-t')),(t,x))\leq 2^{-N'}$.
\end{enumerate}
In case (2) we have $2^{-N'-1}\leq 2^{-N}+|B(t-t')-x|$ and so
\bn \label{lm19}
2^{-N'}\leq 2^{1-N}+|B(t-t')-x|.
\en
(\ref{lm19}) is trivial in case (1). If $y$ is between $x$ and $B(t-t')$, argue as in (\ref{lm14}), this time using \\
$(t,x)\in Z(N',n,K,\beta)$, to see that
\bn \label{lm19.1}
|u'_{1,a_n^{\alpha}}(t',y)|\leq  2^{-76} \tilde{\Delta}_{u_1}(m,n,\alpha,\eps_0,\nu_1,a_n^{\alpha/2}\vee 2^{-N'},\eta)+a_n^\beta.
\en
Use (\ref{lm19}) and the monotonicity of $\tilde{\Delta}_{u_1}$ again to get
\bn \label{lm20}
&& a_n^{3\eps_0}\tilde{\Delta}_{u_1}(m,n,\alpha,\eps_0,\nu_1,a_n^{\alpha/2}\vee 2^{-N'},\eta)  \\
&&\leq 8(a_n^{\alpha/2}+2^{-N}+|B(t-t')-x|)^{(1-\nu_1/2)} \bigg\{(a_n^{\alpha/2}+2^{-N}+|B(t-t')-x|)^{(\gamma\tilde\gamma_m-2+\eta/2)\wedge 0}\nonumber \\
&&\quad +a_n^{-\alpha(1-\eta/4)}\bigg[(a_n^{\alpha/2}+2^{-N}+|B(t-t')-x|)^{ \gamma\tilde{\gamma}_m}+a_n^{\beta\gamma}(a_n^{\alpha/2}+2^{-N}+|B(t-t')-x|)^\gamma \bigg]\bigg\}. \nonumber
\en
 Substitute (\ref{lm20}) in (\ref{lm19.1}) and use the Mean Value Theorem (the expectation is on $B$ alone, $N_{\ref{probIncFdl}}$ remains fixed) to get
\bn \label{lm21}
&&E_x\big(\mathds{1}_{\{|B(t-t')-x|\leq 2^{-\frac{3}{2}N_{\ref{probIncFdl}}}\}}|u_{1,a_n^\alpha}(t',B(t-t'))-u_{1,a_n^\alpha}(t',x)|\big) \nonumber \\
&&\leq E_0\bigg(|B(t-t')|\bigg\{a_n^{\beta}+a_n^{-3\eps_0}(a_n^{(\alpha/2)(1-\nu_1/2)}+2^{-N(1-\nu_1/2)}+|B(t-t')-x|^{1-\nu_1/2})\nonumber \\
&&\quad \times \bigg[a_n^{-\alpha(1-\eta/4)}(a_n^{\alpha/2}+2^{-N}+|B(t-t')|)^{
\gamma\tilde{\gamma}_m}+a_n^{-\alpha(1-\eta/4)}a_n^{\beta\gamma}(a_n^{\alpha/2}+2^{-N}+|B(t-t')|)^\gamma  \nonumber \\
&&\quad +(a_n^{\alpha/2}\vee 2^{-N})^{(\gamma\tilde\gamma_m-2+\eta/2)\wedge 0}\bigg]\bigg\} \nonumber \\
&&\leq C\sqrt{t-t'}\bigg\{a_n^{\beta}+a_n^{-3\eps_0}((a_n^{\alpha/2}\vee 2^{-N})^{(1-\nu_1/2)}+(t-t')^{1/2-\nu_1/4})\nonumber \\
&&\quad \times\bigg[a_n^{-\alpha(1-\eta/4)}(a_n^{\alpha/2}\vee2^{-N}+\sqrt{t-t'})^{
\gamma\tilde{\gamma}_m}+a_n^{-\alpha(1-\eta/4)}a_n^{\beta\gamma}(a_n^{\alpha/2} \vee 2^{-N}+\sqrt{t-t'})^\gamma \nonumber \\
&&\quad +(a_n^{\alpha/2}\vee 2^{-N})^{(\gamma\tilde\gamma_m-2+\eta/2)\wedge 0}\bigg] \bigg\} \nonumber \\
&&\leq C\sqrt{t-t'}\bigg\{a_n^{\beta}+a_n^{-3\eps_0}((a_n^{\alpha/2}\vee 2^{-N})^{(1-\nu_1/2)})\nonumber \\
&&\quad \times \bigg[a_n^{-\alpha(1-\eta/4)}(a_n^{\alpha/2}\vee2^{-N})^{
\gamma\tilde{\gamma}_m}+a_n^{-\alpha(1-\eta/4)}a_n^{\beta\gamma}(a_n^{\alpha/2} \vee 2^{-N})^\gamma \nonumber \\
&&\quad +(a_n^{\alpha/2}\vee 2^{-N})^{(\gamma\tilde\gamma_m-2+\eta/2)\wedge 0}\bigg] \bigg\} \ (\textrm{since } \sqrt{t-t'}\leq 2^{-N}) \nonumber \\
&&=C\sqrt{t-t'}[a_n^{\beta}+\tilde{\Delta}_{u_1}(m,n,\alpha,\eps_0,\nu_1,a_n^{\alpha/2}\vee 2^{-N},\eta)].
\en
Assume now that
\bn \label{lm22}
|B(t-t')-x| > 2^{-\frac{3}{2}N_{\ref{probIncFdl}}}.
\en
Note that for $K\geq K_1$ and $t'\leq T_K$,therefore from (\ref{u1}) we get
\bn \label{lm23}
|u_{1,a_n^{\alpha}}(t',y)|&\leq& E_y(|u((t'-a_n^{\alpha})^{+},B(a_n^{\alpha}))|)\nonumber \\
&\leq & 2KE_y(e^{|B(a_n^{\alpha})|})\nonumber \\
&\leq& 2Ke^{1+|y|}.
\en
From (\ref{lm23}) and the fact that $\sqrt{t-t'}\leq 2^{-2N_{\ref{probIncFdl}}}$ we get
\bn \label{lm24}
&&E_x\big(\mathds{1}_{\{|B(t-t')-x|> 2^{-3N_{\ref{probIncFdl}/2}}\}}|u_{1,a_n^\alpha}(t',B(t-t'))-u_{1,a_n^\alpha}(t',x)|\big) \nonumber \\
&&\leq P_0\big(|B(t-t')|> 2^{-3N_{\ref{probIncFdl}}/2}\big)^{1/2}8Ke E_x(e^{2|B(t'-t)|}+e^{2|x|})^{1/2} \nonumber \\
&&\leq C(K)P_0(|B(1)|>(t-t')^{-1/8})^{1/2} \qquad  \qquad  \qquad \  \textrm{( since } |x|\leq K \textrm{ by (\ref{lm0}))}  \nonumber \\
&&\leq C(K)(t-t') \nonumber \\
&&=C(K)\sqrt{t-t'}[\tilde{\Delta}_{u_1}(m,n,\alpha,\eps_0,\nu_1,a_n^{\alpha/2}\vee 2^{-N},\eta)],
\en
where in the last line we have used
\bn \label{lm25}
\tilde{\Delta}_{u_1}(m,n,\alpha,\eps_0,\nu_1,a_n^{\alpha/2}\vee 2^{-N},\eta) &\geq & (a_n^{\alpha/2}\vee 2^{-N})^{1-\nu_1/2} \nonumber \\
&\geq & 2^{-N} \nonumber \\
&\geq & \sqrt{t'-t}.
\en
From (\ref{lm21}) and  (\ref{lm24}) we get
\bn \label{lm26}
T_2\leq C(K)\sqrt{t-t'}[a_n^{\beta}+\tilde{\Delta}_{u_1}(m,n,\alpha,\eps_0,\nu_1,a_n^{\alpha/2}\vee 2^{-N},\eta)].
\en
Use (\ref{lm15}), (\ref{lm16}) and (\ref{lm26}) in (\ref{lm-22}), then use $d((t,x),(t',x'))\leq 2^{-N}$, to get
\bn \label{lm27}
&&|u_{1,a_n^{\alpha}}(t',x')-u_{1,a_n^{\alpha}}(t,x)|  \nonumber \\ \nonumber \\
&&\leq  C(K)d((t,x),(t',x'))[a_n^{\beta}+\tilde{\Delta}_{u_1}(m,n,\alpha,\eps_0,\nu_1,a_n^{\alpha/2}\vee 2^{-N},\eta)] \nonumber \\ \nonumber \\
&&\quad +2^{-92}(t'-t)^{(1-\nu_1)/2}a_n^{-\alpha(1/2-\eta/4)-3\eps_0}a_n^{-\alpha(1/2-\eta/4)} \big[(a_n^{\alpha/2}\vee2^{-N})^{\gamma\tilde{\gamma}_m}+a_n^{\beta\gamma}(a_n^{\alpha/2} \vee 2^{-N})^{\gamma}\big] \nonumber \\ \nonumber \\
&&=  C(K)d((t,x),(t',x'))\bigg[a_n^{\beta}+a_n^{-3\eps_0}(a_n^{\alpha/2}\vee2^{-N})^{(1-\nu_1/2)} \big\{(a_n^{\alpha/2}\vee 2^{-N})^{(\gamma\tilde\gamma_m-2+\eta/2)\wedge 0}\nonumber \\ \nonumber \\
&&\quad +a_n^{-\alpha(1-\eta/4)}(a_n^{\alpha/2}\vee2^{-N})^{ \gamma\tilde{\gamma}_m}+a_n^{-\alpha(1-\eta/4)}a_n^{\beta\gamma}(a_n^{\alpha/2}\vee 2^{-N})^\gamma \big\} \bigg]\nonumber \\ \nonumber \\
&&\quad +2^{-92}(t'-t)^{(1-\nu_1)/2}a_n^{\alpha/2}a_n^{-\alpha(1-\eta/4)-3\eps_0} \big[(a_n^{\alpha/2}\vee2^{-N})^{\gamma\tilde{\gamma}_m}+a_n^{\beta\gamma}(a_n^{\alpha/2} \vee 2^{-N})^{\gamma}\big] \nonumber \\  \nonumber \\
&&\leq  C(K)d((t,x),(t',x'))^{1-\nu_1}2^{-N\nu_1}a_n^{-\alpha(1-\eta/4)-3\eps_0} \bigg[a_n^{\alpha(1-\eta/4)} a_n^{\beta}\nonumber \\ \nonumber \\
&&\quad +2^{N\nu_1/2}(a_n^{\alpha/2}\vee2^{-N})a_n^{\alpha(1-\eta/4)} (a_n^{\alpha/2}\vee 2^{-N})^{(\gamma\tilde\gamma_m-2+\eta/2)\wedge 0}\nonumber \\ \nonumber \\
&&\quad +(a_n^{\alpha/2}\vee2^{-N})^{ \gamma\tilde{\gamma}_m}+a_n^{\beta\gamma}(a_n^{\alpha/2}\vee 2^{-N})^\gamma \big\} \bigg]\nonumber \\ \nonumber \\
&&\quad +2^{-92}d((t,x),(t',x'))^{(1-\nu_1)}a_n^{\alpha/2}a_n^{-\alpha(1-\eta/4)-3\eps_0} \big[(a_n^{\alpha/2}\vee2^{-N})^{\gamma\tilde{\gamma}_m}+a_n^{\beta\gamma}(a_n^{\alpha/2} \vee 2^{-N})^{\gamma}\big] \nonumber \\  \nonumber \\
&&\leq (C_{\ref{lm27}}(K)2^{-N\nu_1/2}+2^{-92})d((t,x),(t',x'))^{1-\nu_1}a_n^{-\alpha(1-\eta/4)-3\eps_0} \bigg[a_n^{\alpha(1-\eta/4)} a_n^{\beta}\nonumber \\ \nonumber \\
&&\quad +(a_n^{\alpha/2}\vee2^{-N}) \big\{a_n^{\alpha(1-\eta/4)} (a_n^{\alpha/2}\vee 2^{-N})^{(\gamma\tilde\gamma_m-2+\eta/2)\wedge 0}\nonumber \\ \nonumber \\
&&\quad  +(a_n^{\alpha/2}\vee2^{-N})^{ \gamma\tilde{\gamma}_m}+a_n^{\beta\gamma}(a_n^{\alpha/2}\vee 2^{-N})^\gamma \big\} \bigg].
\en
Choose $N_1(K,\nu_1)$ such that
\bn
2^{-N_1\nu_1/2}C_{\ref{lm27}}(K)\leq2^{-92},
\en
and define $N_{\ref{Prop5.13}}=N''_{\ref{Prop5.13}}\vee N_1$ which is stochastically bounded uniformly in $(n,\alpha,\beta)\in \mathds{N}\times [0,2\al]\times [0,\frac{\eta}{\eta+1}]$. Assume $N\geq N_{\ref{Prop5.13}}$. Note that if $m<\bar{m}$, from (\ref{gamma-rec}) and (\ref{m-bar-def}) we get
\bn \label{lm28}
a_n^{\alpha(1-\eta/4)}(a_n^{\alpha/2}\vee 2^{-N})^{(\gamma\tilde\gamma_m-2+\eta/2)\wedge0}&=&a_n^{\alpha(1-\eta/4)}(a_n^{\alpha/2}\vee 2^{-N})^{\gamma\gamma_{m}-2+\eta/2}\nonumber \\ \nonumber \\
&\leq& (a_n^{\alpha/2}\vee 2^{-N})^{\gamma\gamma_{m}} \nonumber \\ \nonumber \\
&\leq& (a_n^{\alpha/2}\vee 2^{-N})^{\gamma\tilde{\gamma}_{m}}.
\en
If $m\geq \bar{m}$ we have
\bn \label{lm29}
a_n^{\alpha(1-\eta/4)}(a_n^{\alpha/2}\vee 2^{-N})^{(\gamma\tilde\gamma_m-2+\eta/2)\wedge 0}&=&a_n^{\alpha(1-\eta/4)}(a_n^{\alpha/2}\vee 2^{-N})^{\eta-1}.
\en
From (\ref{lm27})--(\ref{lm29}) we get
\bn \label{lm30}
&&|u_{1,a_n^{\alpha}}(t',x')-u_{1,a_n^{\alpha}}(t,x)|  \nonumber \\ \nonumber \\
&&\leq 2^{-90}d((t,x),(t',x'))^{1-\nu_1}a_n^{-\alpha(1-\eta/4)-3\eps_0} \bigg[a_n^{\alpha(1-\eta/4)+\beta}+\mathds{1}_{\{m\geq\bar{m}\}}(a_n^{\alpha/2}\vee2^{-N})^{\eta}a_n^{\alpha(1-\eta/4)}\nonumber \\ \nonumber \\
&&\quad +(a_n^{\alpha/2}\vee2^{-N})^{ \gamma\tilde{\gamma}_m+1}+a_n^{\beta\gamma}(a_n^{\alpha/2}\vee 2^{-N})^{\gamma+1} \bigg]\nonumber \\ \nonumber \\
&&= 2^{-90}d((t,x),(t',x'))^{1-\nu_1}\bar{\Delta}_{u_1}(m,n,\alpha,\eps_0,2^{-N},\eta). \nonumber
\en
\end{proof}
\\\\
We need to bound the increments of $u_{2,a_n^{\alpha}}$ as we did for $u_{1,a_n^{\alpha}}$ in Proposition \ref{Prop5.13}. First we introduce the following notation.
\paragraph{Notation.}
\bn \label{notat-Del12}
\bar{\Delta}_{1,u_2}(m,n,\eps_0,2^{-N},\eta)&=&a_n^{-3\eps_0}2^{-N\gamma}\big[(2^{-N}\vee a_n^{\alpha})^{\gamma(\tilde{\gamma}_m-1)}+a_n^{\gamma\beta}\big], \nonumber \\
\bar{\Delta}_{2,u_2}(m,n,\eps_0,\eta)&=&a_n^{-3\eps_0}\big[a_n^{(\alpha/2)(\gamma\tilde{\gamma}_m-1+\eta/2)}+a_n^{(\alpha/2)(\gamma-1+\eta/2)}a_n^{\gamma\beta}\big].
\en
\begin{proposition} \label{Prop5.14}
Let $0\leq m \leq \bar{m}+1$ and assume $(P_m)$. Let $\theta\in (0,\gamma-1+\eta/2)$. then for any $n\in \mathds{N}, \nu_1\in (0,\theta), \eps_0 \in(0,1), K\in \mathds{N}^\mathds{\geq K_1}, \alpha\in[0,2\al]$, and $\beta\in [0,\frac{\eta}{\eta+1}]$, there is an $N_{\ref{Prop5.14} }=N_{\ref{Prop5.14} }(m,n,\nu_1,\eps_0,K,\alpha,\beta,\eta)$ a.s. such that for all $N\geq N_{\ref{Prop5.14}}, (t,x)\in Z(N,n,K,\beta)$, and
$t'\leq T_K$,
\bn
&&d\equiv d((t,x),(t',x'))\leq 2^{-N} \textrm{ implies that } \nonumber \\
&&|u_{2,a_n^\alpha}(t,x)- u_{2,a_n^\alpha}(t',x')| \leq 2^{-89}\bigg[d^{\frac{\eta-\nu_1}{2}}\bar{\Delta}_{1,u_2}(m,n,\eps_0,2^{-N},\eta)+d^{1-\nu_1}\bar{\Delta}_{2,u_2}(m,n,\eps_0,\eta)\bigg]. \nonumber
\en
Moreover, $N_{\ref{Prop5.14}}$ is stochastically bounded, uniformly in $(n,\alpha,\beta)$.
\end{proposition}
The proof of Proposition \ref{Prop5.14} follows the same lines of the proof of Proposition \ref{Prop5.13}, and hence we omit it.
Also see Section 7 in \cite{MP09}.\medskip \\
The following lemma is crucial to the proof of Proposition \ref{Prop-Induction}.
\begin{lemma}  \label{Lemma5.15}
Let $\gamma$ satisfy (\ref{optsol}), that is $\gamma>1-\frac{\eta}{2(\eta+1)}$. For all $n,m\in \mathds{N}, 0\leq \beta \leq \frac{\eta}{\eta+1} $ and $0<d\leq 1$,
\bn \label{lem-new}
a_n^{\beta\gamma}(a_n^{\al}\vee d)^{\gamma_1-1}\leq a_n^{\beta}+(d\vee a_n^{\al})^{\tilde{\gamma}_{m+1}-1}.
\en
\end{lemma}
\begin{proof}
The proof of Lemma \ref{Lemma5.15} follows that same lines of Lemma 5.15 in \cite{MP09}. \medskip \\
From (\ref{gamma-rec}) we have
\bn \label{ins00}
\gamma_1-1=\gamma-1+\eta/2.
\en
From (\ref{optsol}) we have
\bn \label{001}
\gamma+\frac{\eta}{2}-1>\frac{\eta^2}{2(\eta+1)}.
\en
From (\ref{alpha_0}), (\ref{001}), (\ref{optsol}) and (\ref{beta-up-lim}) we have
\bn \label{ins111}
\beta(\gamma-1)+\al(\gamma+\eta/2-1) &\geq&  \beta(\gamma-1)+\frac{1}{\eta+1}\bigg(\frac{\eta^2}{2(\eta+1)}\bigg) \nonumber \\
&\geq& -\beta\frac{\eta}{2(\eta+1)}+\frac{1}{\eta+1}\bigg(\frac{\eta^2}{2(\eta+1)}\bigg) \nonumber \\
&\geq& -\frac{\eta}{\eta+1}\frac{\eta}{2(\eta+1)}+\frac{1}{\eta+1}\bigg(\frac{\eta^2}{2(\eta+1)}\bigg) \nonumber \\
&\geq& 0.
\en
From (\ref{ins111}) we have
\bn \label{ins1}
\beta\gamma+\al(\gamma+\eta/2-1) &\geq& \beta.
\en
\textrm{Case 1.} $d\leq a_n^{\al}$.
From (\ref{ins00}) and (\ref{ins1}) we get
\bn \label{case1}
a_n^{\beta\gamma}(a_n^{\al}\vee d)^{\gamma_1-1} &=& a_n^{\beta\gamma+\al(\gamma+\eta/2-1)} \nonumber \\
&\leq&a_n^{\beta}. \nonumber \\
\en
Case 2. $a_n^{\al}<d \leq a_n^{\beta}$ and $a_n^{\beta}\leq d^\eta$. From (\ref{optsol}) we have,
\bn \label{case2}
a_n^{\beta\gamma}(a_n^{\al}\vee d)^{\gamma_1-1} &\leq & d^{(1+\eta)\gamma+\eta/2-1} \nonumber \\
&\leq& d^{\eta}\nonumber \\
&=& d^{\tilde \gamma_{\tilde m +1}-1}\nonumber \\
&\leq& d^{\tilde{\gamma}_{m+1}-1},
\en
where we have used (\ref{gamma-m-up-lim}). \medskip \\
Case 3. $a_n^{\al}<d \leq a_n^{\beta}$ and $d^\eta < a_n^{\beta}$. From (\ref{optsol}) we have
\bn \label{case0}
\beta\gamma+(\gamma+\eta/2-1)\frac{\beta}{\eta} &\geq& \beta \bigg(\gamma\frac{\eta+1}{\eta}+\frac{1}{2}-\frac{1}{\eta}\bigg)\nonumber \\
&\geq&\beta.
\en
From (\ref{case0}), (\ref{ins00}) and the assumption of this case we have
\bn \label{case3}
a_n^{\beta\gamma}(a_n^{\al}\vee d)^{\gamma_1-1} &\leq & a_n^{\beta\gamma+(\gamma+\eta/2-1)\beta/\eta} \nonumber \\
&\leq& a_n^{\beta}.
\en
Case 4. $a_n^{\beta}<d $. The following inequity follows directly from (\ref{optsol}),
\bn \label{ox-in}
2\gamma+\frac{\eta}{2}-1 \geq \eta.
\en
Note that by (\ref{beta-up-lim}) and (\ref{alpha_0}),  $\beta<\al$. In this case we have in particular $a_n^{\al}<d $. This, (\ref{optsol}) and (\ref{ox-in}) imply
\bn \label{case4}
a_n^{\beta\gamma}(a_n^{\al}\vee d)^{\gamma_1-1} &\leq & d^{2\gamma+\eta/2-1} \nonumber \\
&\leq& d^\eta \nonumber \\
&\leq& d^{\tilde\gamma_{m+1}-1},
\en
where for the last inequality we have used the same argument as in (\ref{case2}). \medskip \\
From (\ref{case1}), (\ref{case2}), (\ref{case3}) and (\ref{case4}), (\ref{lem-new}) follows.
\end{proof}
\paragraph{Proof of Proposition \ref{Prop-Induction}}
The proof follows the same lines as the proof of Proposition 5.1 in \cite{MP09}. \\\\
Let $0\leq m \leq \bar{m}$ and assume $(P_m)$. Our goal is to derive $(P_{m+1})$. Let $\eps_0\in (0,1), M=\lceil\frac{2}{\eta\eps_0}\rceil, \eps_2=\frac{2\al}{M}\leq \al\eps_0\eta $ and set $\alpha_i=i\eps_2$ for $i=0,1,...,M$, so that
$\alpha_i\in[\eps_2,2\al]$ for $i\geq 1$. Recall that $\gamma$ satisfies (\ref{optsol}) is fixed. From (\ref{optsol}) we get that $\gamma>1-\eta/2$. Let $n,\xi,K$ be as in $(P_m)$ where we may assume $\xi>2-\gamma-\eta/2$ without loss of generality. Define $\nu_1=1-\xi\in(0,\gamma-1+\eta/2)$, $\xi'=\xi+(1-\xi)/2 \in (\xi,1)$,
\bn
N_2(m,n,\xi,\eps_0,K,\beta,\eta)(\omega) &=& \vee_{i=1}^{M}N_{\ref{Prop5.13}}(m,n,\nu_1,\eps_0/6,K+1,\alpha_i,\beta,\eta)(\omega),  \nonumber \\
N_3(m,n,\xi,\eps_0,K,\beta,\eta)(\omega) &=& \vee_{i=1}^{M}N_{\ref{Prop5.14}}(m,n,\nu_1,\eps_0/6,K+1,\alpha_i,\beta,\eta)(\omega), \nonumber
\en
\bn \label{N4-N5-def}
N_4(m,n,\xi,\eps_0,K,\beta)&=&\lceil \frac{2}{1-\xi}((N_{\ref{Prop5.11}}(m,n,\nu_1/2,\eps_0 /6,K+1,\beta,\eta)\vee N'_{\ref{Prop5.13}}(\nu_1\eps_2/(2\al)))+1)\rceil \nonumber \\
&\equiv& \lceil \frac{1}{1-\xi}N_5(m,n,\nu_1,\eps_0,K,\beta)\rceil,
\en
\bq
N_6(\xi,K,\beta)=N_1(0,\xi',K,\eta),
\eq
and
\bn
N_1(m,n,\xi,\eps_0,K,\beta,\eta)=(N_2\vee N_3\vee N_4(m,n,\xi,\eps_0,K,\beta,\eta))\vee N_6(\xi,K,\beta)+1 \in \mathds{N}, \  P -\textrm{a.s.} \nonumber
\en
Recall that in the verification of $(P_0)$, we chose $\eps_0=0$, and $N_1=N_1(0,\xi',K,\eta)$ was independent of $n$ and $\beta$. Therefore $N_1$ is stochastically bounded uniformly in $(n,\beta)$ because $N_{\ref{Prop5.11}},N_{\ref{Prop5.13}},N_{\ref{Prop5.14}}$ are.  \\\\
Assume $N\geq N_1, \ (t,x)\in Z(N,n,K,\beta), \ t^{'}\leq T_K \textrm{ and } d((t,x),(t',x')) \leq 2^{-N}$.
Consider first the case of
\bn \label{vx1}
a^{2\al}_n>2^{-N_5(m,n,\nu_1,\eps_0,K,\beta,\eta)}.
\en
Recall that $\tilde{\gamma}_{m+1}-1\leq \eta $. Since $N\geq N_1(0,\xi',K)$, we get from $(P_0)$ with $\eps_0=0$ and $\xi'$ in place of $\xi$ and then, (\ref{vx1}) and our choice of $\xi'$,
\bn \label{vx2}
|u(t',x')|&\leq & 2^{-N\xi'} \nonumber \\
&\leq &  2^{-N\xi'}\big[(a_n^{\al}\vee2^{-N})^{\tilde{\gamma}_{m+1}-1}\big]2^{N_5/2} \nonumber \\
&\leq &  2^{-N(1-\xi)/2}2^{N_5/2}2^{-N\xi}\big[(a_n^{\al}\vee2^{-N})^{\tilde{\gamma}_{m+1}-1}+a_n^{\beta}\big] \nonumber \\
&\leq & 2^{-N\xi}\big[(a_n^{\al}\vee2^{-N})^{\tilde{\gamma}_{m+1}-1}+a_n^{\beta}\big],
\en
where we have used $N\geq N_{4}\geq (1-\xi)^{-1}N_5$ in the last line. From (\ref{vx2}), $(P_{m+1})$ follows.  \\\\
Next, we deal with the complement of (\ref{vx1}). Assume that
\bn \label{vx3}
a_n^{2\al} \leq 2^{-N_5(m,n,\nu_1,\eps_0,K,\beta,\eta)}.
\en
Let $N'=N-1\geq N_2\vee N_3$. Note that $(\hat{t}_0,\hat{x}_0)$, which is defined as the point near $(t,x)$ in the set $Z(N,n,K,\beta)$, satisfies $(\hat{t}_0,\hat{x}_0)\in Z(N,n,K+1,\beta) \subset Z(N',n,K+1,\beta)$. We also get from the triangle inequality that $d((\hat{t}_0,\hat{x}_0),(t',x'))\leq 2^{-N'}$. From (\ref{N4-N5-def}) and (\ref{vx3}) we get that (\ref{lm1}) holds with $(\eps_0/6,K+1)$ instead of $(\eps_0,K)$. Therefore, from inequality $N'\geq N_2$ we can apply Proposition \ref{Prop5.13} to $\alpha = \alpha_i \geq \eps_2, \ i=1,...,M$, with $(\hat{t}_0,\hat{x}_0)$ instead of $(t,x)$, $\eps_0/6$ instead of $\eps_0$, and $N'$ instead of $N$. Since $N'\geq N_3$ we may apply Proposition \ref{Prop5.14} with the same parameters. \\
Choose $i\in \{1,...,M\}$ such that
\bn \label{vx4}
&&(i) \textrm{ if } 2^{-N'}>a_n^{\al}, \textrm{ then } a_n^{\alpha_i/2}<2^{-N'}\leq a_n^{\alpha_{i-1}/2}=a_n^{\alpha_i/2}a_n^{-\eps_2/2}, \nonumber \\ \nonumber \\
&&(ii) \textrm{ if } 2^{-N'}\leq a_n^{\al}, \textrm{ then } i=M \textrm{ and so } a_n^{\alpha_i/2}=a_n^{\al} \leq 2^{-N'}.
\en
In either case we have
\bn \label{vx5}
a_n^{\alpha_i/2}\vee 2^{-N'}\leq a_n^{\al}\vee 2^{-N'},
\en
and
\bn \label{vx6}
a_n^{-(1-\eta/4)\alpha_i}(a_n^{\al}\vee 2^{-N'})^{2-\eta/2}\leq a_n^{-(1-\eta/4)\eps_2}.
\en
Apply Proposition \ref{Prop5.13} as described above with $1-\nu_1=\xi$ and use the facts that \\ $d((\hat{t}_0,\hat{x}_0)),(t',x')\leq 2^{-N'}$, $\tilde{\gamma}_m=\gamma_m$ for $m\leq \bar{m}$ and $\gamma_{m+1}=\gamma\gamma_m+\eta/2$. Then use (\ref{vx5}), (\ref{vx6}) to get
\bn \label{vx7}
&&|u_{1,a_n^{\alpha_i}}(\hat{t}_0,\hat{x}_0)-u_{1,a_n^{\alpha_i}}(t',x')| \nonumber \\
 &&\leq 2^{-90}d((\hat{t}_0,\hat{x}_0),(t',x'))^{\xi}a_n^{-\eps_0/2-(1-\eta/4)\alpha_i}\big[a_n^{\beta}a_n^{(1-\eta/4)\alpha_i}
\quad +a_n^{\beta \gamma} (a_n^{\alpha_i/2}\vee 2^{-N'})^{\gamma+1}   \nonumber \\
&&\quad+(a_n^{\alpha_i/2}\vee 2^{-N'})^{\gamma\tilde{\gamma}_m+1}+\mathds{1}_{\{m\geq \bar{m}\}}a_n^{\alpha_i(1-\eta/4)}(a_n^{\alpha_i/2}\vee 2^{-N'})^{\eta}\big].  \nonumber \\
&&\leq 2^{-89}2^{-N'\xi}a_n^{-\eps_0/2}\big[a_n^{\beta}
\quad a_n^{-(1-\eta/4)\alpha_i}(a_n^{\al}\vee 2^{-N'})^{2-\eta/2}a_n^{\beta \gamma} (a_n^{\al}\vee 2^{-N'})^{\gamma-1+\eta/2}   \nonumber \\
&&\quad +a_n^{-(1-\eta/4)\alpha_i}(a_n^{\al}\vee 2^{-N'})^{2-\eta/2}(a_n^{\al}\vee 2^{-N'})^{\gamma\gamma_{m}-1+\eta/2}+\mathds{1}_{\{m\geq \bar{m}\}}(a_n^{\al}\vee 2^{-N'})^{\eta}\big].  \nonumber \\
&&\leq 2^{-89}2^{-N'\xi}a_n^{-\eps_0/2}\big[a_n^{\beta}
+a_n^{-(1-\eta/4)\eps_2}a_n^{\beta \gamma} (a_n^{\al}\vee 2^{-N'})^{\gamma-1+\eta/2}   \nonumber \\
&&\quad+a_n^{-(1-\eta/4)\eps_2}(a_n^{\al}\vee 2^{-N'})^{\gamma\gamma_{m}-1+\eta/2}+\mathds{1}_{\{m\geq \bar{m}\}}(a_n^{\al}\vee 2^{-N'})^{\eta}\big], \ i=1,...,M.  \nonumber \\
\en
Apply Proposition \ref{Prop5.14} with $\alpha=\alpha_i\geq \eps_2$, $(\hat t_0,\hat x_0)$ instead of $(t,x)$, $\eps_0/6$ instead of $\eps_0$, $N'$ instead of $N$, $1-\nu_1=\xi$. Use the facts that $d((\hat{t}_0,\hat{x}_0),(t',x'))\leq 2^{-N'}$, $\tilde{\gamma}_m=\gamma_m$ for $m\leq \bar{m}$ and $\gamma_{m+1}=\gamma\gamma_m+\eta/2$, and use (\ref{vx5}), to get
\bn \label{vx8}
&&|u_{2,a_n^{\alpha_i}}(\hat{t}_0,\hat{x}_0)-u_{2,a_n^{\alpha_i}}(t',x')| \nonumber \\
 &&\leq
2^{-89}\bigg[d((\hat{t}_0,\hat{x}_0),(t',x'))^{\xi/2-(1-\eta)/2}a_n^{-\eps_0/2}2^{-N'\gamma}\big[(a_n^{\al}\vee 2^{-N'})^{\gamma(\tilde\gamma_m-1)}+a_n^{\beta\gamma} \big] \nonumber \\
&&\quad+d((\hat{t}_0,\hat{x}_0),(t',x'))^{\xi}a_n^{-\eps_0/2}\big[a_n^{\frac{\alpha_i}{2}(\gamma \tilde{\gamma}_m-1+\frac{\eta}{2})}+a_n^{\beta\gamma}a_n^{\frac{\alpha_i}{2}(\gamma -1+\frac{\eta}{2})} \big]\bigg] \nonumber \\
&&\leq 2^{-89}a_n^{-\eps_0/2}\bigg[2^{-N'(\xi/2+\gamma-(1-\eta)/2)}\big[(a_n^{\al}\vee 2^{-N'})^{\gamma(\tilde\gamma_m-1)}+a_n^{\beta\gamma} \big] \nonumber \\
&&\quad+2^{-N'\xi}\big[a_n^{\frac{\alpha_i}{2}(\gamma \tilde{\gamma}_m-1+\frac{\eta}{2})}+a_n^{\beta\gamma}a_n^{\frac{\alpha_i}{2}(\gamma -1+\frac{\eta}{2})} \big]\bigg] \nonumber \\
&&\leq 2^{-89}a_n^{-\eps_0/2}\bigg[2^{-N'(\xi/2+1/2)}2^{-N'(\gamma-1+\eta/2)}\big[(a_n^{\al}\vee 2^{-N'})^{\gamma(\gamma_{m}-1)}+a_n^{\beta\gamma} \big] \nonumber \\
&&
\quad+2^{-N'\xi}\big[(a_n^{\al}\vee2^{-N'})^{(\gamma\gamma_m-1+\eta/2)}+a_n^{\beta\gamma}(a_n^{\al}\vee2^{-N'})^{(\gamma -1+\frac{\eta}{2})} \big]\bigg] \nonumber \\
&&\leq 2^{-89}a_n^{-\eps_0/2}2^{-N'\xi}
\bigg[\big[(a_n^{\al}\vee 2^{-N'})^{\gamma\gamma_m-1+\eta/2}+a_n^{\beta\gamma}(a_n^{\al}\vee 2^{-N'})^{\gamma-1+\eta/2} \big] \nonumber \\
&&\quad+\big[(a_n^{\al}\vee2^{-N'})^{\gamma\gamma_{m}-1+\eta/2}+a_n^{\beta\gamma}(a_n^{\al}\vee2^{-N'})^{\gamma -1+\frac{\eta}{2}} \big]\bigg] \nonumber \\
&&\leq 2^{-88}a_n^{-\eps_0/2}2^{-N'\xi}
\big[(a_n^{\al}\vee 2^{-N'})^{\gamma\gamma_m-1+\eta/2}+a_n^{\beta\gamma}(a_n^{\al}\vee 2^{-N'})^{\gamma-1+\eta/2} \big],  \ \ i=1,...,M. \nonumber \en
From (\ref{vx7}), (\ref{vx8}) we get
\bn \label{vx9}
&& |u(\hat{t}_0,\hat{x}_0)-u(t',x')|  \nonumber \\
&&\leq|u_{2,a_n^{\alpha_i}}(\hat{t}_0,\hat{x}_0)-u_{2,a_n^{\alpha_i}}(t',x')|+|u_{1,a_n^{\alpha_i}}(\hat{t}_0,\hat{x}_0)-u_{1,a_n^{\alpha_i}}(t',x')| \nonumber \\
&&\leq 2^{-89}2^{-N'\xi}a_n^{-\eps_0/2}\big[a_n^{\beta}
+a_n^{-(1-\eta/4)\eps_2}a_n^{\beta \gamma} (a_n^{\al}\vee 2^{-N'})^{\gamma-1+\eta/2}   \nonumber \\
&&\quad +a_n^{-(1-\eta/4)\eps_2}(a_n^{\al}\vee 2^{-N'})^{\gamma\gamma_{m}-1+\eta/2}+\mathds{1}_{\{m= \bar{m}\}}(a_n^{\al}\vee 2^{-N'})^{\eta}\big] \nonumber \\
&&\quad +2^{-88}a_n^{-\eps_0/2}2^{-N'\xi}
\big[(a_n^{\al}\vee2^{-N'})^{\gamma\gamma_{m}-1+\eta/2}+a_n^{\beta\gamma}(a_n^{\al}\vee2^{-N'})^{\gamma -1+\frac{\eta}{2}} \big]  \nonumber \\
&&\leq 2^{-87}2^{-N'\xi}a_n^{-\eps_0/2}a_n^{-(1-\eta/4)\eps_2}\big[a_n^{\beta}
\quad +a_n^{\beta \gamma} (a_n^{\al}\vee 2^{-N'})^{\gamma-1+\eta/2}   \nonumber \\
&&\quad+(a_n^{\al}\vee 2^{-N'})^{\gamma\gamma_{m}-1+\eta/2}+\mathds{1}_{\{m= \bar{m}\}}(a_n^{\al}\vee 2^{-N'})^{\eta}\big].  \nonumber \\
\en
Consider $m=\bar{m}$ and $m<\bar{m}$ separately to get
\bn \label{vx10}
(a_n^{\al}\vee 2^{-N'})^{\gamma\gamma_{m}-1+\eta/2}+\mathds{1}_{\{m= \bar{m}\}}(a_n^{\al}\vee 2^{-N'})^{\eta}\leq 2(a_n^{\al}\vee 2^{-N'})^{\tilde{\gamma}_{m+1}-1}.
\en
Recall that $\eps_2\leq \eta \al \eps_0$ and therefore $(1-\eta/4)\eps_2\leq \eps_0/2$. Use this and (\ref{vx10}) on (\ref{vx9}) to get
\bn \label{vx11}
&& |u(\hat{t}_0,\hat{x}_0)-u(t',x')|  \nonumber \\
&&\leq 2^{-86}2^{-N'\xi}a_n^{-\eps_0}\big[a_n^{\beta}
+a_n^{\beta \gamma} (a_n^{\al}\vee 2^{-N'})^{\gamma-1+\eta/2}+(a_n^{\al}\vee 2^{-N'})^{{\tilde{\gamma}_{m+1}-1}}\big]  \nonumber \\
&&\leq2^{-84}2^{-N\xi}a_n^{-\eps_0}\big[a_n^{\beta}
+a_n^{\beta \gamma} (a_n^{\al}\vee 2^{-N})^{\gamma-1+\eta/2}+(a_n^{\al}\vee 2^{-N})^{{\tilde{\gamma}_{m+1}-1}}\big].  \nonumber \\
\en
From (\ref{vx11}) and the facts that $|u(\hat{t}_0,\hat{x}_0)|<a_n^{1-\al}2^{-N}$ and $\gamma_0=1, \ \gamma_1=\gamma+\eta/2$, we get
\bn \label{vx12}
&&|u(t',x')| \nonumber \\
&&\leq a_n^{1-\al}2^{-N}+ 2^{-84}2^{-N\xi}a_n^{-\eps_0}\big[a_n^{\beta}+a_n^{\beta \gamma}(a_n^{\al}\vee 2^{-N})^{\gamma_1-1}+(a_n^{\al}\vee 2^{-N})^{\tilde{\gamma}_{m+1}-1}\big].  \nonumber \\
&&\leq a_n^{-\eps_0}2^{-N\xi}\big[a_n^{1-\al}2^{-N(1-\xi)}+ 2^{-84}\big[a_n^{\beta}+a_n^{\beta \gamma}(a_n^{\al}\vee 2^{-N})^{\gamma_1-1}+(a_n^{\al}\vee 2^{-N})^{\tilde{\gamma}_{m+1}-1}\big]\big].  \nonumber \\
\en
By our choice of $N_1$ and $N_4$ we have $N(1-\xi)\geq 1$. Recall that $\beta\in [0,\frac{\eta}{1+\eta}]=[0,1- \al]$, therefore
\bn \label{vx13}
a_n^{1-\al}2^{-N(1-\xi)}\leq \frac{a_n^{1-\al}}{2}\leq \frac{a_n^{\beta}}{2}.
\en
From Lemma \ref{Lemma5.15} we have
\bn \label{vx14}
a_n^{\beta\gamma}(a_n^{\al}\vee 2^{-N})^{\gamma_1-1}\leq a_n^{\beta}+(2^{-N}\vee a_n^{\al})^{\tilde{\gamma}_{m+1}-1}.
\en
Apply (\ref{vx13}), (\ref{vx14}) to (\ref{vx12}) to get
\bn \label{vx15}
|u(t',x')|&\leq& a_n^{-\eps_0}2^{-N\xi}\big[a_n^{\beta}+(2^{-N}\vee a_n^{\al})^{\tilde{\gamma}_{m+1}-1}\big],
\en
which implies $(P_{m+1})$. \qed

\section{Proof of Proposition \ref{PropStopTimes}} \label{Sec-Proof-Prop3.3}
This section is dedicated to the proof of Proposition \ref{PropStopTimes}. The proof follows the same lines as the proof of Proposition 3.3 in \cite{MP09}. Before we start with the proof we will need the following Proposition.
\bn \label{aaa}
\hat{\Delta}_{1,u_2}(m,n,\eps_0,2^{-N},\eta)&=&2^{-N\gamma(1-\eps_0)}\big[(2^{-N}\vee a_n^{\al})^{\gamma(\bar{\gamma}_m-1)}+a_n^{\gamma\beta}\big], \nonumber \\
\hat{\Delta}_{2,u_2}(m,n,\eps_0,\eta)&=&a_n^{\alpha(1+\eta/2-\eps_0)/2}.
\en
\begin{proposition} \label{Prop5.14mod}
Let $0\leq m \leq \bar{m}+1$ and assume $(P_m)$. Let $\theta\in (0,\gamma-1+\eta/2)$. Then for any $n\in \mathds{N},$ \\ $ \eps_0 \in(0,1), \nu_1\in (0,\theta\wedge \eps_0], K\in \mathds{N}^\mathds{\geq K_1}, \alpha\in[0,2\al]$, and $\beta\in [0,\frac{\eta}{\eta+1}]$, there is an $N_{\ref{Prop5.14mod} }=N_{\ref{Prop5.14mod} }(m,n,\nu_1,\eps_0,K,\alpha,\beta,\eta)$ a.s. such that for all $N\geq N_{\ref{Prop5.14mod}}, (t,x)\in Z(N,n,K,\beta)$,
\bn \label{rtt}
&&d\equiv d((t,x),(t,x'))\leq 2^{-N} \textrm{ implies that } \nonumber \\
&&|\tilde u_{2,a_n^\alpha}(t,a_n^{2\al+2\eps_0},x)- \tilde u_{2,a_n^\alpha}(t,a_n^{2\al+2\eps_0},x')|\nonumber \\
&&\leq 2^{-89}a_n^{-\eps_0}\big[( a_n^{-\al(1-\eta/2)-3\eps_0}d)\wedge d^{\eta/2-\eps_0}\big] \bar{\Delta}_{1,u_2}(m,n,\eps_,2^{-N},\eta)+d^{1-\nu_1}\bar{\Delta}_{2,u_2}(m,n,\eps_,\eta).
\en
Moreover, $N_{\ref{Prop5.14}}$ is stochastically bounded, uniformly in $(n,\alpha,\beta)$.
\end{proposition}
The proof of Proposition \ref{Prop5.14mod} is given in Section \ref{proof-u2}. \medskip \\
In Proposition \ref{Prop-Induction} we established the bound $(P_{\bar{m}+1})$. Therefore, we can use the conclusions of Corollary \ref{corollary-u-tag-reg} and  Propositions \ref{Prop5.11}, \ref{Prop5.14mod} with $\bar{m}+1$. We will use these conclusions to derive the modulus of continuity for $u_{1,a_n^{\alpha}}$ and $\tilde u_{2,a_n^{\alpha}}$.\\\\ We will construct a few sequences of stopping times $\{U^{(i)}_{M,n,\beta_j}\}_{M,n}, \ \ i=1,2,\dots ,5\ ,j=1,2,...$ and, roughly speaking, we will show that their minimum is the required sequence of stopping times $\{U_{M,n}\}_{M,n}$. We fix a $K_0\in \mathds{N}^{\geq K_1}$ and positive constants $\eps_0, \eps_1$ as in (\ref{eps1}). For $0< \beta \leq \frac{\eta}{\eta+1}- \eta\eps_1$, define
\bn \label{alpha-def}
\alpha = \alpha(\beta) = 2(\beta/\eta+\eps_1) \in [0,2\al],
\en
and
\bn \label{U1-def}
U^{(1)}_{M,n,\beta}&=&\inf\bigg\{t: \textrm{ there are } \eps \in [a_n^{\al+\eps_0},2^{-M}],|x|\leq K_0+1, \ \hat{x}_0, x' \in \re,  \textrm{ s.t. } |x-x'|\leq 2^{-M}, \nonumber \\
&& |x- \hat{x}_0| \leq \eps, \ | u(t,\hat{x}_0) |\leq a_n\wedge (a_n^{1-\al}\eps), \ |u'_{1,a^{2\al}_n}(t, \hat{x}_0)|\leq a_n^{\beta}, \nonumber \\
&&\textrm{ and }
\ |u'_{1,a_n^{\alpha}}(t+a_n^{2\al+2\eps_0},x)-u'_{1,a_n^{\alpha}}(t+a_n^{2\al+2\eps_0},x')| >  2^{-76}a_n^{-\eps_0-(2-\eta/2)\eps_1}(|x-x'|\vee a_n^{\al+\eps_0})^{1-\eps_0}\nonumber \\
&&\times\big[a_n^{-(2-\eta/2)\beta/\eta}(\eps\vee |x'-x|)^{(\eta+1)\gamma}
+a_n^{\beta(\eta-1)/\eta}
+a_n^{\beta\gamma-\beta(2-\eta/2)/\eta} (\eps\vee |x'-x|)^{\gamma}\big] \bigg\}\wedge T_{K_0}. \nonumber \\
\en
Define $U^{(1)}_{M,n,\beta}$ by the same expression with $\beta=0$, but without the condition on $|u'_{1,a^{2\al}_n}(t, \hat{x}_0)|$. Just as in Section 6 of \cite{MP09}, $U^{(1)}_{M,n,\beta}$ is an $\mathcal{F}_t$-stopping time.
\begin{lemma} \label{Lemma6.1}
For each $n\in \mathds{N}$ and $\beta$ as in (\ref{beta-up-lim}), $U^{(1)}_{M,n,\beta} \uparrow T_{K_0}$ as $M\uparrow \infty$ and
\bn \label{u1-lem}
\lim_{M\rr \infty } \sup_{n,0\leq \beta\leq \frac{\eta}{\eta+1}-\eta\eps_1} P(U^{(1)}_{M,n,\beta}<T_{K_0})=0.
\en
\end{lemma}
\begin{proof}
From the monotonicity in $M$ and (\ref{u1-lem}) the first assertion is trivial. Let us consider the second assertion. Recall that $n_M$ was defined before (\ref{n1}).
By Proposition \ref{Prop-Induction} we can use Corollary \ref{corollary-u-tag-reg} with $\eps_0/2$ instead of $\eps_0$, $m=\bar{m}+1$, $\nu_1=\eps_0, K=K_0+1$, and $\alpha,\beta,\al$ as in (\ref{alpha-def}), (\ref{u1-lem}) and (\ref{beta-up-lim}) respectively.
Therefore, there exists $N_0(n,\eps_0,\eps_1,K_0+1, \beta) \in \mathds{N}$ a.s., stochastically bounded uniformly in $(n,\beta)$ (where $\beta$ is as in (\ref{beta-up-lim})), such that if
\bn \label{mj0}
N\geq N_0(\omega), \ (t,x)\in Z(N,n,K_0+1,\beta), \ a_n^{\al+\eps_0}+|x-x'|\leq2^{-N},
\en
then,
\bn \label{mj1}
&&|u'_{1,a_n^\alpha}(t+a_n^{2\al+2\eps_0},x)-u'_{1,a_n^\alpha}(t+a_n^{2\al+2\eps_0},x')| \nonumber \\
&&\leq |u'_{1,a_n^\alpha}(t+a_n^{2\al+2\eps_0},x)-u'_{1,a_n^\alpha}(t,x)|
+ |u'_{1,a_n^\alpha}(t+a_n^{2\al+2\eps_0},x')-u'_{1,a_n^\alpha}(t,x)| \nonumber \\
&&\leq 2^{-83}(|x-x'|\vee a_n^{\al+\eps_0})^{1-\eps_0}a_n^{-\eps_0}\big[a_n^{-2(\beta/\eta+\eps_1)(1-\eta/4)}2^{-N\gamma(\eta+1)}+(2^{-N}\vee a_n^{(\beta/\eta+\eps_1)})^{(\eta+1)\gamma-2+\eta/2} \nonumber \\
&&\quad +a_n^{-2(\beta/\eta+\eps_1)(1-\eta/4)+\beta\gamma}(a_n^{\beta/\eta+\eps_1}\vee 2^{-N})^\gamma\big].
\en
Since $\gamma>1-\frac{\eta}{2(\eta+1)}+100\eps_1$ (see (\ref{eps1})) we have
\bn \label{plp1}
(\eta+1)\gamma-2+\eta/2\geq \eta-1.
\en
 and
\bn  \label{mj2}
a_n^{-\beta(2-\eta/2)/\eta+\beta \gamma +(\beta/\eta+\eps_1)\gamma } &\leq& a_n^{\beta(\gamma(\eta+1)/\eta-2/\eta+1/2)} \nonumber \\
&\leq& a_n^{\beta((1+\eta/2)/\eta-2/\eta+1/2)}\nonumber \\
&\leq& a_n^{(\eta-1)\beta/\eta}.
\en
Note that
\bn \label{mj2.3}
a_n^{(\beta/\eta+\eps_1)(\eta-1)}\leq a_n^{\beta(\eta-1)/\eta}a_n^{-(2-\eta/2)\eps_1}.
\en
From (\ref{mj1}), (\ref{plp1}), (\ref{mj2}) and (\ref{mj2.3}) and  we have
\bn \label{mj3}
&&|u'_{1,a_n^\alpha}(t+a_n^{2\al+2\eps_0},x)-u'_{1,a_n^\alpha}(t+a_n^{2\al+2\eps_0},x')| \nonumber \\
&&= 2^{-83}(|x-x'|\vee a_n^{\al+\eps_0})^{1-\eps_0}a_n^{-\eps_0-(2-\eta/2)\eps_1}\big[a_n^{-\beta(2-\eta/2)/\eta}2^{-N\gamma(\eta+1)}+(2^{-N}\vee a_n^{\beta/\eta})^{\eta-1}  \nonumber \\
&&\quad+ a_n^{-\beta(2-\eta/2)/\eta+\beta\gamma} 2^{-\gamma N}\big].
\en
Assume that $\beta>0$ (if $\beta=0$ we can omit the bound on $|u'_{1,a_n^{\alpha}}(t,\hat{x}_0)|$ in what follows). Assume \\ $M\geq  N_0(n,\eps_0,\eps_1,K_0+1,\beta)+2$. Suppose for some $t<T_{K_0}(\leq T_{K_0+1})$ there are $\eps\in[a_n^{\al+\eps_0},2^{-M}], |x|\leq K_0+1, \hat{x}_0,x'\in \re$ satisfying $|x-x'|\leq 2^{-M}, \ |\hat{x}_0-x|\leq \eps, \ |u(t,\hat{x}_0)|\leq a_n \wedge (a_n^{1-\al}\eps)$ and $|u'_{1,a_n^\alpha}(t+a_n^{2\al},\hat{x}_0)|\leq a_n^{\beta}$.
If $x\not=x'$, then
\bn \label{se1}
0 < (|x-x'|+ a_n^{\al+\eps})\vee \eps \leq 2^{-M+1} \leq 2^{-N_0-1}.
\en
By (\ref{se1}) we may choose $N>N_0$ such that $2^{-N-1}<\eps\vee (a_n^{\al+\eps_0}+ |x-x'|)\leq 2^{-N}$. Then (\ref{mj0}) holds and we get from (\ref{mj3})
\bn \label{mj4}
&&|u'_{1,a_n^\alpha}(t+a_n^{2\al+2\eps_0},x)-u'_{1,a_n^\alpha}(t+a_n^{2\al+2\eps_0},x')| \nonumber \\
&&\leq 2^{-78}(|x-x'|\vee a_n^{\al+\eps_0})^{1-\eps_0}a_n^{-\eps_0-(2-\eta/2)\eps_1}\big[a_n^{-\beta(2-\eta/2)/\eta}(|x-x'|\vee a_n^{\al+\eps_0}\vee \eps)^{\gamma(\eta+1)}+a_n^{\beta(\eta-1)/\eta} \nonumber \\
&&\quad+ a_n^{-\beta(2-\eta/2)/\eta+\beta\gamma} (|x-x'|\vee a_n^{\al+\eps_0}\vee \eps)^{\gamma }\big]  \nonumber \\
&&\leq 2^{-76}(|x-x'|\vee a_n^{\al+\eps_0})^{1-\eps_0}a_n^{-\eps_0-(2-\eta/2)\eps_1}\big[a_n^{-\beta(2-\eta/2)/\eta}(|x-x'|\vee \eps)^{\gamma(\eta+1)}+a_n^{\beta(\eta-1)/\eta} \nonumber \\
&&\quad+ a_n^{-\beta(2-\eta/2)/\eta+\beta\gamma} (|x-x'|\vee \eps)^{\gamma }\big].
\en
If $x=x'$ the bound in (\ref{mj4}) is trivial. From (\ref{mj4}) we get that if $M\geq  N_0(n,\eps_0,\eps_1,K_0+1,\beta)+2$, then $U^{(1)}_{M,n,\beta}=T_{K_0}$. Therefore, we have shown that
\bn
P(U^{(1)}_{M,n,\beta}<T_{K_0}) = P(M<N_0+2). \nonumber
\en
This completes the proof because $N_0(n,\eps_0,\eps_1,K_0+1, \beta)$ is stochastically bounded uniformly in $(n,\beta)$, where $\beta$ satisfies (\ref{beta-up-lim}).
\end{proof}
\\\\
Next we define a stopping time related to the increment of $u_{2,a_n^{\alpha}}$. For $0<\beta < \frac{\eta}{\eta+1}-\eps_1$ define
\bn \label{U2-def}
U^{(2)}_{M,n,\beta}&=&\inf\bigg\{t: \textrm{ there are } \eps \in [0,2^{-M}],|x|\leq K_0+1, \ \hat{x}_0, x' \in \re,  \textrm{ s.t. } |x-x'|\leq 2^{-M}, \nonumber \\
&& |x- \hat{x}_0| \leq \eps, \ | u(t,\hat{x}_0) |\leq a_n\wedge (a_n^{1-\al}\eps), \nonumber \\
&& \ |u'_{1,a^{2\al}_n}(t, \hat{x}_0)|\leq a_n^{\beta}, \textrm{ and }
 \ |\tilde u_{2,a_n^{\alpha}}(t,a_n^{2\al+2\eps_0},x)-\tilde u_{2,a_n^{\alpha}}(t,a_n^{2\al+2\eps_0},x')| \nonumber \\
&& > 2^{-87}a_n^{-\eps_0}\bigg[((a_n^{-\al(1-\eta/2)-3\eps_0}|x-x'|)\wedge|x-x'|^{\eta/2-\eps_0})( \eps\vee |x'-x|)^{\gamma} \nonumber \\
&&\quad \times\big[((a_n^{\al}\vee \eps\vee |x'-x|)^{\eta\gamma}
 +a_n^{\beta\gamma}\big]+|x-x'|^{1-\eps_0}a_n^{\beta+99\eps_1 \beta / \eta + \eps_1\eta^2/4}\big)\bigg] \bigg\}\wedge T_{K_0}.\nonumber \\
\en
Define $U^{(2)}_{M,n,0}$ by the same expression with $\beta=0$, but without the condition on $|u'_{1,a^{2\al}_n}(t, \hat{x}_0)|$. Just as in Section 6 of \cite{MP09}, $U^{(2)}_{M,n,\beta}$ is an $\mathcal{F}_t$-stopping time.
\begin{lemma} \label{Lemma6.2}
For each $n\in \mathds{N}$ and $\beta$ as in (\ref{beta-up-lim}), $U^{(2)}_{M,n,\beta} \uparrow T_{K_0}$ as $M\uparrow \infty$ and
\bn
\lim_{M\rr \infty } \sup_{n,0\leq \beta\leq \frac{\eta}{\eta+1}-\eta\eps_1} P(U^{(2)}_{M,n,\beta}<T_{K_0})=0. \nonumber
\en
\end{lemma}
\begin{proof}
As before, we only need to show the second assertion. By Proposition \ref{Prop-Induction} we can use Proposition \ref{Prop5.14mod} with $m=\bar{m}+1$, $\nu_1=\eps_0, K=K_0+1$ and $\alpha,\beta$ as in (\ref{alpha-def}) and (\ref{beta-up-lim}) respectively. Therefore, there exists $N_0(n,\eps_0,\eps_1,K_0+1, \beta) \in \mathds{N}$ a.s., stochastically bounded uniformly in $(n,\beta)$ as in (\ref{beta-up-lim}),
such that if
\bn \label{mj11}
N\geq N_0(\omega), \ (t,x)\in Z(N,n,K_0+1,\beta), \ |x-x'|\leq2^{-N},
\en
\bn \label{mj12}
&&|\tilde u_{2,a_n^\alpha}(t,a_n^{2\al+2\eps_0},x)- \tilde u_{2,a_n^\alpha}(t,a_n^{2\al+2\eps_0},x')| \nonumber \\
&&\leq 2^{-89}a_n^{-\eps_0}\bigg[((|x-x'|a_n^{-\al(1-\eta/2)-3\eps_0})\wedge |x-x'|^{\eta/2-\eps_0})2^{-N\gamma}\big[(2^{-N}\vee a_n^{\al})^{\gamma \eta}+a_n^{\gamma\beta}\big]\nonumber \\
&&\quad +|x-x'|^{1-\eps_0}\big[a_n^{(\beta/\eta+\eps_1)(\gamma(\eta+1)-1+\eta/2)}+a_n^{(\beta/\eta+\eps_1)(\gamma-1+\eta/2)}a_n^{\gamma\beta}\big] \bigg]. \nonumber\\
\en
Use (\ref{eps1}), (\ref{beta-def}) and (\ref{beta-up-lim}), to get the following inequalities,
\bn\label{mj35}
\gamma(\eta+1)-1+\frac\eta 2 -l\eps_0 \geq \eta +99\eps_1, \ \ \forall l=0,1,2.
\en
\bn\label{mj35.5}
\bigg(\frac{\beta}{\eta}+\eps_1\bigg)(\gamma-1+\eta/2-l\eps_0)+\beta\gamma  &\geq & \bigg(\frac{\beta}{\eta}+\eps_1\bigg)\bigg(\frac{\eta^2}{2(\eta+1)}+99\eps_1\bigg)+\beta\gamma \nonumber \\
&\geq&\beta\bigg(\frac{\eta}{2(\eta+1)}+\gamma\bigg)  +99\eps_1\frac \beta \eta  + \eps_1\frac{\eta^2}{4} \nonumber \\
&\geq& \beta+ +99\eps_1\frac \beta \eta  + \eps_1\frac{\eta^2}{4},  \ \ \forall l=0,1.
\en
From (\ref{mj35}), (\ref{mj35.5}) and $\beta \in [0,\frac{\eta}{\eta+1}]$ we have,
\bn \label{mj13}
a_n^{(\beta/\eta+\eps_1)(\gamma(\eta+1)-1+\eta/2)}+a_n^{(\beta/\eta+\eps_1)(\gamma-1+\eta/2)}a_n^{\gamma\beta}\leq 2a_n^{\beta+99\eps_1 \beta / \eta + \eps_1\eta^2/4}
\en
Apply (\ref{mj13}) to (\ref{mj12}) to get
\bn \label{mj14}
&&|u_{2,a_n^\alpha}(t+a_n^{2\al+2\eps_0},x)- u_{2,a_n^\alpha}(t+a_n^{2\al+2\eps_0},x')|  \\
&&\leq 2^{-89}a_n^{-\eps_0}\bigg[((|x-x'|a_n^{-\al(1-\eta/2)-3\eps_0})\wedge |x-x'|^{\eta/2-\eps_0})2^{-\gamma N}\big[(2^{-N}\vee a_n^{\al})^{\eta\gamma}+a_n^{\gamma\beta}\big] \nonumber \\
&&\quad +2|x-x'|^{1-\eps_0}a_n^{\beta+99\eps_1 \beta / \eta + \eps_1\eta^2/4} \bigg]. \nonumber
\en
The rest of the proof is similar to the proof of Lemma \ref{Lemma6.1}, where (\ref{mj14}) is used instead of (\ref{mj4}).
\end{proof}
\paragraph{Notation.}
\bn
\tilde{\Delta}_{u'_1}(n,\eps,\eps_0,\beta,\eta) = a_n^{-\eps_0}\eps^{-\eps_0}\big\{\eps (\eps\vee a_n^{\al})^{\eta-1} +(\eps a_n^{-\al(2-\eta/2)}+a_n^{-1+\eta/4+\al})\big(\eps^{(\eta+1)\gamma}+a_n^{\beta \gamma}(\eps\vee a_n^{\al})^{\gamma}\big)\big\}.\nonumber \\
\en
For $0<\beta < \frac{\eta}{\eta+1}-\eps_1$ define
\bn \label{U3-def}
U^{(3)}_{M,n,\beta}&=&\inf\bigg\{t: \textrm{ there are } \eps \in [2^{-a_n^{-(\beta/\eta+\eps_1)\eps_0/4}},2^{-M}],|x|\leq K_0+1, \ \hat{x}_0 \in \re,  \textrm{ s.t. } |x-\hat{x}_0|\leq \eps, \nonumber \\
&&  \ | u(t,\hat{x}_0) |\leq a_n\wedge (a_n^{1-\al}\eps), \ |u'_{1,a^{2\al}_n}(t, \hat{x}_0)|\leq a_n^{\beta}, \textrm{ and } \nonumber \\
&&\ |u'_{1,a_n^{2\al}}(t,x)-u'_{1,a_n^{\alpha}}(t+a_n^{2\al+2\eps_0},x)| >  2^{-78}(\tilde{\Delta}_{u'_1}(n,\eps,\eps_0,\beta,\eta)+a_n^{\beta+\eps_1\eta^2/4})  \bigg\}\wedge T_{K_0}. \nonumber \\
\en
Define $U^{(3)}_{M,n,\beta}$ by the same expression with $\beta=0$, but without the condition on $|u'_{1,a^{2\al}_n}(t, \hat{x}_0)|$. Just as in Section 6 of \cite{MP09}, $U^{(3)}_{M,n,\beta}$ is an $\mathcal{F}_t$-stopping time.
\begin{lemma} \label{Lemma6.3}
For each $n\in \mathds{N}$ and $\beta$ as in (\ref{beta-up-lim}), $U^{(3)}_{M,n,\beta} \uparrow T_{K_0}$ as $M\uparrow \infty$ and
\bn
\lim_{M\rr \infty } \sup_{n,0\leq \beta\leq \frac{\eta}{\eta+1}-\eta\eps_1} P(U^{(3)}_{M,n,\beta}<T_{K_0})=0. \nonumber
\en
\end{lemma}
\begin{proof}
We only need to show the second assertion. By Proposition \ref{Prop-Induction} we can use Proposition \ref{Prop5.11} with  $m=\bar{m}+1$, $\eps_0/3$ instead of $\eps_0$, $\nu_1=\eps_0, K=K_0+1$ and $\alpha,\beta$ as in (\ref{alpha-def}) and (\ref{beta-up-lim}) respectively. Note that if $s=t-2a_n^{2\alpha+2\eps_0}+a_n^{2\al}$, then
\bn \label{mj31}
\sqrt{t-s}\leq a_n^{\alpha}
=a_n^{\beta/\eta+\eps_1},
\en
and
\bn \label{mj32}
a_n^{\beta/\eta+\eps_1}\leq N^{-4/\eps_0} \Leftrightarrow 2^{-N}> 2^{-a_n^{-(\beta/\eta+\eps_1)\eps_0/4}}.
\en
Therefore, by Proposition \ref{Prop5.11}, (\ref{u-tag-F-an}) and (\ref{mj31}), there exists $N_0(n,\eps_0,\eps_1,K_0+1, \beta) \in \mathds{N}$ a.s., stochastically bounded uniformly in $(n,\beta)$ as in (\ref{beta-up-lim}),
such that if
\bn \label{mj33}
N\geq N_0(\omega), \ (t,x)\in Z(N,n,K_0+1,\beta), \ 2^{-N}\geq 2^{-a_n^{-(\beta+\eps_1)\eps_0/4}},
\en
\bn \label{mj34}
&&|u'_{1,a_n^{2\al}}(t,x)-u'_{1,a_n^{\alpha}}(t+a_n^{2\al+2\eps_0},x)|  \nonumber  \\
&&=|F_{a_n^{2\al}}(t,t,x)-F_{a_n^{2\al}}(t-a_n^{\alpha}+a_n^{2\al+2\eps_0}+a_n^{2\al},t+a_n^{2\al+2\eps_0},x)|  \nonumber  \\
&&\leq 2^{-81}a_n^{-\eps_0}\bigg\{2^{-N(1-\eps_0)}( a_n^{\al}\vee 2^{-N})^{\eta-1}  \nonumber  \\ &&\quad +2^{N\eps_0}a_n^{-1+\eta/4+\al}\bigg(\frac{2^{-N}}{a_n^{3\al-\al\eta/2-1+\eta/4}}+1\bigg)\big(2^{-N\gamma(\eta+1)}+a_n^{\beta\gamma}
(a_n^{\al}\vee2^{-N})^{\gamma}\big)\nonumber  \\
&&\quad +a_n^{(\beta/\eta+\eps_1)(1-\eps_0)}\big((a_n^{\beta/\eta+\eps_1}\vee a_n^{\al})^{\gamma(\eta+1)-2+\eta/2}+a_n^{\beta\gamma}(a_n^{\beta/\eta+\eps_1}\vee a_n^{\al})^{\gamma-2+\eta/2}\big)\bigg\}.
\en
From (\ref{alpha_0}) and our constrains on $\beta$ we have $\beta/\eta+\eps_1\leq \al$. Use the last inequality with  (\ref{mj35}), (\ref{mj35.5}) and (\ref{mj34}) to get
\bn \label{mj38}
&&|u'_{1,a_n^{2\al}}(t,x)-u'_{1,a_n^\alpha}(t+a_n^{2\al+2\eps_0},x)|  \nonumber  \\
&&\leq 2^{-81}\bigg\{a_n^{-\eps_0}2^{N\eps_0}\bigg[2^{-N}( a_n^{\al}\vee 2^{-N})^{\eta-1}  +a_n^{-1+\eta/4+\al}\bigg(\frac{2^{-N}}{a_n^{3\al-\al\eta/2-1+\eta/4}}+1\bigg) \nonumber  \\
&&\quad\times \big(2^{-N\gamma(\eta+1)}+a_n^{\beta\gamma}
( a_n^{\al}\vee2^{-N})^{\gamma}\bigg]
+a_n^{\beta+\eps_1\eta^{2}/4}\bigg\} \nonumber  \\
&&\leq 2^{-81}\big[\tilde{\Delta}_{u'_1}(n,2^{-N},\eps_0,\beta,\eta)+a_n^{\beta+\eps_1\eta^2/4}\big].
\en
The rest of the proof follows that same lines as the proof of Lemma \ref{Lemma6.1}. Assume that $\beta>0$ (if $\beta=0$ we can omit the bound $|u'_{1,a_n^{\alpha}}(t,\hat{x}_0)|$ in what follows). Assume $M\geq  N_0(n,\eps_0,\eps_1,K_0+1,\beta)$. Suppose for some $t<T_{K_0}$ there are  $\eps \in [2^{-a_n^{-(\beta+\eps_1)\eps_0/4}},2^{-M}]$, $|x|\leq K_0+1, \hat{x}_0\in \re$ satisfying $|x-\hat{x}_0|\leq \eps, |u(t,\hat{x}_0)|\leq a_n \wedge (a_n^{1-\al}\eps)$ and $|u'_{1,a_n^{2\al}}(t,\hat{x}_0)|\leq a_n^{\beta}$.
We may choose $N\geq M \geq N_0(\omega)$ such that $2^{-N-1}<\eps \leq 2^{-N}$. Then (\ref{mj33}) holds and we get from (\ref{mj34}) and the fact that $\tilde{\Delta}_{u'_1}(n,2\eps,\eps_0,\beta,\eta)\leq \frac{1}{16}\tilde{\Delta}_{u'_1}(n,\eps,\eps_0,\beta,\eta)$,
\bn \label{mj39}
|u'_{1,a^{2\al}_n}(t,x)-u'_{1,a_n^{\alpha}}(t+a_n^{2\al+2\eps_0},x)| \leq 2^{-77}\big[\tilde{\Delta}_{u'_1}(n,\eps,\eps_0,\beta,\eta)+a_n^{\beta+\eps_1\eta^2/4}\big].
\en
From (\ref{mj39}) we get that if $M\geq  N_0(n,\eps_0,\eps_1,K_0+1,\beta)$,  $U^{(1)}_{M,n,\beta}=T_{K_0}$. We have shown that
\bn
\sup_{n,0\leq \beta \leq 1/2 -\eps_1}P(U^{(3)}_{M,n,\beta}<T_{K_0}) \leq  \sup_{n,0\leq \beta \leq 1/2 -\eps_1} P(N_0>M). \nonumber
\en
This completes the proof because $N_0(n,\eps_0,\eps_1,K_0+1, \beta)$ is stochastically bounded uniformly in $(n,\beta)$, where $\beta$ satisfies (\ref{beta-up-lim}).
\end{proof}
\\\\
For $0<\beta < \frac{\eta}{\eta+1}-\eps_1$ define
\bn \label{U5-def}
U^{(4)}_{M,n,\beta}&=&\inf\bigg\{t: \textrm{ there are } \eps \in [0,2^{-M}],|x|\leq K_0+1, \ \hat{x}_0 \in \re,  \textrm{ s.t. } |x-x'|\leq 2^{-M}, \nonumber \\
&& |x- \hat{x}_0| \leq \eps, \ | u(t, \hat{x}_0) |\leq a_n\wedge (a_n^{1-\al}\eps), \ |u'_{1,a^{2\al}_n}(t, \hat{x}_0)|\leq a_n^{\beta}, \nonumber \\
&&\textrm{ and }
|u(s,x')|>2a_n^{-\eps_0/8}\eps^{1-\eps_0/8}\big[(a_n^{\al}\vee\eps)^{\eta}+a_n^{\beta}\big]  \bigg\}\wedge T_{K_0}.
\en
Define $U^{(4)}_{M,n,0}$ by the same expression with $\beta=0$, but without the condition on $|u'_{1,a^{2\al+2\eps_0}_n}(t, \hat{x}_0)|$. Just as in Section 6 of \cite{MP09}, $U^{(4)}_{M,n,\beta}$ is an $\mathcal{F}_t$-stopping time.
\begin{lemma} \label{Lemma601}
For each $n\in \mathds{N}$ and $\beta$ as in (\ref{beta-up-lim}), $U^{(4)}_{M,n,\beta} \uparrow T_{K_0}$ as $M\uparrow \infty$ and
\bn \label{u5-lem}
\lim_{M\rr \infty } \sup_{n,0\leq \beta\leq \frac{\eta}{\eta+1}-\eta\eps_1} P(U^{(4)}_{M,n,\beta}<T_{K_0})=0.
\en
\end{lemma}
\begin{proof}
From the monotonicity in $M$ and (\ref{u1-lem}) the first assertion is trivial. Let us consider the second assertion. By Proposition \ref{Prop-Induction} with $\eps_0/8$ instead of $\eps_0$, $1-\eps_0/8$ instead of $\xi$, $m=\bar{m}+1$, $ K=K_0+1$, and $\beta,\alpha$ as in (\ref{beta-up-lim}) and (\ref{alpha-def}), respectively. Therefore, there exists $N_1(n,\eps_0/8,\eps_1,K_0+1, \beta,\eta) \in \mathds{N}$ a.s., stochastically bounded uniformly in $(n,\beta)$ as in (\ref{beta-up-lim}), such that if
\bn \label{rj0}
N\geq N_1(\omega), \ (t,x)\in Z(N,n,K_0+1,\beta),
\en
then,
\bn \label{rj1}
|u(t,x')| \leq a_n^{-\eps_0/8} 2^{-N(1-\eps_0/8)}\big[(a_n^{\al}\vee 2^{-N}])^{\eta}+a_n^{\beta}\big], \  \forall |x-x'|\leq 2^{-N}.
\en
Assume that $\beta>0$ (if $\beta=0$ we can omit the bound $|u'_{1,a_n^{\alpha}}(t,\hat{x}_0)|$ in what follows). Assume \\ $M\geq  N_1(n,\eps_0/8,\eps_1,K_0+1,\beta,\eta)$. Suppose for some $t<T_{K_0}(\leq T_{K_0+1})$ there are $\eps\in[0,2^{-M}], |x|\leq K_0+1, \hat{x}_0\in \re$ satisfying $|x-\hat{x}_0|\leq \eps, |u(t,\hat{x}_0)|\leq a_n \wedge (a_n^{1-\al}\eps)$ and $|u'_{1,a^{2\al}_n}(t,\hat{x}_0)|\leq a_n^{\beta}$.
We may choose $N\geq N_1$ such that $2^{-N-1}<\eps\leq 2^{-N}$. Then (\ref{rj0}) holds and we get from (\ref{rj1})
\bn \label{rmj1}
|u(t,x')| \leq 2a_n^{-\eps_0/8} \eps^{1-\eps_0/8}\big[(a_n^{\al}\vee \eps)^{\eta}+a_n^{\beta}\big], \  \forall |x-x'|\leq 2^{-N}.
\en
We have shown that
\bn
P(U^{(4)}_{M,n,\beta}<T_{K_0}) = P(M<N_1). \nonumber
\en
This completes the proof because $N_1(n,\eps_0/8,\eps_1,K_0+1, \beta)$ is stochastically bounded uniformly in $(n,\beta)$, where $\beta$ satisfies (\ref{beta-up-lim}).
\end{proof} \\\\
Let
\bn
U_{M,n,\beta} = \wedge_{j=1}^{4}U^{(j)}_{M,n,\beta},
\en
and
\bn
U_{M,n} = \wedge_{i=0}^{L(\eps_0,\eps_1)}U_{M,n,\beta_i},  
\en
where $\{\beta_i\}_{i=0}^{L(\eps_0,\eps_1)}$ are defined in (\ref{beta-def}). Recall that $U_{M,n}$ is dependent in the fixed values of $K_0,\eps_0,\eps_1$. Note that $\beta_i\in [0,\frac{\eta}{\eta+1}-\eps_1]$, for $i=0,...,L$ by (\ref{beta-up-lim}) and $\alpha_i=\alpha(\beta_i)$ by (\ref{alpha-def}). By Lemmas \ref{Lemma6.1}--\ref{Lemma601}, $\{U_{M,n}\}$ satisfy (\ref{H1}) in Proposition \ref{PropStop}. \\\\
To complete the proof of Proposition \ref{PropStopTimes}, we need to prove that the sets $\tilde{J}_{n,i}(s)$ are compact for all $s\geq 0$ and to show that $\tilde{J}_{n,i}(s)$ contains $J_{n,i}(s)$, for all $0\leq s \leq U_{M,n}$ and $i=0,...,L$.
In the next lemmas we will prove the inclusion $J_{n,i}(s)\subset\tilde{J}_{n,i}(s)$. We assume that $M,n$ satisfy (\ref{nM}) throughout the rest of the section. \medskip \\
We will also need the following lemma. Recall that $n_1$ is defined in (\ref{n1}).
\begin{lemma}  \label{Lemma-cont-u}
Let $ s\in  [0,T_K]$ and $x\in \re$. Assume that $n\geq n_1(\eps_0,K)$. If
\be \label{mol-hyp}
| \< u(s,\cdot),G_{a_n^{2\al+2\eps_0}}(x-\cdot) \> |\leq \frac{a_n}{2},
\ee
then,
\bd
|u(s,\hat{x}_n(s,x))|\leq a_n.
\ed
\end{lemma}
\begin{proof}
Let $(s,x)\in [0,T_K] \times \re $. Suppose that (\ref{mol-hyp}) is satisfied for some $n>n_1(\eps_0,K)$.
By a simple change of variable we get
\bn \label{ex2}
\int_{-\infty}^{x-a_n^{\al}}e^{|y|}G_{a_n^{2\al+2\eps_0}}(x-y)dy+\int_{x+a_n^{\al}}^{\infty}e^{|y|}G_{a_n^{2\al+2\eps_0}}(x-y)dy
\leq 2\int_{a_n^{-\eps_0}}^{\infty}e^{|y|}G_1(y)dy .
\en
Assume that $|u(s,\hat{x}_n(s,x))|> a_n$.
From (\ref{ex2}) and the continuity of $u$ and our choice of $s\leq T_K$ follows that
\bn
\bigg|\int_{\re}u(s,x)G_{a_n^{2\al+2\eps_0}}(x-y)dy\bigg|&\geq & a_n\int_{x-a_n^{\al}}^{x+a_n^{\al}}G_{a_n^{2\al+2\eps_0}}(x-y)dy \nonumber \\
&&-K\int_{-\infty}^{x-a_n^{\al}}e^{|y|}G_{a_n^{2\al+2\eps_0}}(x-y)dy \nonumber \\
&&-K\int_{x+a_n^{\al}}^{\infty}e^{|y|}G_{a_n^{2\al+2\eps_0}}(x-y)dy \nonumber \\
&\geq & a_n\int_{-a_n^{-\eps_0}}^{a_n^{-\eps_0}}G_{1}(y)dy
-2K\int_{a_n^{-\eps_0}}^{\infty}e^{|y|}G_{1}(y)dy \nonumber \\
&> & \frac{a_n}{2}.
\en
where the last inequality follows from the our assumption that $n>n_1$. We get the contradiction with (\ref{mol-hyp}) and the result follows.
\end{proof}
\paragraph{Notation.} Denote by
\bn
n_{2}(\eps_0,\al,\gamma,\eta,K,R_1)=\inf{\bigg\{n\in \mathds{N}:\frac{2C_{(\ref{mu1cond})}(1,2R_1,K,\eta,\eps_0/4) }{KC_{\ref{LemmaNewBound2}}(\eta,\eps_0,2R_1+1)e^{(2R_1+1)K_1}}\geq a_n^{-2\gamma(\al+\frac{\eta}{\eta+1})+\al(1-\eta)+2\eps_0}e^{-\frac{a_n^{-2\eps_0}}{32}}\bigg\}}. \nonumber
\en
\begin{lemma} \label{Lemma666}
If $i\in\{0,...,L\}, \ 0\leq s \leq U_{M,n}$, and $x\in J_{n,i}(s)$, then for all $n\geq n_2(\eps_0,\al,\gamma,\eta,K,K_1)$.
\bn
&&\int_{\re}e^{2R_1|y|}|u(s,y)|^{2\gamma}G_{a_n^{2\al+2\eps_0}}(x-y)\mu(dy) \leq 8C_{(\ref{mu1cond})}(1,2R_1,K,\eta,\eps_0/4) e^{2K_1}a_n^{2\gamma(\al+\beta_i)-\al(1-\eta)-2\eps_0}. \nonumber
\en
\end{lemma}
\begin{proof}
Assume $(n,i,s,x)$ are as above and set $\eps= a_n^{\al}$. We have $|\langle u_s,G_{a_n^{2\al+2\eps_0}}(x-\cdot) \rangle|\leq a_n/2$. By Lemma \ref{Lemma-cont-u} we get
\bn \label{lmj42}
|u(s,\hat{x}_n(s,x))| \leq a_n\leq a_n \wedge (a_n^{1-\al}\eps),  \ |\hat{x}_n(s,x)-x|\leq \eps.
\en
The definition of $J_{n,i}$ implies
\bn \label{lmj425}
|u'_{1,a^{2\al}_n}(s,\hat{x}_n(s,x))|\leq a_n^{\beta_{i+1}}/4.
\en
From (\ref{eps1}) and (\ref{alpha_0}) we note that $\eps_1<\al$. Use this fact and (\ref{nM}) to get
\bn \label{lmj43}
a_n^{\al}\leq 2^{-M}.
\en
Use (\ref{lmj42}) -- (\ref{lmj43}), with $|\hat{x}_n(s,x)|\leq K_0+1$ and $s<U_{M,n}\leq U^{(4)}_{M,n,\beta_i}$, and take $x=\hat x_0=\hat{x}_n(s,x)$ in the definition of $U^{(4)}_{M,n,\beta_i}$ to get
\bn \label{lmj44}
|u(s,y)| &\leq& 2a_n^{-\eps_0/8} a_n^{\al(1-\eps_0/8)}\big[a_n^{\eta\al}+a_n^{\beta_i}\big] \nonumber \\
&\leq & 4a_n^{-\eps_0/8}  a_n^{\al(1-\eps_0/8)+\beta_i} \nonumber \\
&\leq & 4 a_n^{\al + \beta_i -\eps_0/4}
, \forall y \in[x-a_n^{\al},x+a_n^{\al}],
\en
where we have used (\ref{alpha_0}) and (\ref{beta-up-lim}) to get $\beta_i\leq \al\eta$ in the second inequity. \medskip   \\
From Lemma \ref{condmu}(a) with $r=1$, $\lam=2R_1$, $s=0$, $t=a_n^{2\al+2\eps_0}$ and $\varpi=\eps_0/4$ we have
\bn \label{LemmaNewBound22}
\int_{\re} e^{2R_1|y|}G_{a_n^{2\al+2\eps_0}}(x-y)\mu(dy)
&\leq & C_{(\ref{mu1cond})}(1,2R_1,K,\eta,\eps_0/4) a_n^{-(\al+\eps_0)(1-\eta+\eps_0/4)}\nonumber \\
&\leq& C_{(\ref{mu1cond})}(1,2R_1,K,\eta,\eps_0/4) a_n^{-\al(1-\eta)-5\eps_0/4},  \ \forall  x\in [-K_1,K_1].
\en
From Lemma \ref{LemmaNewBound2} with $t=a_n^{2\al+2\eps_0}$ and $\nu_1=\eps_0$ we have
\bn \label{LemmaNewBound23}
&&\int_{\re}e^{(2R_1+1)|y|}G_{a_n^{2\al+2\eps_0}}(x-y)\mathds{1}_{\{|x-y|\geq a_n^{\al}\}}\mu(dy) \nonumber \\
 &&\leq   e^{(2R_1+1)|x|}\int_{\re}e^{(2R_1+1)|y-x|}G_{a_n^{2\al+2\eps_0}}(x-y)\mathds{1}_{\{|x-y|\geq a_n^{\al}\}}\mu(dy)  \nonumber \\ && \leq C_{\ref{LemmaNewBound2}}(\eta,\eps_0,2R_1+1)e^{(2R_1+1)K_1}e^{\frac{a_n^{2\al+2\eps_0}}{4}}e^{-\frac{a_n^{-2\eps_0}}{32}}  \nonumber \\
&&\leq  2C_{\ref{LemmaNewBound2}}(\eta,\eps_0,2R_1+1)e^{(2R_1+1)K_1}e^{-\frac{a_n^{-2\eps_0}}{32}}, \ \forall  x\in [-K_1,K_1].
\en
From (\ref{lmj44})--(\ref{LemmaNewBound23})
and the fact that $(x,s)\in [-K_1,K_1]\times[0,T_K]$ we have
\bn
\int_{\re}e^{2R_1|y|}|u(s,y)|^{2\gamma}G_{a_n^{2\al+2\eps_0}}(x-y)\mu(dy)
&=&\int_{x-a_n^{\al}}^{x+a_n^{\al}}e^{2R_1|y|}|u(s,y)|^{2\gamma}G_{a_n^{2\al+2\eps_0}}(x-y)\mu(dy)  \\
&&+\int_{\re}e^{2R_1|y|}|u(s,y)|^{2\gamma}G_{a_n^{2\al+2\eps_0}}(x-y)\mathds{1}_{\{|x-y|\geq a_n^{\al}\}}\mu(dy) \nonumber \\
&\leq& a_n^{2\gamma(\al+\beta_i)-\eps_0/2} \int_{x-a_n^{\al}}^{x+a_n^{\al}}e^{2R_1|y|}G_{a_n^{2\al+2\eps_0}}(x-y)\mu(dy) \nonumber \\
&&+K\int_{\re}e^{|y|}G_{a_n^{2\al+2\eps_0}}(x-y)\mathds{1}_{\{|x-y|\geq a_n^{\al}\}}\mu(dy) \nonumber \\
&\leq & 4C_{(\ref{mu1cond})}(1,2R_1,K,\eta,\eps_0/4)a_n^{2\gamma(\al+\beta_i)-\eps_0/2}a_n^{-\al(1-\eta)-5\eps_0/4}  \nonumber \\
&&+2KC_{\ref{LemmaNewBound2}}(\eta,\eps_0,2R_1+1)e^{(2R_1+1)K_1}e^{-\frac{a_n^{-2\eps_0}}{32}} \nonumber \\
&\leq &  8C_{(\ref{mu1cond})}(1,2R_1,K,\eta,\eps_0/4)a_n^{2\gamma(\al+\beta_i)-\al(1-\eta)-2\eps_0}, \ \forall i=1,...,L,   \nonumber
\en
where the last inequality follows from our choice $n_2$ and $(\ref{beta-up-lim})$.
\end{proof}
\begin{lemma} \label{Lemma6.5}
If $i\in\{0,...,L\}, \ 0\leq s \leq U_{M,n}$, and $x\in J_{n,i}(s)$, then
\begin{itemize}
  \item [\bf{(a)}] $|u'_{1,a_n^{2\al}}(s,\hat{x}_n(s,x))-u'_{1,a_n^{\alpha_i}}(s+a_n^{2\al+2\eps_0},\hat{x}_n(s,x))|\leq 2^{-73}a_n^{\beta_i+\eps_1\eta^2/4}$,
\item [\bf{(b)}] for $i>0$, $|u'_{1,a_n^{\alpha_i}}(s+a_n^{2\al+2\eps_0},\hat{x}_n(s,x))|\leq a_n^{\beta_i}/2$,
\item [\bf{(c)}] for $i<L$, $u'_{1,a_n^{\alpha_i}}(s+a_n^{2\al+2\eps_0},\hat{x}_n(s,x))>a_n^{\beta_{i+1}}/8$.
\end{itemize}
\end{lemma}
\begin{proof}
Assume $(n,i,s,x)$ are as above and set $\eps= a_n^{\al}$. We have $|\langle u_s,G_{a_n^{2\al+2\eps_0}}(x-\cdot) \rangle|\leq a_n/2$. By Lemma \ref{Lemma-cont-u} we get
\bn \label{mj42}
|u(s,\hat{x}_n(s,x))| \leq a_n\leq a_n\wedge(a_n^{1-\al}\eps),  \ |\hat{x}_n(s,x)-x|\leq \eps.
\en
The definition of $J_{n,i}$ implies
\bn \label{mj425}
|u'_{1,a^{2\al}_n}(s,\hat{x}_n(s,x))|\leq a_n^{\beta_{i+1}}/4.
\en
From (\ref{eps1}) and (\ref{alpha_0}) we note that $\eps_1<\al$. Use this fact and (\ref{nM}) to get
\bn \label{mj43}
2^{-M}\geq a_n^{\al+\eps_0}=\eps\geq 2^{-a_n^{-\eps_0\eps_1/4}}.
\en
Use (\ref{mj42}),(\ref{mj43}), with $|\hat{x}_n(s,x)|\leq K_0+1$ and $s<U_{M,n}\leq U^{(3)}_{M,n,\beta_i}$, and take $x=\hat x_0=\hat{x}_n(s,x)$ in the definition of $U^{(3)}_{M,n,\beta_i}$ to get
\bn \label{mj44}
&&|u'_{1,a^{2\al}_n}(s,\hat{x}_n(s,x))-u'_{1,a_n^{\alpha_i}}(s+a_n^{2\al+2\eps_0},\hat{x}_n(s,x))| \nonumber \\
 &&\leq  2^{-78}(\tilde{\Delta}_{u'_1}(n,a_n^{\al},\eps_0,\beta_i,\eta)+a_n^{\beta_i+\eps_1\eta^2/4}).
\en
Recall that $\eta\in(0,1),\ \al =\frac{1}{\eta+1}$ are fixed. From (\ref{beta-def}) we have
\bn \label{ne12.1}
\beta_i+\eps_1\eta \leq \eta\al, \ i=0,...,L+1.
\en
Note that for $\eta\in(0,1)$,
\bn \label{ne13.2}
\frac{1}{\eta+1}-1+\frac{\eta}{4}+\bigg(1-\frac{\eta}{2(\eta+1)}\bigg)\geq \frac{\eta}{\eta+1}.
\en
Recall that $\gamma\leq 1$. Use (\ref{alpha_0}), (\ref{eps1}), (\ref{beta-def}) and (\ref{beta-up-lim}) to get
\bn  \label{ne13.3}
&&\al-1+\frac{\eta}{4}+\gamma\big(\al+\beta_i\big)-\beta_i \nonumber \\
&&\geq  \frac{1}{\eta+1}-1+\frac{\eta}{4}+\bigg(1-\frac{\eta}{2(\eta+1)}+100\eps_1\bigg)\big(\frac{1}{\eta+1}+\beta_i\big)-\beta_i \nonumber \\
&&\geq \frac{1}{\eta+1}-1+\frac{\eta}{4}+\bigg(1-\frac{\eta}{2(\eta+1)}\bigg)\big(\frac{1}{\eta+1}+\frac{\eta}{\eta+1}\big)
-\frac{\eta}{\eta+1}+50\eps_1  \nonumber \\ \nonumber
&&= \frac{1}{\eta+1}-1+\frac{\eta}{4}+\bigg(1-\frac{\eta}{2(\eta+1)}\bigg)
-\frac{\eta}{\eta+1}+50\eps_1 \\
&&\geq 50\eps_1,
\en
Recall that $\gamma\leq 1$. Use (\ref{alpha_0}), (\ref{eps1}), (\ref{beta-def}) and (\ref{beta-up-lim}) to get
\bn  \label{ne13.4}
&&-\al+\al\frac{\eta}{2}+\gamma\big(\al+\beta_i\big)-\beta_i \nonumber \\
&&\geq  -\frac{1}{\eta+1}+\frac{\eta}{2(\eta+1)}+\bigg(1-\frac{\eta}{2(\eta+1)}+100\eps_1\bigg)\big(\frac{1}{\eta+1}+\beta_i\big)-\beta_i \nonumber \\
&&\geq -\frac{1}{\eta+1}+\frac{\eta}{2(\eta+1)}+\bigg(1-\frac{\eta}{2(\eta+1)}\bigg)\big(\frac{1}{\eta+1}+\frac{\eta}{\eta+1}\big)
-\frac{\eta}{\eta+1}+50\eps_1  \nonumber \\ \nonumber
&&=  -\frac{1}{\eta+1}+\frac{\eta}{2(\eta+1)}+\bigg(1-\frac{\eta}{2(\eta+1)}\bigg)
-\frac{\eta}{\eta+1}+50\eps_1 \\
&&\geq 50\eps_1.
\en
From (\ref{ne13.3}) and (\ref{ne13.4}) we immediately get
\bn  \label{ne13}
\al-1+\frac{\eta}{4}+\gamma\big(\al+\beta_i\big)&\geq& \beta_i+\eps_1,
\en
and
\bn  \label{ne13.1}
-\al\big(1-\frac{\eta}{2}\big)+\gamma\big(\al+\beta_i\big)&\geq& \beta_i+\eps_1,
\en
From (\ref{ne12.1}), (\ref{ne13}), (\ref{ne13.1}) and (\ref{eps1}) we get,
\bn \label{mj45}
&&\tilde{\Delta}_{u'_1}(n,a_n^{\al},\eps_0,\beta_i,\eta)\\ \nonumber
&&= a_n^{-\eps_0(1+\al)}\big\{a_n^{\eta\al}+(a_n^{\al-1+\eta/4}+a_n^{-\al(1-\eta/2)})\big(a_n^{\al(\eta+1)\gamma}+a_n^{\beta_i \gamma}a_n^{\gamma\al}\big)\big\}\\ \nonumber
&&\leq 2a_n^{-\eps_0(1+\al)}\big\{a_n^{\eta\al}+(a_n^{\al-1+\eta/4}+a_n^{-\al(1-\eta/2)})a_n^{\gamma(\al+\beta_i)}\big\} \\ \nonumber
&& \leq 4\big\{a_n^{\beta_i+\eps_1\eta-2\eps_0}+a_n^{\beta_i+\eps_1/2}\big\} \\ \nonumber
&& \leq 8a_n^{\beta_i+\eps_1\eta/2}.
\en
From (\ref{mj44}) and (\ref{mj45}) to get
\bn
&&|u'_{1,a_n^{2\al+2\eps_0}}(s,\hat{x}_n(s,x))-u'_{1,a_n^{\alpha_i}}(s+a_n^{2\al+2\eps_0},\hat{x}_n(s,x))|  \\ \nonumber
&&\leq 2^{-73}a_n^{\beta_i+\eps_1\eta^2/4}, \nonumber
\en
and we are done with claim (a) of the lemma.
(b) is immediate from (a) and the fact that \\ $|u'_{1,a_n^{2\al}}(s,\hat{x}_n(s,x))|\leq a_n^{\beta_i}/4$ (by the definition of $J_{n,i}$ for $i>0$). \\
(c) Since $\eps_0\leq \eps_1\eta^2/100$ by (\ref{eps1}), $a_n^{\beta_i+\eps_1\eta^2/4}\leq a_n^{\beta_{i+1}}$ by (\ref{beta-def}). For $i\leq L$ we have by the definition of $J_{n,i}$, $u'_{1,a_n^{2\al}}(s,\hat{x}_n(s,x))\geq a_n^{\beta_{i+1}}$. Then (c) follows from (a) and the  triangle inequality.
\end{proof}
\begin{lemma} \label{Lemma6.6}
If $i\in\{0,...,L\}, \ 0\leq s \leq U_{M,n}$, $x\in J_{n,i}(s)$, and $|x-x'|\leq 5\bar{l}_n(\beta_i)$, then
\begin{itemize}
  \item [\bf{(a)}] for $i>0$, $|u'_{1,a_n^{\alpha_i}}(s+a_n^{2\al+2\eps_0},x')|\leq a_n^{\beta_i}$,
  \item [\bf{(b)}] for $i<L$, $u'_{1,a_n^{\alpha_i}}(s+a_n^{2\al+2\eps_0},x')>a_n^{\beta_{i+1}}/16$.
\end{itemize}
\end{lemma}
\begin{proof}
Let $(n,i,s,x,x')$ as above and set $\eps=|x-x'|+a_n^{\al}$. Then from (\ref{beta-up-lim}) and (\ref{nM}) we have
\bn \label{mj51}
\eps\leq 5\bar{l}_n(\beta_i)+ a_n^{\al} = 5a_n^{\beta_i/\eta+5\eps_1}+a_n^{\al} \leq5a_n^{5\eps_1}+a_n^{\al} \leq 2^{-M}.
\en
From Definition \ref{Xn-def} we have
\bn \label{mj52}
|\hat{x}_n(s,x)-x'|\leq a_n^{\al}+|x-x'|\leq  \eps, \ |\hat{x}_n(s,x)|\leq |x|+1\leq K_0+1.
\en
By (\ref{mj52}), $s<U_{M,n}\leq U^{(1)}_{M,n,\beta_i}$ and the definition of $U^{(1)}_{M,n,\beta_i}$, with $\hat{x}_n(s,x)$ in the role of $x$ we conclude that
\bn\label{mj53}
&&|u'_{1,a_n^{\alpha_i}}(s+a_n^{2\al+2\eps_0},x')-u'_{1,a_n^{\alpha_i}}(s+a_n^{2\al+2\eps_0},\hat{x}_n(s,x))| \leq  2^{-81}a_n^{-\eps_0-(2-\eta/2)\eps_1}(|x-x'|+a_n^{\al})^{1-\eps_0}\nonumber \\
&&\quad \times\big[a_n^{-(2-\eta/2)\beta_i/\eta}(|x'-x|+a_n^{\al})^{(\eta+1)\gamma}
+a_n^{\beta_i(\eta-1)/\eta}+a_n^{\beta_i(\gamma\eta-2+\eta/2)/\eta} ( |x'-x|+a_n^{\al})^{\gamma}\big].
\en
From (\ref{eps1}) and (\ref{beta-def}) we have
\bn \label{mj56}
\beta_{i+1}=\beta_i+\eps_0\leq \beta_i -2\eps_0-(2-\eta/2)\eps_1+5\eps_1(1-\eps_0), \ i=0,...,L.
\en
Use  (\ref{alpha_0}) and (\ref{beta-def}) to deduce $\beta_i/\eta+5\eps_1\leq \al$, and therefore
\bn \label{mj503}
|x-x'|+a_n^{\al}\leq 6a_n^{\beta_i/\eta+5\eps_1}\leq a_n^{\beta_i/\eta}.
\en
From (\ref{alpha-def}) and (\ref{beta-def}) we have $\beta_i/\eta\leq \al$, for $i=0,\dots L$.  Use this and (\ref{plp1}), (\ref{mj56}), (\ref{mj503}), (\ref{mj53}) we get,
\bn\label{mj55}
&&|u'_{1,a_n^{\alpha_i}}(t+a_n^{2\al+2\eps_0},x')-u'_{1,a_n^{\alpha_i}}(t+a_n^{2\al+2\eps_0},\hat{x}_n(s,x))| \nonumber \\
&&\leq  2^{-78}a_n^{-\eps_0-(2-\eta/2)\eps_1}a_n^{(\beta_i/\eta+5\eps_1)(1-\eps_0)}
\big[a_n^{(\eta-1)\beta_i/\eta}+2a_n^{(\gamma(\eta+1)-2+\eta/2)\beta_i/\eta}\big] \nonumber \\
&&\leq  2^{-77}a_n^{-\eps_0-(2-\eta/2)\eps_1}a_n^{5\eps_1(1-\eps_0)}a_n^{(1-\eps_0)\beta_i/\eta}
a_n^{(\eta-1)\beta_i/\eta} \nonumber \\
&&\leq 2^{-77}a_n^{-\eps_0-(2-\eta/2)\eps_1+5\eps_1(1-\eps_0)}3a_n^{\beta_i-\eps_0}\nonumber \\
&&= 2^{-75}a_n^{\beta_{i+1}}.
\en
From (\ref{mj55}) and Lemma \ref{Lemma6.5}(b) and (c) we immediately get claims (a) and (b) of this lemma.
\end{proof}
\begin{lemma}\label{Lemma6.7}
If $i\in\{0,...,L\}, \ 0\leq s \leq U_{M,n}$, $x\in J_{n,i}(s)$, and $|x-x'|\leq 4a_n^{\al}$, then
\bn \label{hyp-Lemma6.7}
|\tilde u_{2,a_n^{\alpha_i}}(s,a_n^{2\al+2\eps_0},x')-\tilde u_{2,a_n^{\alpha_i}}(s,a_n^{2\al+2\eps_0},x'')|&\leq& 2^{-75}a_n^{\beta_{i+1}}(|x'-x''|^{}\vee a_n), \
 \forall |x'-x''|<\bar{l}_n(\beta_i). \nonumber \\
\en
\end{lemma}
\begin{proof}
The proof uses the ideas from the proofs of Lemma 6.7 in \cite{MP09}. \\\\
Let $(n,i,s,x,x')$ as above and set $\eps=5a_n^{\al}\leq2^{-M}$, by (\ref{nM}). Then
\bn \label{mj60}
|x'-\hat{x}_n(s,x)|\leq |x'-x|+a_n^{\al}\leq \eps, \ |x'|\leq |x|+1 \leq K_0+1.
\en
From Lemma \ref{Lemma-cont-u} we have
\bn
| u(s,\hat{x}_n(s,x))|\leq a_n  =a_n\wedge (a_n^{1-\al}\eps).
\en
From the definition of $(s,x)\in J_{n,i}$ we have for $i>0$,
\bn \label{mj61}
|  u'_{1,a^{2\al}_n}(s,\hat{x}_n(s,x))|\leq a_n^{\beta_i}/4 \leq a_n^{\beta_i}.
\en
Denote by
\bn \label{mj62}
Q(n,\eps_0,\beta_i,r) = a_n^{-\eps_0}\big[\big(a_n^{-\al(1-\eta/2)-3\eps_0}r\big)\wedge \big(r^{(\eta-2\eps_0)/2}\big)\big]\big[(r\vee a_n^{\al})^{(1+\eta)\gamma}
 +(r\vee a_n^{\al})^\gamma a_n^{\beta_i\gamma}\big]. \nonumber
\en
Assume $|x'-x''|\leq \bar{l}_n(\beta_i)\leq 2^{-M}$ where the last inequality is by (\ref{nM}).
From $s<U_{M,n}\leq U^{(2)}_{M,n,\beta_i}$, and the definition of $U^{(2)}_{M,n,\beta_i}$, with $(x',x'')$ replacing $(x,x')$ we get
\bn  \label{mj63}
&&|\tilde u_{2,a_n^{\alpha_i}}(s,a_n^{2\al+2\eps_0},x'')-\tilde u_{2,a_n^{\alpha_i}}(s,a_n^{2\al+2\eps_0},x')|  \nonumber \\
&&\leq 2^{-87}a_n^{-\eps_0}\big[\big(|x''-x'|a_n^{-\al(1-\eta/2)-3\eps_0}\big)\wedge\big(|x''-x'|^{(\eta-2\eps_0)/2}\big)\big]
\big[((5a_n^{\al})\vee |x''-x'|)^{(\eta+1)\gamma} \nonumber \\
&&\quad+((5a_n^{\al})\vee |x''-x'|)^{\gamma}a_n^{\beta_i\gamma} \big]
 +|x''-x'|^{1-\eps_0}a_n^{\beta_i+\eps_1}\big).\nonumber \\
&&\leq 2^{-82}(Q(n,\eps_0,\beta_i,|x''-x'|)+|x''-x'|^{1-\eps_0}a_n^{\beta_i+99\beta_i\eps_1/\eta+\eps_1\eta^2/4}\big).
\en
We show that
\bn \label{mj64}
Q(n,\eps_0,\beta_i,r)\leq 2a_n^{\beta_{i+1}}r^{}, \  \ \forall \ 0\leq r\leq \bar{l}_n(\beta_i).
\en
\textbf{Case 1.} $a_n^{\al}\leq r\leq \bar{l}_n(\beta_i)$. \\
\bn \label{mj65}
Q(n,\eps_0,\beta_i,r)&\leq& a_n^{-\eps_0}r^{(\eta-2\eps_0)/2}\big[r^{(\eta+1)\gamma}
 +a_n^{\beta_i\gamma} r^{\gamma}\big] \nonumber \\
&=&  a_n^{-\eps_0}\big[r^{(\eta+1)\gamma+\eta/2-\eps_0}
 +a_n^{\beta_i\gamma} r^{\gamma+\eta/2-\eps_0}\big].
\en
Therefore, (\ref{mj64}) holds if
\bn  \label{mj66}
r^{(\eta+1)\gamma+\eta/2-1-\eps_0} \leq a_n^{\beta_{i+1}+\eps_0},
\en
and
\bn \label{mj67}
a_n^{\beta_i\gamma} r^{\gamma+\eta/2-1-\eps_0} \leq a_n^{\beta_{i+1}+\eps_0}.
\en
From (\ref{eps1}) we have
\bn \label{rr1}
 (1+\eta)\gamma+\eta/2-1-\eps_0>\eta+\eps_1,
\en
and hence
\bn \label{mj68}
r^{(\eta+1)\gamma+\eta/2-1-\eps_0} \leq r^{\eta},  \nonumber
\en
Hence by the upper bound on $r$ in this case, it suffices to show that $a_n^{\eta(\beta_i/\eta+5\eps_1)} \leq a_n^{\beta_{i+1}+\eps_0}$, which is clear from (\ref{eps1}) and (\ref{beta-def}). Hence, (\ref{mj66}) follows. Turning to (\ref{mj67}), from the upper bound on $r$ we have
\bn\label{mj69}
a_n^{\beta_i\gamma} r^{\gamma+\eta/2-1-\eps_0}a_n^{-\beta_{i+1}-\eps_0}&\leq& a_n^{\beta_i\gamma+ (\gamma+\eta/2-1-\eps_0)(\beta_i/\eta+5\eps_1)-\beta_{i+1}-\eps_0} \nonumber \\
&\leq &
a_n^{\frac{\beta_i}{\eta}(\gamma(\eta+1)-1+\eta/2)-\beta_i+5\eps_1(\gamma-1+\eta/2-\eps_0)-2\eps_0} \nonumber \\
&\leq & a_n^{5\eps_1(\gamma-1+\eta/2-\eps_0)-2\eps_0} \nonumber \\
&\leq& 1,
\en
where we used (\ref{eps1}) and (\ref{rr1}) in the last two inequalities. From (\ref{mj69}) we get (\ref{mj67}). This proves (\ref{mj64}) in the first case. \\\\
\textbf{Case 2.} $0\leq r <a_n^{\al}$. \\
Now, let us show that (\ref{mj64}) is satisfied in this case. Recall that in this case we get from (\ref{alpha_0}),
\bn \label{mj70}
Q(n,\eps_0,\beta_i,r) &\leq & a_n^{-\eps_0-\al(1-\eta/2)-3\eps_0}r\big[a_n^{\al(\eta+1)\gamma}+a_n^{\gamma\al+\beta_i\gamma}\big]\nonumber \\
&= &a_n^{-4\eps_0-\al(1-\eta/2)}r\big[a_n^{\gamma}+a_n^{\gamma\al+\beta_i\gamma}\big] .
\en
From (\ref{beta-up-lim}) and (\ref{alpha_0}) we have
\bn \label{mj7011}
a_n^{\gamma}\leq a_n^{\gamma\al+\beta_i\gamma}.
\en
From (\ref{mj70}) and (\ref{mj7011}) we conclude that (\ref{mj64}) holds if
\bn \label{mj71}
r^{}a_n^{\beta_{i+1}}\geq ra_n^{\gamma\al+\beta_i\gamma-4\eps_0-\al(1-\eta/2)}.
\en
Note that from (\ref{eps1}) we have
\bn \label{mj7222}
1-\gamma\leq \frac{\eta}{2(\eta+1)}-10\eps_1,
\en
and therefore,
\bn \label{mj321}
\frac{\eta}{\eta+1}(1-\gamma)\leq \big(\gamma-1+\frac{\eta}{2}\big)\frac{1}{\eta+1}-10\eps_1.
\en
From (\ref{mj321}), (\ref{alpha_0}) and (\ref{beta-up-lim}) we have
\bn\label{mj322}
\beta_{i+1}(1-\gamma)
& \leq & \big( \gamma-1+\frac{\eta}{2}\big)\al-10\eps_1.
\en
From (\ref{mj322}) and (\ref{beta-def}) we get
\bn \label{mj72}
\beta_{i+1}\leq \gamma\beta_i+ \big(\gamma-1+\frac{\eta}{2}\big)\al-4\eps_0,
\en
and therefore (\ref{mj71}) is satisfied.
From (\ref{mj70})-(\ref{mj71}) and (\ref{mj64}) follows for Case 2.  \\\\
Consider the second term in (\ref{mj63}). From (\ref{eps1}) we have $\eps_1<1/200$, use this and (\ref{beta-def}) to get
\bn \label{nn1}
\beta_{L}-\beta_{L+1}+ 99\eps_1 \frac{\beta_L}{\eta} &=& -5\eta\eps_1+ 99\eps_1\bigg(\frac{1}{\eta+1} -6\eps_1\bigg) \nonumber \\
&\geq & -5\eta\eps_1+ 44\eps_1-3\eps_1  \nonumber \\
&\geq& 10\eps_1 \nonumber \\
&\geq& 10\eps_0.
\en
Use (\ref{eps1}) and (\ref{beta-def}) again to get
\bn \label{nn2}
\beta_{i}-\beta_{i+1}+ \eps_1\frac{\eta^2}{4} \geq  \eps_1\frac{\eta^2}{4} -\eps_0 \geq 10\eps_0, \  \forall  i=0,...,L-1.
\en
 If $r\geq a_n$, then from (\ref{nn1}) and (\ref{nn2}) we get,
\bn \label{mj74}
r^{1-\eps_0}a_n^{\beta_i+99\eps_1\beta_i/\eta+ \eps_1\eta^2/4}(a_n^{\beta_{i+1}}r^{})^{-1} &\leq & r^{-\eps_0}a_n^{10\eps_0}\nonumber \\
&\leq& a_n^{9\eps_0} \nonumber \\
&\leq& 1, \  \forall  i=0,...,L.
\en
From (\ref{mj74}) follows
\bn  \label{mj75}
r^{1-\eps_0}a_n^{\beta_i+99\eps_1\beta_i/\eta+\eps_1\eta^2/4}\leq a_n^{\beta_{i+1}}(r^{}\vee a_n), \  \forall  i=0,...,L.
\en
From (\ref{mj63}), (\ref{mj64}) and (\ref{mj75}) we get (\ref{hyp-Lemma6.7}).
\end{proof}
\paragraph{Proof of Proposition \ref{PropStopTimes}}
The proof of the compactness is similar to that in the proof of Proposition 3.3 in \cite{MP09}. The inclusions $J_{n,i}(s)\subset \tilde{J}_{n,i}(s)$ for $0\leq s \leq U_{M,n}$ follows directly from Lemmas \ref{Lemma-cont-u}, \ref{Lemma666} and \ref{Lemma6.6}, \ref{Lemma6.7}.

\section{Proof of Proposition \ref{Prop5.14mod}}  \label{proof-u2}
In this section we prove Proposition \ref{Prop5.14mod}. The proof follows the same lines as the proof of Proposition 5.14 in \cite{MP09}. From (\ref{tilde-u2exp}) we get for $\dl\in [a_n^{2\al+2\eps_0},1]$,
\bn \label{u2Dec}
&&|\tilde u_{2,\dl}(t,a_n^{2\al+2\eps_0},x')- \tilde u_{2,\dl}(t,a_n^{2\al+2\eps_0},x)|  \\ \nonumber
&&= \bigg|\int_{(t+a_n^{2\al+2\eps_0}-\dl)^{+}}^{t}\int_{\re}G_{t+a_n^{2\al+2\eps_0}-s}(y-x')D(s,y)W(ds,dy) \\ \nonumber
&& \quad -\int_{(t+a_n^{2\al+2\eps_0}-\dl)^{+}}^{t}\int_{\re}G_{t+a_n^{2\al+2\eps_0}-s}(y-x)D(s,y)W(ds,dy)\bigg| \\ \nonumber
&& = \bigg|\int_{(t-a_n^{2\al+2\eps_0}-\dl)^{+}}^{t}\int_{\re}(G_{t+a_n^{2\al+2\eps_0}-s}(y-x')-G_{t+a_n^{2\al+2\eps_0}-s}(y-x))D(s,y)W(ds,dy)\bigg|.
\en
From (\ref{Dbound}) and (\ref{u2Dec}) we conclude that we need to bound the following quadratic variations
\bn
\hat{Q}_{S,1,\dl,\nu_0}(t,x,x')&=&\int_{(t+a_n^{2\al+2\eps_0}-\dl)^{+}}^{t}\int_{\re}\mathds{1}_{\{|x-y|>(t+a_n^{2\al+2\eps_0}-s)^{1/2-\nu_0}\vee (2|x-x'|)\}}
 \\ \nonumber
&&\times (G_{t+a_n^{2\al+2\eps_0}-s}(y-x')-G_{t+a_n^{2\al+2\eps_0}-s}(y-x))^2e^{2R_1|y|}|u(s,y)|^{2\gamma}\mu(dy)ds  \\ \nonumber
\hat{Q}_{S,2,\dl,\nu_0}(t,x,x')&=&\int_{(t+a_n^{2\al+2\eps_0}-\dl)^{+}}^{t}\int_{\re}\mathds{1}_{\{|x-y|\leq(t+a_n^{2\al+2\eps_0}-s)^{1/2-\nu_0}\vee (2|x-x'|)\}}
 \\ \nonumber
&&\times (G_{t+a_n^{2\al+2\eps_0}-s}(y-x')-G_{t+a_n^{2\al+2\eps_0}-s}(y-x))^2e^{2R_1|y|}|u(s,y)|^{2\gamma}\mu(dy)ds
\en
\begin{lemma} \label{Lemma7.1}
For any $K\in \mathds{N}^{\geq K_1}$ and $R>2$ there is a $C_{\ref{Lemma7.1}}(K,R_1,\eta,\nu_0,\nu_1)>0$ and an $N_{\ref{Lemma7.1}}=N_{\ref{Lemma7.1}}(K,\omega,\eta)\in \mathds{N}$ a.s. such that for all $\nu_0,\nu_1\in(1/R,1/2), \dl\in (0,1], N,n\in \mathds{N},\beta \in [0,\frac{\eta}{\eta+1}]$, $\eps\in (0,\eta/2)$ and $(t,x) \in \re_+\times \re$, on
\bn \label{7.1w-set}
\{\omega:(t,x)\in Z(N,n,K,\beta), N\geq N_{\ref{Lemma7.1}}\},
\en
\bn
\hat{Q}_{S,1,\dl,\nu_0}(t,x,x')\leq 2^{4N_{\ref{Lemma7.1}}}C_{\ref{Lemma7.1}}(K,R_1,\eta,\nu_0,\nu_1)2^{4N_{\ref{Lemma7.1}}}[d((t,x),(t,x'))\wedge \sqrt{\dl}]^{2-\nu_1}\dl^{1+\eta/2-\eps}, \ \forall x'\in\re.
\en
\end{lemma}
\begin{proof}
The proof follows the same lines as the proof of Lemma 7.1 in \cite{MP09}. Let $d=d((t,x),(t',x')$ and $N_{\ref{Lemma7.1}}=N_1(0,1-\eta/4,K)$, where $N_1$ is as in $(P_0)$. Then, as in $(P_0)$, $N_1$ depends only on $(1-\eta/4,K)$, and then for $\omega$ as in (\ref{7.1w-set}) we can use Lemma \ref{LemmaBound-u} with $m=0$ to get
\bn \label{rn4}
&&\hat{Q}_{S,1,\dl,\nu_0}(t,x,x')  \nonumber \\ \nonumber
&&\leq C_{\ref{LemmaBound-u}}(\omega)\int_{(t-a_n^{2\al+2\eps_0}-\dl)^{+}}^{t}\int_{\re}\mathds{1}_{\{|x-y|>(t+a_n^{2\al+2\eps_0}-s)^{1/2-\nu_0}\vee (2|x-x'|)\}}
 \\ \nonumber
&&\quad \times (G_{t+a_n^{2\al+2\eps_0}-s}(y-x')-G_{t+a_n^{2\al+2\eps_0}-s}(y-x))^2e^{2R_1|y|}e^{2|y-x|}(\bar{d}_N^{1-\eta/4+\eps/2})^{2\gamma}\mu(dy)ds  \\ \nonumber
&&= C_{\ref{LemmaBound-u}}(\omega)\int_{(t-a_n^{2\al+2\eps_0}-\dl)^{+}}^{t}\int_{\re}\mathds{1}_{\{|x-y|>(t+a_n^{2\al+2\eps_0}-s)^{1/2-\nu_0}\vee (2|x-x'|)\}}(G_{t+a_n^{2\al+2\eps_0}-s}(y-x')
 \\ \nonumber
&&\quad  -G_{t+a_n^{2\al+2\eps_0}-s}(y-x))^2e^{2R_1|y|}e^{2|y-x|}(2^{-N}\vee(|y-x|+\sqrt{t-s})^{(2-\eta/2+\eps)\gamma}\mu(dy)ds  \\ \nonumber
&&\leq C_{\ref{LemmaBound-u}}(\omega)e^{2R_1K}\int_{(t-a_n^{2\al+2\eps_0}-\dl)^{+}}^{t}\int_{\re}\mathds{1}_{\{|x-y|>(t+a_n^{2\al+2\eps_0}-s)^{1/2-\nu_0}\vee (2|x-x'|)\}}
 \\ \nonumber
&&\quad \times (G_{t+a_n^{2\al+2\eps_0}-s}(y-x')-G_{t+a_n^{2\al+2\eps_0}-s}(y-x))^2e^{2(R_1+1)|y-x|}(1+|x-y|)^{(2-\eta/2+\eps)\gamma}\mu(dy)ds  \\ \nonumber
&&\leq C_{\ref{LemmaBound-u}}(\omega)C(R_1,K,\eta,\nu_0,\nu_1)\int_{t-a_n^{2\al+2\eps_0}-\dl}^{t}
|t+a_n^{2\al+2\eps_0}-s|^{\eta/2-\eps_0-1}e^{-\nu_1(t+a_n^{2\al+2\eps_0}-s)^{-2\nu_0}/32}   \\
&&\quad \times \bigg[1\wedge \frac{d^2}{t+a_n^{2\al+2\eps_0}-s}\bigg]^{1-\nu_1/2}ds,
\en
where we used Lemma \ref{Lemma-MP4.3}(b) in the last inequality.
From (\ref{rn4}), change of variable and Lemma \ref{Lem4.1MP}(a) we get
\bn \label{rn6}
&&\hat{Q}_{S,1,\dl,\nu_0}(t,x,x')  \\ \nonumber
&&\leq C_{\ref{LemmaBound-u}}(\omega) C(R_1,K,\eta,\nu_0,\nu_1)\int_{t-\dl}^{t}e^{-\nu_1(t-s)^{-2\nu_0}/64}
(t-s)^{\eta/2-1-\eps_0}\bigg[1\wedge \frac{d^2}{t-s}\bigg]^{1-\nu_1/2}ds \\ \nonumber
&&\leq C_{\ref{LemmaBound-u}}(\omega) C(R_1,K,\eta,\nu_0,\nu_1)\int_{t-\dl}^{t}(t-s)^{1+\eta/2-\eps_0}\bigg[1\wedge \frac{d^2}{t-s}\bigg]^{1-\nu_1/2}ds \\ \nonumber
&&\leq C_{\ref{LemmaBound-u}}(\omega) C(R_1,K,\eta,\nu_0,\nu_1)(\dl\wedge d^2)^{1-\nu_1/2}\dl^{1+\eta/2-\eps_0}. \nonumber
\en
By remark \ref{Remark5.3} we can choose $C_{\ref{LemmaBound-u}}$ with $\eps_0=0$ and this completes the proof.
\end{proof}
\begin{lemma} \label{Lemma7.2}
Let $0\leq m \leq \bar{m}+1$ and assume $(P_m)$. Fix $\theta\in (0,\gamma-1+\eta/2-\eps_0)$. Then, for any $K\in \mathds{N}^{\geq K_1},$ \\ $R>2/\theta,n\in \mathds{N}, \eps_0\in (0,1)$ and $\beta\in [0,\frac{\eta}{\eta+1}]$, there exists a $C_{\ref{Lemma7.2}}(\eta,\eps_0,K,R_1)>0$ and \\ $N_{\ref{Lemma7.2}}=N_{\ref{Lemma7.2}}(m,n,R,\eps_0,K,\beta,\eta)(\omega)\in \mathds{N}$ a.s. such that for any $\nu_1\in(R^{-1},\theta/2), \nu_0\in(0,\nu_1/24), \dl\in[a_n^{2\al},1],$ \\ $N\in \mathds{N}$, and $(t,x)\in \re_+\times \re$, on
\bn \label{rn7}
\{\omega:(t,x)\in Z(N,n,K,\beta), N\geq N_{\ref{Lemma7.2}}\},
\en
\bn \label{rn888}
&&\hat{Q}_{S,2,\dl,\nu_0}(t,x,x') \nonumber \\
&&\leq C_{\ref{Lemma7.2}}(\eta,\eps_0,K,R_1)[a_n^{-2\eps_0}+2^{4N_{\ref{Lemma7.2}}}]\big[(d^2a_n^{-2\al(1-\eta/2)-3\eps_0})\wedge (d^{\eta-2\eps_0})\big]
 \bar{d}_N^{2\gamma-\nu_0/2}\big[(\bar{d}_{n,N})^{2\gamma(\tilde{\gamma}_m-1)} +a_n^{2\gamma\beta}\big] \nonumber \\
   && \ \forall t  \leq K, \ |x'|\leq K+1.
\en
Here $d=d((t,x),(t,x')), \ \bar{d}_N=d\vee2^{-N}$ and $\bar{d}_{n,N}=a_n^{\al}\vee \bar{d}_N$. Moreover $N_{\ref{Lemma7.2}}$ is stochastically bounded uniformly in $(n,\beta)$.
\end{lemma}
\begin{proof}
The proof follows the same lines of the proof of Lemma 7.2 in \cite{MP09}. Set $\xi=1-1/(24R)$ and $N_{\ref{Lemma7.2}}(m,n,R,\eps_0,K,\beta,\eta)=N_1(m,n,\xi,\eps_0,K,\beta,\eta)$, which is stochastically bounded informally in $(n,\beta)$ by $(P_m)$. For $\omega$ as in (\ref{rn7}), $t\leq K, \ |x'|\leq K+1$ we get from Lemma \ref{LemmaBound-u} and then Lemma \ref{Lemma-MP4.3}(a),
\bn \label{rn8}
&&\hat{Q}_{S,2,\dl,\nu_0}(t,x,x') \nonumber \\
&&=\int_{(t+a_n^{2\al+2\eps_0}-\dl)^{+}}^{t}\int_{\re}\mathds{1}_{\{|x-y|\leq(t+a_n^{2\al+2\eps_0}-s)^{1/2-\nu_0}\vee (2|x-x'|)\}} (G_{t+a_n^{2\al+2\eps_0}-s}(y-x')-G_{t+a_n^{2\al+2\eps_0}-s}(y-x))^2 \nonumber \\
&&\quad \times e^{2R_1|y|}|u(s,y)|^{2\gamma}\mu(dy)ds  \nonumber \\
&&\leq C_{\ref{LemmaBound-u}}(\omega)\int_{(t+a_n^{2\al+2\eps_0}-\dl)^{+}}^{t} \int_{\re} (G_{t+a_n^{2\al+2\eps_0}-s}(y-x')-G_{t+a_n^{2\al+2\eps_0}-s}(y-x))^2 e^{2R_1K}e^{2(R_1+1)4(K+1)}
 \nonumber \\
&&\quad\times \bar{d}_N^{2\gamma\xi}\big[(a_n^{\al}\vee  \bar{d}_N)^{\bar{\gamma}_m-1}+\mathds{1}_{\{m>0\}}a_n^{\beta}\big]|^{2\gamma}\mu(dy)ds \nonumber \\
&&\leq C_{\ref{LemmaBound-u}}(\omega) C(\eta,\eps_0,R_1,K)\bar{d}_N^{2\gamma\xi}\big[(a_n^{\al}\vee \bar{d}_N)^{\bar{\gamma}_m-1}+\mathds{1}_{\{m>0\}}a_n^{\beta}\big]^{2\gamma}  \nonumber \\
&&\quad \times \int_{(t+a_n^{2\al+2\eps_0}-\dl)^{+}}^{t} (t+a_n^{2\al+2\eps_0}-s)^{\eta/2-1-\eps_0}\bigg(1\wedge \frac{d^2}{t+a_n^{2\al+2\eps_0}-s}\bigg)ds.
\en
From Lemma \ref{Lem4.1MP}(b) with $t+a_n^{2\al+2\eps_0}$ instead of $t$, $\Delta=d^2$, $\Delta_1=\dl$ and $\Delta_2=a_n^{2\al+2\eps_0}$ we have
\bn \label{newww2}
&&\int_{(t+a_n^{2\al+2\eps_0}-\dl)^{+}}^{t} (t+a_n^{2\al+2\eps_0}-s)^{\eta/2-1-\eps_0}\bigg(1\wedge \frac{d^2}{t+a_n^{2\al+2\eps_0}-s}\bigg)ds  \nonumber \\
&&\leq C[(d^{2}\wedge \dl)^{\eta/2-\eps_0}+(d^2\wedge \dl)a_n^{(2\al+2\eps_0)(\eta/2-1-\eps_0)}].
\en
On the other hand by a simple integration we have
\bn \label{newww3}
&&\int_{(t+a_n^{2\al+2\eps_0}-\dl)^{+}}^{t} (t+a_n^{2\al+2\eps_0}-s)^{\eta/2-1-\eps_0}\bigg(1\wedge \frac{d^2}{t+a_n^{2\al+2\eps_0}-s}\bigg)ds  \nonumber \\
&&\leq d^2\int_{(t+a_n^{2\al+2\eps_0}-\dl)^{+}}^{t} (t+a_n^{2\al+2\eps_0}-s)^{\eta/2-2-\eps_0}ds  \nonumber \\
&&\leq C d^2a_n^{(2\al+2\eps_0)(\eta/2-1-\eps_0)},
\en
where we have used the bound $\dl \in [a_n^{2\al},1]$ in the last inequality.
From (\ref{rn8}), (\ref{newww2}), (\ref{newww3}) and  we get (\ref{rn888}).
\end{proof}
\paragraph{Proof of Proposition \ref{Prop5.14mod}}
The proof follows the same lines as the proof of Proposition 5.14 in \cite{MP09}. Let $R=\frac{25}{\nu_1 }$ and choose $\nu_0\in \big(\frac{1}{R},\frac{\nu_1}{24}\big)$. Recall the previously introduced notation $\bar{d}_N=d\vee 2^{-N}$. Let
\bn
\hat{Q}_{a_n^\alpha}(t,x,t',x')=\sum_{i=1}^{2} \hat{Q}_{S,i,a_n^\alpha,\nu_0}(t,x,x').
\en
By Lemmas \ref{Lemma7.1} and \ref{Lemma7.2} we have for all $K\in \mathds{N}$ there is a constant $C_1(K,R_1,\eta,\nu_0,\nu_1,\eps_0)>0$ and \\
$N_2(m,n,\nu_1,\eps_0,K,\beta,\eta)\in \mathds{N}$ a.s. stochastically bounded uniformly in $(n,\beta)$, such that for all $N\in \mathds{N}$, $(t,x)\in \re_+\times \re$,
\bn
\textrm{on } \{\omega:(t,x)\in Z(N,n,m,K+1,\beta), N\geq N_2\} \nonumber
\en
\bn \label{tt100}
&&R_0^{\gamma}\hat{Q}_{a_n^\alpha}(t,x,x')^{1/2} \nonumber \\
&&\leq  C_12^{2N_2}[d\wedge a_n^{\alpha/2}]^{1-\nu_1/2}a_n^{\alpha(1+\eta/2-\eps_0)/2} \nonumber \\
&&\quad +C_1[a_n^{-\eps_0}+2^{2N_2}]\big[( a_n^{-\al(1-\eta/2)-2\eps_0}d)\wedge d^{\eta/2-\eps_0}\big]
 \bar{d}_N^{\gamma-\nu_0/4}  \big[(\bar{d}_{n,N})^{\gamma(\tilde{\gamma}_m-1)} +a_n^{\gamma\beta}\big]    \nonumber \\
&& \forall \ t \leq T_K, \ |x'|\leq K+2.
\en
Let $N_3=\frac{25}{\nu_1}[N_2+N_4(K,R_1,\eta,\nu_0,\nu_1,\eps_0)]$, where $N_4(K,R_1,\eta,\nu_0,\nu_1,\eps_0)$ is chosen large enough so that
\bn \label{tt101}
C_1(K,R_1,\eta,\nu_0,\nu_1,\eps_0)[a_n^{-\eps_0}+2^{2N_2}]2^{-N_3\nu_1/8} &\leq& C_1(K,R_1,\eta,\nu_0,\nu_1,\eps_0)[a_n^{-\eps_0}+2^{2N_2}]2^{-6N_2}2^{N_4(K,R_1,\eta,\nu_0,\nu_1,\eps_0)}\nonumber \\
&\leq& a_n^{-\eps_0}2^{-104}.
\en
Recall the notation introduced in (\ref{aaa}),
\bn
\hat{\Delta}_{1,u_2}(m,n,\eps_0,2^{-N},\eta)&=&2^{-N\gamma(1-\eps_0)}\big[(2^{-N}\vee a_n^{\al})^{\gamma(\bar{\gamma}_m-1)}+a_n^{\gamma\beta}\big], \nonumber \\
\hat{\Delta}_{2,u_2}(m,n,\eps_0,\eta)&=&a_n^{\alpha(1+\eta/2-\eps_0)/2}. \nonumber
\en
Let $\Delta_{i,u_2}=2^{-100}\hat{\Delta}_{i,u_2}, \ i=1,2$. Assume $d\leq 2^{-N}$ and $\nu_1\leq \eps_0$. From (\ref{tt100}) and (\ref{tt101}) we get for all
$(t,x)$ and $N$ on
\bn \label{tt102}
 \{\omega:(t,x)\in Z(N,n,m,K+1,\beta), N\geq N_3\},
\en
\bn \label{tt103}
R_0^{\gamma}\hat{Q}_{a_n^\alpha}(t,x,x')^{1/2} &\leq&2^{-104}a_n^{-\eps_0}\big[(a_n^{-\al(1-\eta/2)-2\eps_0}d)\wedge d^{\eta/2-\eps_0}\big]2^{-N(\gamma-5\nu_1/8)}
\big[(2^{-N}\vee a_n^{\al})^{\gamma(\bar{\gamma}_m-1)}+a_n^{\gamma\beta}\big] \nonumber \\
&&+ 2^{-104}a_n^{-\eps_0}(d\wedge a_n^{\alpha/2})^{1-5\nu_1/8}a_n^{\alpha(1+\eta/2-\eps_0)/2}\nonumber \nonumber \\
&=&a_n^{-\eps_0}\big[(a_n^{-\al(1-\eta/2)-2\eps_0}d)\wedge d^{\eta/2-\eps_0}\big]\hat \Delta_{1,u_2}(m,n,2^{-N},\eta)/16 \nonumber \\
&&+ (d\wedge a_n^{\alpha/2})^{1-5\nu_1/8}\hat \Delta_{2,u_2}(m,n,2^{-N},\eta)/16, \ \forall \ t\leq t' \leq T_K, \ |x'|\leq K+2.
\en
The rest of the proof is identical the proof of Proposition 5.14 in \cite{MP09}. We use Dubins-Schwartz theorem and (\ref{tt103}) to bound $|\tilde u_{2,\dl}(t,a_n^{2\al+2\eps_0},x')- \tilde u_{2,\dl}(t,a_n^{2\al+2\eps_0},x)|$ and we get (\ref{rtt}). \qed

\section{Proof of Theorems \ref{Theorem2.3MP-P0-ind} and Theorem \ref{thm-reg}} \label{Section-Proof2.3}
In this section we prove Theorem \ref{Theorem2.3MP-P0-ind}. Later in this section we prove Theorem \ref{thm-reg} as a consequence of Theorem \ref{Theorem2.3MP-P0-ind}. Before we start with the proofs of theses theorems, let us introduce some notation and prove a weaker auxiliary result.
\paragraph{Notation}
Denote by
\bn \label{Def-Z-K-N-xi-Set}
Z_{K,N,\xi}&=&\{(t,x)\in \re_{+}\times \re: t\leq T_K,|x|\leq K, d((t,x),(\hat{t},\hat{x}))\leq 2^{-N} \textrm{ for some }
\nonumber \\ &&(\hat{t},\hat{x}) \in [0,T_k] \times \re \textrm{ satisfying } |u(\hat{t},\hat{x})| \leq 2^{-N\xi}\}.\nonumber
\en
Recall that $\mu\in M_f^{\eta}$ for a fixed $\eta\in(0,1)$. Denote by $\eta'\equiv \eta-\varpi$ where $\varpi>0$ is arbitrarily small.
 \medskip \\
\begin{theorem} \label{Theorem4.1-MPS} Assume the hypothesis of Theorem \ref{holder-conj}, except allow $\gamma\in(0,1]$. Let $u_0\in \mathcal{C}_{tem}$ and $u=u^1-u^2$, where $u^i$ is a
$\mathcal{C}(\re_{+},\mathcal{C}(\re))$ a.s. solution of (\ref{SHE}), for $i=1,2$. Let $\xi\in(0,1)$ satisfy
 \bn
\label{hyp-Thm-4.1-MPS} &&\exists N_\xi=N_\xi(K,\omega)\in \mathds{N} \textrm{ a.s. such that } \forall N\geq N_\xi, (t,x)\in
Z_{K,N,\xi} \nonumber \\ && d((x,t),(t',y))\leq 2^{-N},\ t,t'\leq T_K \Rightarrow |u(t,x)-u(t',y)|\leq 2^{-N\xi}.
\en
Let
$0<\xi_1<(\gamma\xi+\eta'/2)\wedge 1$. Then there is an $N_{\xi_1}=N_{\xi_1}(K,\omega)\in \mathds{N}$ a.s. such that for any
$N\geq N_{\xi_1}\in \mathds{N}$ and any $(t,x)\in Z_{K,N,\xi}$
\bn  \label{res-Thm-4.1-MPS}
 d((x,t),(t',y)) \leq 2^{-N}, \ t,t'\leq T_K \Rightarrow |u(t,x)-u(t',y)|\leq 2^{-N\xi_1}.
\en
 Moreover, there are strictly positive constants $R,\dl,C_{\ref{Theorem4.1-MPS}.1},
C_{\ref{Theorem4.1-MPS}.2}$ depending only on $(\xi,\xi_1)$ and $N(K)\in \mathds{N}$, such that
\bn
P(N_{\xi_1}\geq N) \leq C_{\ref{Theorem4.1-MPS}.1}(P(N_{\xi} \geq N/R)+K^{2}\exp{(-C_{\ref{Theorem4.1-MPS}.2}2^{N\dl}})), \ \forall N\geq N(K).
\en
\end{theorem}
Theorem \ref{Theorem2.3MP-P0-ind} was proved in \cite{MPS06} for the case of the $d$-dimensional stochastic heat equation driven by  colored noise.
\paragraph{Proof of Theorem \ref{Theorem4.1-MPS}}
The proof follows the same lines as the proof of Theorem 4.1 in \cite{MPS06}. Fix arbitrary
$(t,x),(t',y)$ such that $d((x,t),(t',y))\leq \eps\equiv 2^{-N}, \ (N\in\mathds{N})$ and $t\leq t'$ (the case $t'\leq t$ works
analogously). Since $\xi_1<(\gamma\xi+\eta'/2)\wedge 1$ and $\xi\gamma \in (0,1)$, we can choose $\dl\in (0,\eta'/2)$ such that \bn
\label{rt1} 1>\gamma\xi+\eta'/2-\dl>\xi_1. \en We can also pick $\dl'\in(0,\dl)$ and $p\in(0,\xi\gamma)$ such that \bn
\label{rt2} 1>p+\eta'/2-\dl >\xi_1, \en and \bn \label{rt3} 1>\xi\gamma+\eta'/2-\dl' >\xi_1. \en Let $N_1=N_1(\omega,\xi,\xi_1)$
to be chosen below, \bn \label{rt4} &&P\big(|u(t,x)-u(t,y)|\geq |x-y|^{\eta'/2-\dl}\eps^p, (t,x)\in Z_{K,N,\xi}, N\geq
N_1\big)\nonumber \\ &&+P\big(|u(t,x)-u(t',x)|\geq |t'-t|^{\eta'/4-\dl/2}\eps^p, (t,x)\in Z_{K,N,\xi},t'\leq T_K, N\geq
N_1\big). \en
We introduce the following notation. Let
\bn \label{rt5}
D^{x,y,t,t'}(z,s)&=&[G_{t-s}(x-z)-G_{t'-s}(y-z)]^{2}u^{2\gamma}(s,z), \nonumber \\
D^{x,t'}(z,s)&=&G^2_{t'-s}(x-z)u^{2\gamma}(s,z).
\en
From (\ref{rt5}) follows that (\ref{rt4}) is bounded by
\bn \label{rt6}
&& P\bigg(|u(t,x)-u(t,y)|\geq |x-y|^{\eta'/2-\dl}\eps^p, (t,x)\in Z_{K,N,\xi}, N\geq N_1, \nonumber \\
&&\quad\int_{0}^{t}\int_{\re}D^{x,y,t,t}(z,s)\mu(dz)ds\leq |x-y|^{\eta'-2\dl'}\eps^{2p} \bigg)\nonumber \\
&&+P\bigg(|u(t,x)-u(t',x)|\geq |t'-t|^{\eta'/4-\dl/2}\eps^p, (t,x)\in Z_{K,N,\xi},t'\leq T_K, N\geq N_1,\nonumber \\
&&\quad\int_{t}^{t'}\int_{\re}D^{x,t'}(z,s)\mu(dz)ds+ \int_{0}^{t}\int_{\re}D^{x,x,t,t'}(z,s)\mu(dz)ds\leq
|t'-t|^{\eta'/2-\dl'}\eps^{2p} \bigg)\nonumber \\ &&+ P\bigg( \int_{0}^{t}\int_{\re}D^{x,y,t,t}(z,s)\mu(dz)ds >
|x-y|^{\eta'-2\dl'}\eps^{2p},(t,x)\in Z_{K,N,\xi}, N\geq N_1 \bigg)\nonumber \\
&&+P\bigg(\int_{t}^{t'}\int_{\re}D^{x,t'}(z,s)\mu(dz)ds+ \int_{0}^{t}\int_{\re}D^{x,x,t,t'}(z,s)\mu(dz)ds \nonumber \\ &&
\qquad \qquad \qquad \qquad \qquad \qquad \qquad > |t'-t|^{\eta'/2-\dl'}\eps^{2p},(t,x)\in Z_{K,N,\xi},t'\leq T_K, N\geq
N_1\bigg)\nonumber \\ &&=: P_1+P_2+P_3+P_4.
\en
The bounds on $P_1,P_2$ are derived by the Dubins-Schwarz theorem in a similar way as in the proof of Theorem 4.1.
in \cite{MPS06}. If $\dl''=\dl-\dl'$, then we can show that
\bn\label{rt67} P_1&\leq&
C_{\ref{rt67}}e^{-C'_{\ref{rt67}}|x-y|^{-\dl^{''}}}, \nonumber \\ P_2&\leq&
C_{\ref{rt67}}e^{-C'_{\ref{rt67}}|t-t'|^{-\dl^{''}/2}},
\en
 where the constants $C_{\ref{rt67}},C'_{\ref{rt67}}$ only depend on $R_0,R_1$ in (\ref{HolSigmaCon2}). \medskip \\
 In order to bound $P_3,P_4$, we need to split the integrals that are related to them into several parts. Let $\dl_1\in(0,\eta'/4)$ and $t_0=0,t_1=t-\eps^2,t_2=t$ and
$t_3=t'$. We also define
\bn \label{rt8}
A_1^{1,s}(x)&=&\{z\in \re: |x-z|\leq 2\sqrt{t-s}\eps^{-\dl_1}\} \textrm{ and }
A_2^{1,s}(x)=\re\setminus A_1^{1,s}(x), \nonumber \\ A_1^{2}(x)&=&\{z\in \re: |x-z|\leq 2\eps^{1-\dl_1}\} \textrm{ and }
A_2^{2}(x)=\re\setminus A_1^{2}(x).
\en For notational convenience the index $s$ in $A_i^{1,s}(x)$ is sometimes
omitted. Define,
\bn \label{rt9}
Q^{x,y,t,t'}:=\int_{0}^{t}\int_{\re}D^{x,y,t,t'}(z,s)\mu(dz)ds=\sum_{i,j=1,2}Q_{i,j}^{x,y,t,t'}, \nonumber
\en where
\bn
\label{rt10} Q_{i,j}^{x,y,t,t'}:=\int_{t_{i-1}}^{t_i}\int_{A_j^{1,s}(x)}D^{x,y,t,t'}(z,s)\mu(dz)ds,
\en
and
\bn \label{rt11}
Q^{x,t,t'}:=\int_{t}^{t'}\int_{\re}D^{x,t'}(z,s)\mu(dz)ds=\sum_{j=1,2}Q_{j}^{x,t,t'},\nonumber \en
 where
 \bn \label{rt12}
Q_{j}^{x,t,t'}:=\int_{t}^{t'}\int_{A_j^{2}(x)}D^{x,t'}(z,s)\mu(dz)ds.
\en Set
\bn N_1(\omega) =
\bigg\lceil\frac{5N_{\xi}(\omega)}{\dl_1}\bigg\rceil \geq \bigg\lceil\frac{N_{\xi}(\omega)+4}{1-\dl_1}\bigg\rceil\in
\mathds{N},
\en
where $N_{\xi}(\omega)$ was chosen in the hypothesis of Theorem \ref{Theorem4.1-MPS} and $\lceil
\cdot \rceil$ is the greatest integer value function. We introduce three
lemmas that will give us bounds on the quadratic variation terms in (\ref{rt10}) and (\ref{rt12}). Recall that
$\lam>0$ is a fixed constant that was used in the definition of $T_K$. Therefore, $Z_{K,N,\xi}$ depends on $\lam$, but its dependence in $\lam$ is often suppressed in
our notation.
\begin{lemma} \label{lemma5.5-MPS} Let $N\geq N_1$. Then on $\{\omega:(t,x) \in Z_{K,N,\xi}\}$,
\bn
|u(s,z)|&\leq& 10\eps^{(1-\dl_1)\xi} , \qquad \qquad  \qquad  \qquad  \quad  \ \ \forall s\in[t-\eps^2,t'], z\in A^{2}_1(x),
\nonumber \\ |u(s,z)|&\leq& (8+3K2^{N_{\xi}\xi})e^{\lam|z|}(t-s)^{\xi/2}\eps^{-\dl_1\xi} , \ \forall s\in[0,t-\eps^2], z\in
A^{1,s}_1(x).
\en
\end{lemma} The proof of Lemma \ref{lemma5.5-MPS} is similar to the proof of Lemma 5.4 in \cite{MPS06},
hence it is omitted.
\begin{lemma} \label{lemma5.6-MPS} If $0<\vartheta<\eta/2$, $\vartheta'\leq \gamma \xi +\eta'/2$, and $\vartheta' <1$, then
on $\{\omega:(t,x)\in Z_{K,N,\xi}\}$,
\begin{itemize}
  \item [\bf{(a)}] $$Q_{2,1}^{x,y,t,t}\leq C(\vartheta,K,\eta')\eps^{2(1-\dl_1)\xi\gamma}|x-y|^{2\vartheta}, $$
  \item [\bf{(b)}] $$Q_{2,1}^{x,x,t,t'}\leq C(\vartheta,K,\eta')\eps^{2(1-\dl_1)\xi\gamma}|t-t'|^{\vartheta},$$
  \item [\bf{(c)}] $$Q_{1,1}^{x,y,t,t}\leq C(\gamma,
      K,\eta')(8+3K2^{N_{\xi}\xi})^{2\gamma}\eps^{-2\dl_1\xi\gamma}|x-y|^{2\vartheta'},$$
  \item [\bf{(d)}] $$Q_{1,1}^{x,x,t,t'}\leq C(\gamma,
      K,\eta')(8+3K2^{N_{\xi}\xi})^{2\gamma}\eps^{-2\dl_1\xi\gamma}|t-t'|^{\vartheta'},$$
  \item [\bf{(e)}] $$Q_{1}^{x,t,t'}\leq C(K,\vartheta)\eps^{2\gamma\xi(1-\dl_1)}|t-t'|^{\eta'/2}.$$
\end{itemize}
\end{lemma}
\begin{proof} (a) From Lemma \ref{lemma5.5-MPS} and Lemma \ref{mu3cond}(c) we get
\bn
\label{rt20} Q_{2,1}^{x,y,t,t}&\leq&
100^{\gamma}\eps^{2(1-\dl_1)\xi\gamma}\int_{t-\eps^2}^{t}\int_{A_1^2(x)}[G_{t-s}(x-z)-G_{t-s}(y-z)]^2\mu(dz)ds \nonumber \\
&\leq& 100^{\gamma}\eps^{2(1-\dl_1)\xi\gamma}\int_{t-\eps^2}^{t}\int_{\re}[G_{t-s}(x-z)-G_{t-s}(y-z)]^2\mu(dz)ds \nonumber \\
&\leq&  C(\eta',K,\vartheta)\eps^{2(1-\dl_1)\xi\gamma}|x-y|^{2\vartheta},
\en
where we use the fact that $t\leq K$. (b) Repeat the same
steps in (\ref{rt21}) to get
\bn \label{rt21}
Q_{2,1}^{x,x,t,t'}\leq C(\vartheta,K,\eta')\eps^{2(1-\dl_1)\xi\gamma}|t-t'|^{\vartheta}.
\en
(c) From Lemma \ref{lemma5.5-MPS} we get
$Q_{1,1}^{x,y,t,t'}$,
\bn \label{rt22} &&Q_{1,1}^{x,y,t,t'} \nonumber \\ &&\leq
(8+3K2^{N_{\xi}\xi})^{2\gamma}\eps^{-2\gamma\dl_1\xi}
\int_{0}^{t-\eps^2}(t-s)^{\xi\gamma}\int_{\re}e^{2\gamma\lambda|z|}[G_{t-s}(x-z)-G_{t'-s}(y-z)]^2\mu(dz)ds. \nonumber \\
\en
From Lemma \ref{lemma-new2} we get
\bn  \label{rt2222}
&&\int_{\re}e^{2\gamma|z|}[G_{t-s}(x-z)-G_{t'-s}(y-z)]^2\mu(dz)ds \nonumber \\
&&\leq C(\lam,\gamma,\eta')(t-s)^{-2+\eta'/2}[|x-y|^{2} + |t-t'|^{}]e^{2\lam|x|}e^{2\lam|x-y|} \nonumber \\
&&\leq C(\lam,\gamma,K,\eta')(t-s)^{-2+\eta'/2}[|x-y|^{2} + |t-t'|^{}]  \nonumber \\
&&\leq C(\lam,\gamma,K,\eta')\eps^{2(1-\vartheta')}(t-s)^{-2+\eta'/2}[|x-y|^{2\vartheta'} + |t-t'|^{\vartheta'}], \nonumber \\
&& \ \forall x\in [-K+1,K+1],  \ t', t\geq 0, \ [x-y]^2+|t-t'|\leq \eps^2.
\en
From (\ref{rt22}), (\ref{rt2222}) and our choice of $\vartheta'$, (c) and (d) follow. \medskip \\
(e) From Lemmas \ref{lemma5.5-MPS} and \ref{condmu}(a) we have
\bn \label{rt24}
Q_{1}^{x,x,t,t'} &\leq & Ce^{2\gamma \xi(1-\dl_1)}\int_{t}^{t'}\int_{\re}e^{2\gamma\lambda|z|}G^2_{t'-s}(x-z)\mu(dz)ds \nonumber \\
&\leq&C(K,\vartheta)e^{2\gamma \xi(1-\dl_1)}|t'-t|^{\eta'/2}.
\en
\end{proof} \\\\
Next we consider the terms for which $j=2$. We
will use the fact that for $t\leq T_K$ we have
\bn \label{rt23}
|u(t,x)|\leq Ke^{ |x|}.
\en

\begin{lemma}\label{lemma5.7-MPS}
For $0<\vartheta < \eta'/2$ we have for $i=1,2$, on $\{\omega:(t,x)\in Z_{K,N,\xi}\}$,
\begin{itemize}
  \item [\bf{(a)}] $$Q_{i,2}^{x,y,t,t}\leq C(K,\vartheta,\eta',\gamma)e^{-\eps^{-2\dl_1}(1-\vartheta)/4}|x-y|^{2\vartheta}, $$
  \item [\bf{(b)}] $$Q_{i,2}^{x,x,t,t'}\leq C(K,\vartheta,\eta',\gamma)e^{-\eps^{-2\dl_1}(1-\vartheta/2)/4}|t-t'|^{\vartheta},$$
  \item [\bf{(c)}] $$Q_{2}^{x,t,t'}\leq C(K,\vartheta,\eta',\gamma)e^{-\eps^{-2\dl_1}(1-\vartheta)/2}|t-t'|^{\eta'/2}.$$
\end{itemize}
\end{lemma}

\begin{proof}
The proof follows the same lines as the proof of Lemma 5.6 in \cite{MPS06}. Recall
$d((t,x),(t',y))\leq \eps$. For $i=1$ we deal with the case where $s\in[0,t-\eps^2]$ and $|x-z|>2\sqrt{t-s}\eps^{-\dl_1}$.
Since $|x-y|<\eps$ we have $|y-z|>||x-z|-|x-y||>2\sqrt{t-s}\eps^{-\dl_1}-\eps>\sqrt{t-s}\eps^{-\dl_1}$. Furthermore,
$t'-s=t'-t+t-s\leq \eps^2+t-s\leq 2(t-s)$. Therefore,
\bn\label{rt33}
\exp{\bigg(-\frac{(x-z)^2}{4(t'-s)}\bigg)}\vee
\exp{\bigg(-\frac{(y-z)^2}{4(t'-s)}\bigg)}&\leq & \exp{\bigg(-\frac{(x-z)^2}{8(t-s)}\bigg)}
\exp{\bigg(-\frac{(y-z)^2}{8(t-s)}\bigg)} \nonumber \\ &\leq &\exp\big({-\frac{\eps^{-2\dl_1}}{8}}\big).
\en
Therefore, for
$v=x$ or $v=y$ and $r=t$ or $r=t'$ we get
\bn \label{rt332} G_{r-s}(v-z)\leq
C(\theta)\exp{\bigg(-\frac{\eps^{-2\dl_1}}{4}(1-\theta)}\bigg)G_{(r-s)/\theta}(v-z),  \ \forall \theta \in (0,1).
\en
From (\ref{rt23}), (\ref{rt332}) with $\theta =1/2$ and H\"{o}lder's
inequality we have
\bn \label{rt34} &&Q_{1,2}^{x,y,t,t}  \nonumber \\ &&\leq
\int_{0}^{t-\eps^2}\bigg(\int_{A_2^{1,s}}[G_{t-s}(x-z)+G_{t'-s}(y-z)]^2|u(s,z)|^{\frac{2\gamma}{1-\vartheta}}\mu(dz)\bigg)^{1-\vartheta}
\nonumber \\ &&\quad \times \bigg(\int_{A_2^{1,s}}[G_{t-s}(x-z)-G_{t'-s}(y-z)]^2\mu(dz)ds\bigg)^{\vartheta}ds \nonumber \\ &&\leq
Ce^{{-\frac{\eps^{-2\dl_1}}{4}(1-\vartheta)}}\int_{0}^{t-\eps^2}\bigg(\int_{\re}e^{\frac{2\gamma|z|}{1-\vartheta}}[G_{2(t-s)}(x-z)+G_{2(t'-s)}(y-z)]^2\mu(dz)\bigg)^{1-\vartheta}
\nonumber \\ && \quad\times \bigg(\int_{\re}[G_{t-s}(x-z)-G_{t'-s}(y-z)]^2\mu(dz)\bigg)^{\vartheta}ds.
\en
Apply Lemmas \ref{condmu}(a) and \ref{Lemma-MP4.3}(a) to (\ref{rt34}) to get
\bn \label{rt36}
&&Q_{1,2}^{x,y,t,t}  \nonumber \\ &&\leq
C(K,\vartheta,\eta',\gamma)e^{{-\frac{\eps^{-2\dl_1}}{4}(1-\vartheta)}}\int_{0}^{t-\eps^2}\frac{1}{(t-s)^{(2-\eta')(1-\vartheta)/2}} \nonumber \\
&&\quad \times \bigg(\int_{\re}[G_{t-s}(x-z)-G_{t-s}(y-z)]^2\mu(dz)\bigg)^{\vartheta}ds  \nonumber \\ &&\leq
C(K,\vartheta,\eta',\gamma)e^{{-\frac{\eps^{-2\dl_1}}{4}(1-\vartheta)}}\int_{0}^{t}\frac{1}{(t-s)^{(2-\eta')(1-\vartheta)/2}}
\frac{(x-y)^{2\vartheta}}{(t-s)^{(2-\eta'/2)\vartheta}}ds  \nonumber \\
 &&\leq C(K,\vartheta,\eta',\gamma)e^{{-\frac{\eps^{-2\dl_1}}{4}(1-\vartheta)}}(x-y)^{2\vartheta},
\en
where we have used in the last inequality the facts
$\vartheta\in(0,\eta'/2)$ and therefore, \\ $(2-\eta')(1-\vartheta)/2+(2-\eta'/2)\vartheta = 1-\eta'/2+\vartheta<1$. From (\ref{rt36}) we get (a) for $i=1$. \medskip \\
To prove (b) for $i=1$, we use the
following inequality from the proof of Lemma 5.2 in \cite{MPS06} (see equation (52)),
\bn \label{rt37}
|G_{t}(w)-G_{t'}(w)|\leq C
t^{-1}|t-t'|(G_{t}(w)+G_{2t'}(w)).
\en
From (\ref{rt37}) and Lemma \ref{condmu}(a) we get
\bn \label{rt377}
&&\int_{\re}[G_{t-s}(x-z)-G_{t'-s}(x-z)]^2\mu(dz) \nonumber \\ &&\leq
C(t-s)^{-2}|t-t'|^2\int_{\re}(G_{t-s}(x-z)+G_{2(t'-s)}(x-z))^2\mu(dz) \nonumber \\ &&\leq
C(K,\eta')(t-s)^{-3+\eta'/2}|t-t'|^2.
\en
Repeat the same steps in (\ref{rt34}), this time with $\vartheta/2$ instead of $\vartheta$,
then, use again Lemma \ref{condmu}(a), (\ref{rt37}) and (\ref{rt377}) to get
\bn \label{rt38}
&&Q_{1,2}^{x,x,t,t'}  \nonumber \\ &&\leq
Ce^{{-\frac{\eps^{-2\dl_1}}{4}(1-\vartheta/2)}}\int_{0}^{t-\eps^2}
\bigg(\int_{\re}e^{\frac{2\gamma|z|}{1-\vartheta/2}}[G_{2(t-s)}(x-z)+G_{2(t'-s)}(y-z)]^2\mu(dz)\bigg)^{1-\vartheta/2}
\nonumber \\ && \quad \times \bigg(\int_{\re}[G_{t-s}(x-z)-G_{t'-s}(y-z)]^2\mu(dz)\bigg)^{\vartheta/2}ds \nonumber \\
&&\leq C(K,\vartheta,\eta',\gamma)e^{{-\frac{\eps^{-2\dl_1}}{4}(1-\vartheta/2)}}\int_{0}^{t}\bigg(\frac{1}{(t-s)^{(1-\eta'/2)(1-\vartheta/2)}} \bigg) \bigg(\frac{|t-t'|^{\vartheta}}{(t-s)^{(3-\eta'/2)\vartheta/2}}\bigg)ds \nonumber \\
&&\leq C(K,\vartheta,\eta',\gamma)e^{{-\frac{\eps^{-2\dl_1}}{4}(1-\vartheta/2)}}|t-t'|^{\vartheta},
\en
where we have used the fact that $(1-\eta'/2)(1-\vartheta/2)+(3-\eta'/2)\vartheta/2=1+\vartheta-\eta'/2<1$ in the last inequality. From (\ref{rt38}) we deduce (b) for $i=1$. \medskip \\
To prove (a) and (b) for $i=2$ and to prove (c), we notice that $s\in[t-\eps^2,t']$ and $|x-z|>2\eps^{1-\dl_1}$. Since $|x-y|\leq \eps$ we have $|y-z|\geq
||x-z|-|x-y||>2\eps^{1-\dl_1}-\eps>\eps^{1-\dl_1}$. Since $|t-t'|\leq \eps^2$ we have $t'-s=t'-t+t-s\leq \eps^2+t-s\leq 2\eps^2$. From these two inequalities, we get
\bn \label{cl1}
\bigg(\frac{(x-z)^2}{t-s} \bigg) \wedge  \bigg ( \frac{(y-z)^2}{t'-s} \bigg) \geq \frac{\eps^{-2\dl_1}}{2},
\en
and
\bn \label{cl3}
 \frac{(x-z)^2}{t'-s}\geq 2\eps^{-2\dl_1}.
\en
From (\ref{cl1}) we get inequality like (\ref{rt33}), then (a) and (b) follow for $i=2$ along the same lines as the case where $i=1$. Finally we use (\ref{cl3}) to get (\ref{rt332}) with $\theta =1-\vartheta$. We also apply (\ref{rt23}) to $Q_2^{x,t,t'}$ to get,
\bn \label{rt39}
Q_2^{x,t,t'}&=&\int_{t}^{t'}\int_{A_2^2} G^2_{t'-s}(x-z)u^{2\gamma}(s,z)\mu(dz)ds \nonumber \\ &\leq&K^{2\gamma}
\int_{t}^{t'}\int_{A_2^2} e^{2\gamma|z|}G^2_{t'-s}(x-z)\mu(dz)ds \nonumber \\
&\leq&C(K,\lam,\vartheta)e^{-\frac{\eps^{-2\dl_1}}{2}(1-\vartheta)} \int_{t}^{t'}\int_{\re}
e^{2\gamma|z|}G^{2}_{(t'-s)/\vartheta}(x-z)\mu(dz)ds \nonumber \\
&\leq&C(K,\vartheta,\eta',\gamma)e^{-\frac{\eps^{-2\dl_1}}{2}(1-\vartheta)} |t-t'|^{\eta'/2},
\en
where we used Lemma \ref{condmu}(a) in
the last inequality. Note that (\ref{rt39}) is (c) and this completes the proof.
\end{proof} \\\\
Let  $\vartheta'=\eta'/2-\dl'+\gamma\xi$.
Note that $\vartheta'<1$ (by(\ref{rt3})). Now use Lemma
\ref{lemma5.6-MPS}(a),(c) and Lemma \ref{lemma5.7-MPS}(a) with $\vartheta=\eta'/2-\dl'$ to get for $(t,x)\in Z_{K,N,\xi},|x-y|<\eps=2^{-N}$ and $N>N_1$,
\bn \label{pp1} Q^{x,y,t,t}&\leq&
Q^{x,y,t,t}_{1,1} + Q^{x,y,t,t}_{2,1}+Q^{x,y,t,t}_{1,2}+ Q^{x,y,t,t}_{2,2} \nonumber \\ &\leq &
C(\dl',K,\eta',\gamma)(\eps^{2(1-\dl_1)\xi\gamma}|x-y|^{2\vartheta} + (8+3K2^{N_{\xi}\xi})^{2\gamma}\eps^{-2\dl_1\xi\gamma}|x-y|^{2\vartheta'}, \nonumber \\
&&+e^{-\eps^{-2\dl_1}(1-\vartheta)/4}|x-y|^{2\vartheta}) \nonumber \\
&=&C(\dl',K,\eta',\gamma)|x-y|^{2(\eta'/2-\dl')}\big[\eps^{2(1-\dl_1)\xi\gamma} \nonumber \\
&&+(8+3K2^{N_{\xi}\xi})^{2\gamma}\eps^{-2\dl_1\xi\gamma}|x-y|^{2\gamma\eps} +e^{-\eps^{-2\dl_1}(1-\eta'/2+\dl')/4}\big]
\nonumber \\ &\leq&C(\dl',K,\eta',\gamma)|x-y|^{2(\eta'/2-\dl')}\big[2^{2N_{\xi}\xi\gamma}\eps^{2(1-\dl_1)\xi\gamma}
+e^{-\eps^{-2\dl_1}(2-\eta')/8}\big].
\en
Use Lemmas \ref{lemma5.6-MPS}(b),(d),(e),
\ref{lemma5.7-MPS}(b),(c) with the same $\vartheta, \vartheta'$ as in (\ref{pp1}) to get for $(t,x)\in Z_{K,N,\xi},|t-t'|<\eps^2$ and $N\geq N_1$,
\bn \label{pp2}
Q^{x,x,t,t'}+Q^{x,t,t'}&\leq& Q^{x,x,t,t'}_{2,1} + Q^{x,x,t,t'}_{1,1}+Q^{x,t,t'}_{1}+
Q_{1,2}^{x,x,t,t'}+Q_{2,2}^{x,x,t,t'}+Q^{x,t,t'}_{2}  \\
&\leq&C(\vartheta,K,\eta',\gamma)(\eps^{2(1-\dl_1)\xi\gamma}|t-t'|^{\vartheta} +(8+3K2^{N_{\xi}\xi})^{2\gamma}\eps^{-2\dl_1\xi\gamma}|t-t'|^{\vartheta'} \nonumber \\
&&+\eps^{2\gamma\xi(1-\dl_1)}|t-t'|^{\eta'/2} +e^{-\eps^{-2\dl_1}(1-\vartheta/2)/4}|t-t'|^{\vartheta}
+e^{-\eps^{-2\dl_1}(1-\vartheta)/2}|t-t'|^{\eta'/2}). \nonumber \en
From (\ref{pp2}) we get
\bn
 Q^{x,x,t,t'}+Q^{x,t,t'} &\leq&
C(\dl',K,\eta',\gamma)|t-t'|^{\vartheta}\big[\eps^{2(1-\dl_1)\xi\gamma}+(8+3K2^{N_{\xi}\xi})^{2\gamma}\eps^{-2\dl_1\xi\gamma}|t-t'|^{\vartheta'-\vartheta}\nonumber
\\ &&+\eps^{2\gamma\xi(1-\dl_1)}|t-t'|^{\eta'/2-\vartheta}+e^{-\eps^{-2\dl_1}(1-\vartheta/2)/4}|t-t'|^{\vartheta}+
e^{-\eps^{-2\dl_1}(1-\vartheta)/2}|t-t'|^{\eta'/2-\vartheta}\big]\nonumber \\
&\leq& C(\dl',K,\eta',\gamma)|t-t'|^{\eta'/2-\dl'}\big[\eps^{2(1-\dl_1)\xi\gamma}+(8+3K2^{N_{\xi}\xi})^{2\gamma}\eps^{-2\dl_1\xi\gamma}|t-t'|^{\gamma\xi}\nonumber
\\ &&+\eps^{2\gamma\xi(1-\dl_1)}|t-t'|^{\dl'}+e^{-\eps^{-2\dl_1}(1-\eta'/4+\dl/2)/4}+
e^{-\eps^{-2\dl_1}(1-\eta'/2+\dl')/2}|t-t'|^{\dl'}\big]\nonumber \\
&\leq&
C(\dl',K,\eta',\gamma)|t-t'|^{\eta'/2-\dl'}\big[\eps^{2(1-\dl_1)\xi\gamma}2^{2N_{\xi}\xi\gamma}+
e^{-\eps^{-2\dl_1}(4-\eta')/16}\big],
\en
where we have used the fact that $\dl'\in(0,\eta'/2)$ in the last inequality.
From (\ref{pp1}) and (\ref{pp2}) we conclude that $P_3=P_4=0$ in (\ref{rt6}) if
\bn \label{pp3}
C(\dl',K,\eta',\gamma)\big[\eps^{2(1-\dl_1)\xi\gamma}2^{2N_{\xi}\xi\gamma}+ e^{-\eps^{-2\dl_1}(2-\eta')/16}\big]&\leq&
\eps^{2p},
\en
The rest of the proof is similar to the proof of Theorem 4.1 in \cite{MPS06}. We find conditions on $N_{\xi}$ and $\dl_1$ so that (\ref{pp3}) is satisfied. Then we use the estimates in (\ref{rt67}) and the fact that $P_3=P_4=0$ to bound the probabilities in (\ref{rt4}). Finally a chaining argument is used to get the hypothesis of Theorem \ref{Theorem4.1-MPS}.
\qed

\paragraph{Proof of Theorem \ref{Theorem2.3MP-P0-ind}} The proof uses ideas
from the proof of Corollary 4.2 in \cite{MPS06}. From Theorem 2.4 in \cite{Zahle2004} we get that $u$ is uniformly H\"{o}lder
-$\rho$ continuous on compacts in $(0,\infty)\times \re$ for every $\rho\in(0,\eta/4)$. Define inductively $\xi_0=\eta/4$ and
$\xi_{n+1}=\big[\big(\xi_n\gamma+\eta/2 \big)\wedge 1 \big]\big(1-\frac{1}{n+3}\big)$ so that
\bn \label{xi-calc}
\xi_n\uparrow \frac{\eta}{2(1-\gamma)}\wedge1.
\en
Recall that $\eta'=\eta-\varpi$ for some arbitrarily small $\varpi\in(0,\eta)$. From the assumptions of Theorem \ref{Theorem2.3MP-P0-ind} we can chose $\varpi$ sufficiently small such that
\bn \label{xi-calc1}
\gamma>1-\eta'/2.
\en
From (\ref{xi-calc}) and  (\ref{xi-calc1}) we have
\bn \label{xi-calc2}
\xi_n\uparrow  1.
\en
Fix $n_0$ so that $\xi_{n_0}\geq \xi > \xi_{n_0-1}$. Apply Theorem \ref{Theorem4.1-MPS} inductively $n_0$ times
to get (\ref{hyp-Thm-4.1-MPS}) for $\xi_{n_0-1}$. (\ref{res-Thm-4.1-MPS}) follows with $\xi_1=\xi_{n_0}$.   \qed
\paragraph{Proof of Theorem \ref{thm-reg}:}
Let $(t_0,x_0)=(t_0,x_0)(\omega)\in S_0(\omega)$. From (\ref{Def-Z-K-N-Set}) follows that
\bn \label{reg11}
(t_0,x_0)\in Z(K,N)(\omega), \ \forall N\geq 0.
\en
From (\ref{reg11}) and Theorem \ref{Theorem2.3MP-P0-ind} follows that there exists $N_0(\xi,K,\omega)$ such that for all $N\geq N_0$ and $(t,x)\in [0,T_K]\times \re$ such that
$d((t,x),(t_0,x_0))\leq 2^{-N}$, we have
\bn \label{reg12}
|u(t,x)-u(x_0,t_0)|\leq 2^{-N\xi}.
\en
Let $(t',x')\in [0,T_K] \times \re$ such that $d((t',x'),(t_0,x_0))\leq 2^{-N_0}$. There exists $N'\geq N_0$ such that $2^{-N'-1}\leq d((t',x'),(t_0,x_0))\leq 2^{-N'}$. We get from (\ref{reg12}) that
\bn \label{reg13}
|u(t',x')-u(x_0,t_0)|\leq 2^{-\xi N'}\leq 2(d((t',x'),(t_0,x_0)))^{\xi},
\en
and we are done.  \qed


%
%
%

\section{Acknowledgments}
This paper was written during my Ph.D. studies under
the supervision of Professor L. Mytnik. I am grateful to him for his guidance
and numerous helpful conversations during the preparation of this work.

\bibliographystyle{plain}
\printindex

\def\cprime{$'$}

\end{document}